\documentclass[reqno]{amsart}
\newif\ifdraft
\usepackage[utf8]{inputenc}
\usepackage[T1]{fontenc}
\usepackage{amssymb, mathtools}
\usepackage{lmodern, microtype}
\usepackage{enumerate, amscd, color}
\usepackage{ifthen}
\usepackage{slashed}
\usepackage{MnSymbol}
\usepackage{tikz, tikz-cd}
\usetikzlibrary{decorations.pathreplacing}

\newcounter{mnotecount}[section]
\renewcommand{\themnotecount}{\thesection.\arabic{mnotecount}}
\ifdraft\newcommand{\mnote}[1]{\protect{\stepcounter{mnotecount}}${\raisebox{0.5\baselineskip}[0pt]{\makebox[0pt][c]{\tiny\em{\red{$\bullet$\themnotecount}}}}}$\marginpar{\raggedright\tiny\em $\!\!\!\!\!\!\,\bullet$\themnotecount: #1}\ignorespaces}\else\newcommand{\mnote}[1]{}\fi

\newcommand{\pct}{\mathrm{pct}}
\newcommand{\inv}{\mathrm{inv}}

\usepackage{xcolor}

\definecolor{darkgreen}{rgb}{0,0.6,0}
\definecolor{darkred}{rgb}{0.7,0,0}
\definecolor{darkblue}{rgb}{0,.1,.6}
\definecolor{darkgray}{rgb}{0.3,.3,.3}
 \usepackage[%
   colorlinks, 
  citecolor=darkgreen, linkcolor=darkblue,
   pdfpagelabels=true,
   unicode=true,
   pdfusetitle
 ]{hyperref}
\ifdraft{\usepackage{showkeys}

\usepackage[normalem]{ulem}
\renewcommand\sout{\bgroup\markoverwith
{\textcolor{red}{\rule[0.7ex]{3pt}{1.4pt}}}\ULon}
\else 
\long\def\sout#1{}
\fi
\newcommand\red[1]{{\color{red}{#1}}}

\newcommand\jonathan[1]{\ifdraft\textcolor{darkblue}{#1}\else#1\fi}

\def\si{{\sigma}}

\let\epsilon\varepsilon
\def\ep{{\varepsilon}}         
         
\def\Tau{\mathcal{T}}         

\let\theta\vartheta

\def\phi{{\varphi}}

\def\Sigmairr{\Sigma_{\mathrm{irr}}}
\newcommand\olM{\overline{M}}
\newcommand\olg{\overline{g}}
\newcommand\geucl{g_{\mathrm{eucl}}}
   
\let\na\nabla     

\let\ti\tilde
\let\witi\widetilde
\let\wihat\widehat
\newcommand{\lgen}{\mathopen{\prec}}
\newcommand{\rgen}{\mathclose{\succ}}

\newcommand\scal{\operatorname{scal}}

\newcommand\grad{\operatorname{grad}}

\newcommand\tr{\mathop{\mathrm{tr}}}
\newcommand\id{\mathop{\mathrm{id}}}
\newcommand\Cl{\mathop{\mathrm{Cl}}}
\newcommand\CCl{\mathop{\mathbb{C}\mathrm{l}}}
\newcommand\rank{\mathop{\mathrm{rank}}}
\newcommand\image{\mathop{\mathrm{im}}}

\DeclareMathOperator{\spann}{span}

\DeclareMathOperator{\KO}{KO}

\newcommand\<{\langle}
\renewcommand\>{\rangle} 
\newcommand{\del}{\partial}
\newcommand{\quotientslash}{\Big/}
\newcommand{\quotientbslash}{\big\backslash}
\newlength{\lowerquotdepth}
\setlength{\lowerquotdepth}{1.5pt}
\newlength{\smalllowerquotdepth}
\setlength{\smalllowerquotdepth}{.9pt}
\newcommand{\quotientspace}[2]{#1\lower\lowerquotdepth\hbox{$\quotientslash$}\lower2\lowerquotdepth\hbox{$#2$}}
\newcommand{\leftsmallquotientspace}[2]{\lower2\smalllowerquotdepth\hbox{$#1$}\lower\smalllowerquotdepth\hbox{$\quotientbslash$}#2}

\newcommand\ie{i.\thinspace e.,\ \ignorespaces}
\newcommand\eg{e.\thinspace g.,\ \ignorespaces}

\newcommand{\arxiv}[1]{Preprint \href{https://arxiv.org/abs/#1}{ArXiv:#1} \href{https://arxiv.org/pdf/#1.pdf}{(pdf)}} 

\makeatletter
\newcommand*\bigcdot{\mathpalette\bigcdot@{.7}}
\newcommand*\Bigcdot{\mathpalette\bigcdot@{.85}}
\newcommand*\bigcdot@[2]{\mathbin{\vcenter{\hbox{\scalebox{#2}{$\m@th#1\bullet$}}}}}
\makeatother

\newcommand{\Dirac}{\slashed{D}}
\newcommand{\olDirac}{\overline{\Dirac}}

\newcommand{\lto}{\longrightarrow}
\newcommand{\lmapsto}{\longmapsto}

\newcommand{\Met}{\mathcal{R}}
\newcommand{\Psc}{\mathcal{R}^>}
\newcommand{\Nnsc}{\mathcal{R}^{\geq}}
\newcommand{\Ini}{\mathcal{I}}
\newcommand{\Dec}{\mathcal{I}^{\geq}}
\newcommand{\Decstr}{\mathcal{I}^>}
\newcommand{\Susp}{S}
\newcommand{\Torus}[1][ ]{\ZZ^{#1}\backslash\RR^{#1}}
\newcommand{\pt}{\{\bullet\}}
\def\Mod{\mathop{\text{$\mathcal{M}$\kern-1.7pt\textit{od}}}\nolimits}

\DeclareMathOperator{\sign}{sign}
\DeclareMathOperator{\Fred}{Fred}
\DeclareMathOperator{\adiff}{\alpha-diff}
\DeclareMathOperator{\oladiff}{\ol\alpha-diff}
\DeclareMathOperator{\bdiff}{\beta-diff}
\DeclareMathOperator{\olbdiff}{\ol\beta-diff}
\DeclareMathOperator{\Ein}{Ein}
\setlength{\marginparwidth}{1.12in}

\newcommand\datver[1]{\def\datverp
{\par\boxed{\boxed{\text{Version: #1; Run: \today}}}}}\datver{0.1}

\setcounter{tocdepth}{2}

\newcommand{\definedas}{\coloneqq}

\newcommand{\dom}{\operatorname{dom}}

\newcommand{\CC}{\mathbb C}
\newcommand{\FF}{\mathbb F}
\newcommand{\HH}{\mathbb{H}}

\newcommand{\NN}{\mathbb N}

\newcommand{\RR}{\mathbb R}

\newcommand{\ZZ}{\mathbb Z}
\newcommand{\QQ}{\mathbb Q}
\newcommand{\Ztwo}{\mathbb{Z}/2\mathbb{Z}}

\newcommand{\ind}{\operatorname{ind}}

\newcommand{\End}{\operatorname{End}}
\newcommand{\GL}{{\operatorname{GL}}}
\newcommand{\SL}{{\operatorname{SL}}}
\newcommand{\SO}{{\operatorname{SO}}}
\newcommand{\SU}{{\operatorname{SU}}}
\newcommand{\Sp}{{\operatorname{Sp}}}

\let\O\relax
\newcommand{\O}{{\operatorname{O}}}
\newcommand{\Spin}{{\operatorname{Spin}}}
\newcommand{\Isom}{{\operatorname{Isom}}}
\newcommand{\Isomaff}{{\operatorname{Isom}\nolimits_{\mathrm{aff}}}}

\newcommand{\Aut}{\operatorname{Aut}}
\newcommand{\Diff}{\operatorname{Diff}}

\renewcommand{\div}{\operatorname{div}}
\let\ol\overline

\newcommand{\maF}{\mathcal F}
\newcommand{\maG}{\mathcal G}
\newcommand{\maH}{\mathcal H}

\newcommand{\maL}{\mathcal L}

\newcommand{\maP}{\mathcal P}

\newcommand{\maT}{\mathcal T}

\newcommand{\platz}{\,\cdot\,}
\newtheorem{theorem}{Theorem}[section]
\newtheorem{proposition}[theorem]{Proposition}
\newtheorem{corollary}[theorem]{Corollary}
\newtheorem{conjecture}[theorem]{Conjecture}

\newtheorem{lemma}[theorem]{Lemma}
\newtheorem{construction}[theorem]{Construction}
\newcommand{\Description}[1]{\noindent\textit{Description #1.}}

\theoremstyle{definition}
\newtheorem{definition}[theorem]{Definition}
\theoremstyle{remark}
\newtheorem{remark}[theorem]{Remark}
\newtheorem{remarks}[theorem]{Remarks}

\newtheorem{example}[theorem]{Example}
\newtheorem{examples}[theorem]{Examples}

\newtheorem*{summary*}{Summary}

\author[B. Ammann]{Bernd Ammann} \address{B. Ammann, Fakult\"at f\"ur
  Mathematik, Universit\"at Regensburg, 93040 Regensburg, Germany}
\email{bernd.ammann@mathematik.uni-regensburg.de}
\author[J. Glöckle]{Jonathan Glöckle} \address{J. Glöckle, Fakult\"at f\"ur
  Mathematik, Universit\"at Regensburg, 93040 Regensburg, Germany}
\email{jonathan.gloeckle@mathematik.uni-regensburg.de}
\thanks{The authors and this article were supported by the SFB 1085 (Higher Invariants) in Regensburg, Germany, funded by the  DFG (German Science Foundation).}
\date\today

\begin{document}

\title{Dominant energy condition and spinors on Lorentzian manifolds}

\begin{abstract}
Let $(\olM,\olg)$ be a time- and space-oriented Lorentzian spin manifold, and  let $M$ be a compact spacelike 
hypersurface of $\olM$ with induced Riemannian metric $g$ and second fundamental form $K$. If $(\olM,\olg)$ satisfies the dominant energy condition
in a strict sense, then the Dirac--Witten operator of $M\subseteq \olM$ is an invertible, self-adjoint Fredholm operator. This allows us to use index theoretical methods in order to detect 
non-trivial homotopy groups in the space of initial on $M$ satisfying the dominant energy condition in a strict sense. The central tool will be a Lorentzian analogue of 
Hitchin's $\alpha$-invariant. In case that the dominant energy condition only holds in a weak sense, the Dirac--Witten operator may be non-invertible, and we will study the kernel of this operator in this case.
We will show that the kernel may only be non-trivial if $\pi_1(M)$ is virtually solvable of derived length at most~$2$. This allows to extend the index theoretical
methods to spaces of initial data, satisfying the dominant energy condition in the weak sense.
We will show further  that a spinor $\phi$ is in the kernel of the Dirac--Witten operator on $(M,g,K)$ if and only if $(M,g,K,\phi)$ admits an extension to a Lorentzian manifold $(\ol N,\ol h)$ with parallel spinor $\bar\phi$ such that $M$ is a Cauchy hypersurface of $(\ol N,\ol h)$, such that $g$ and $K$ are the induced metric and second fundamental form of $M$, respectively, and $\phi$ is the restriction of $\bar\phi$ to $M$.
\end{abstract}

\maketitle
\tableofcontents
\setcounter{page}{1}

%%%%%%%%%%%%%%%%%%%%%%%%%%%%%%%%%%%%%
\section{Introduction}
%%%%%%%%%%%%%%%%%%%%%%%%%%%%%%%%%%%%%

\subsection{Motivation for the dominant energy condition}
The research about the space of metrics with positive (resp.\ non-negative) scalar curvature experienced an overwhelming success during the last decades. In particular, index theoretical methods turned out to be very powerful. This established interesting bridges between topology and geometry, and these bridges provide a strong motivation to study these spaces.

Further motivation to study metrics of positive (resp.\ non-negative) scalar curvature comes from general relativity. In many cases the interest in such metrics goes back to the dominant energy condition. Amazingly to us it seems that
the dominant energy condition is less understood than positive scalar curvature. In this article we aim to establish several results in the dominant energy context which were known so far only for positive scalar curvature metrics.

In fact, the dominant energy condition (DEC) can be seen from two points of view. First, we may see it as a condition on the Ricci tensor of a Lorentzian manifold $(\ol M,\ol g)$, see Definition~\ref{def.dec}.
The physical interpretation of the dominant energy condition is that the rest mass density is non-negative when measured at any point of space-time and in any (causal) direction. 

Second, we can choose a spacelike hypersurface $M$ in $\ol M$. Let $g$ and $K$ be the induced metric and second fundamental form. The pair $(g,K)$ is called an initial data set.  Equation~\eqref{eq:CE} then determines
the energy density $\rho\in C^\infty(M)$ and the momentum density $j\in \Omega^1(M)$. The dominant energy condition from Definition~\ref{def.dec} implies $\rho\geq \|j\|$ which is by definition the initial data version of the dominant energy condition, see Definition~\ref{def.dec.ids}.

In the special case that  $M$ is totally geodesic (\ie the second fundamental form~$K$ vanishes), the dominant energy condition holds
if and only if the scalar curvature of $g$ is non-negative. Thus from this point of view non-negative (resp.\ positive) scalar curvature is a special case of the (strict) dominant energy condition for initial data sets $(g,K)$.

Initial data sets satisfying the dominant energy condition allow many interesting mathematical conclusions. We will mention two of them: the positive mass theorem and black hole boundaries.

The spacetime version of the positive mass theorem -- also called positive energy theorem -- states that the ADM mass $E$ and the ADM momentum vector $P$ of an asymptotically flat initial data set with $\rho\geq \|j\|$ satisfy $E\geq \|P\|$. The theorem was proven for spin manifolds by Witten \cite{witten:81} by defining and studying a Dirac type operator, the so-called \emph{Dirac--Witten operator of $(M,g,K)$}, which will also play a crucial role in our article.
%We should add that Witten's techniques were extended and spelled out in more details in \cite{parker.taubes:82}, and later on the subject  was elaborated in more details \eg in \cite{hijazi.zhang:03}. 

The second type of result we want to mention shows another interesting bridge to positive scalar curvature. We assume, that our asymptotically flat initial data set $(g,K)$ is defined on a manifold with compact boundary $S$, and we assume that this boundary is a marginally outer trapper surface with respect to $(g,K)$.
For initial data sets $(g,K)$ satisfying the dominant energy condition, it was shown by Galloway and coauthors
\cite{cai.galloway:MR1846368,% Bereits 2001, sehr früh: hep-th/0102149 
galloway.schoen:MR2238889,% gr-qc/0509107  A generalization of {H}awking's black hole topology theorem to  higher dimensions
galloway:MR2411473,% gr-qc/0608118 Commun Anal Geom 16(1):217–229, 2008
galloway:MR3768955, % Arxiv 1712.02247, veinfachte Version von Commun Anal Geom 16(1):217–229, 2008
andersson.dahl.galloway.pollack:MR3878142%Arxiv 1508.01896
}
that the induced Riemannian metric on~$S$ is conformal to a metric of non-negative scalar curvature, and that this metric can be deformed to a metric of positive scalar curvature.

%% Another motivation comes from the positive mass (resp.\ positive energy) theorem. In both applications the dominant energy condition (DEC) plays a crucial role.
%% It also plays a prominent role in the positive energy theorem which states that for Cauchy surfaces in Lorentzian manifolds that are  asymptotic to a flat hypersurface in Minkowski space and that satisfying DEC, the ADM mass
%% $E\in\RR$ and the ADM moment $P\in \RR^n$ satisfy $E\geq |P|$. Both quantities $E$ and $P$ were introduced by  Arnowitt, Deser and Misner, see \cite{arnowitt.deser.misner:60a} and some follow-up articles.%arnowitt.deser.misner:60b,arnowitt.deser.misner:61,

%% The dominant energy condition is tightly related to non-negative scalar curvature in another way. In the case that $M$ is a totally geodesic spacelike hypersurface in the Lorentzian manifold $\ol M$, the dominant energy condition on $\olM$ implies non-negative scalar curvature on $M$, and in some sense the converse is true as well, see the discussion at the end of Subsection~\ref{subsec.dominant.energy}.

%% Thus, roughly speaking, one could consider DEC as the approriate space-time version of non-negative scalar curvature.

\subsection{Results of this article}
In this article, we contribute to the idea to describe a Lorentzian manifold from the perspective of initial data sets. Instead of considering the Lorentzian manifold~$\ol M$ itself, we consider the associated \emph{initial data set} $(g,K)$
where $g$ is the induced Riemannian metric and $K$ the second fundamental form on some spacelike hypersurface $M$.
%The DEC can be formulated in terms of $(g,K)$, see Definition~\ref{def.dec.ids}.
Our main interest is the topology of the space $\Dec(M)$ of initial data sets satisfying DEC\footnote{see Definition~\ref{def.dec.ids}} on a closed spin manifold $M$. In particular, we will determine non-trivial homotopy groups in $\Dec(M)$, see \eg Corollary~\ref{cor:Wdec}.
We will use index theory and a Lorentzian version of the index difference
in order to detect these homotopy groups. To construct non-trivial elements
in homotopy groups $\pi_k(\Dec(M))$ we will combine known methods for constructing non-trivial homotopy groups in the space $\Psc(M)$ of positive scalar curvature
metrics with a suspension construction explained in Subsection~\ref{subsec.ini.data.pairs}.

In order to adapt the methods from the space of positive scalar curvature metrics to the space $\Dec(M)$ several steps have to be taken. First, recall from above that DEC is the space-time version of non-negative scalar curvature. Thus the first task is to find non-trivial homotopy groups in the space $\Nnsc(M)$ of metrics with non-negative scalar curvature on $M$. Here important progress was achieved by Schick and Wraith \cite{schick.wraith.2016}.

Second, one needs an initial data version of the Dirac operator. As already mentioned above, such an operator, the \emph{Dirac--Witten operator of $(M,g,K)$}, was introduced by Witten \cite{witten:81,parker.taubes:82} in order to prove the spacetime positive mass theorem on spin manifolds. This operator was elaborated in more details \eg in \cite{hijazi.zhang:03}, and it is a modification of the Dirac operator on $(M,g)$ by a term depending on $K$.
All authors named above considered the Dirac--Witten operator acting on complex spinors. In index theory typically more refined information can be obtained by considering Clifford-graded versions of the Dirac operator, see \eg \cite{lawson.michelsohn:89}. A Clifford-graded version of the Dirac--Witten operator was introduced by the second author of this article
in~\cite{gloeckle:p2019}. In particular, this Clifford-graded version
allowed him to define a Lorentzian index
difference, see Subsection~\ref{subsec.index.diff.idp},
which was a key tool to detect non-trivial homotopy groups in
the space $\Decstr(M)$  of initial data sets strictly satisfying DEC. 
Here  $\Decstr(M)$ is the space-time companion of positive scalar curvature metrics, see Definition~\ref{def.dec.ids}. Section~\ref{sec.homotopy.groups}
of our present article is to a large extent a summary of~\cite{gloeckle:p2019}.

The goal of the present article is to detect non-trivial homotopy groups in~$\Dec(M)$. This requires results analogous to the passage from positive scalar curvature to non-negative scalar curvature. On a closed spin manifold of non-negative scalar curvature every harmonic spinor is parallel, and this implies special holonomy, a subject which is well-studied. In our space-time setting, we will
%There are further analogies of our work to the totally geodesic special case, \ie the $K=0$ case. In particular,
have to study the kernel of the Dirac--Witten operator on initial data sets $(g,K)\in \Dec(M)$ for some closed spin manifold $M$. This will lead us
to the definition of \emph{initial data triples} $(g,K,\phi)$ where $\phi$ is an imaginary $W$-Killing spinor, see Definition~\ref{def.initial.data.triples} and Equation~\eqref{gen.imag.killing}.
Initial data triples play the space-time role of a Ricci-flat metric with a nowhere vanishing parallel spinor. The main goal of Section~\ref{sec.kernel.DWo} is to characterize -- as far as possible -- manifolds carrying initial data triples. In particular, we will show that the existence of an initial data triple on a closed spin manifold $M$ implies that
$\pi_1(M)$ is virtually solvable of derived length at most~$2$, see in particular Corollary~\ref{cor:obstr.lightlike.id-triple}. We immediately obtain non-trivial homotopy groups $\pi_k(\Dec(M))$, provided that $\pi_1(M)$ is not of that kind, see Corollary~\ref{cor:Wdec}.

Our goal for theses investigations can be described by the following diagram, which commutes up to homotopy, whenever all maps exist.
Here $\adiff$ is the classical index difference, $\oladiff$ is ``essentially'' the Lorentzian index difference from\cite{gloeckle:p2019}, see  Subsection~\ref{subsec.index.diff.idp}.

\begin{center}
  \def\yalpha{-.9}
  \def\ymid{-1.5}
  \def\ybeta{-1.9}
  \def\ybelow{-3.0}
  \def\xlinks{0}
  \def\xriem{1}
  \def\xmitte{3}
  \def\xlor{5.1}
  \def\xrechts{6}
\begin{tikzpicture}[scale={.9}]
    \node (riemp) at (\xlinks,0) {$\Psc(M)$};
    \node (riemnn) at (\xlinks,\ybelow) {$\Nnsc(M)$};
    \node (idsp) at (\xrechts,0) {$\Omega\bigl(\Decstr(M)\bigr)$};
    \node (idsnn) at (\xrechts,\ybelow) {$\Omega\bigl(\Dec(M)\bigr)$};
    \node (KO) at (\xmitte,\ymid) {$\Omega^{\infty+n+1}\KO$};
    \draw[->] (riemp) to node[above] {} (idsp);
    \draw[->] (riemp) to node {} (riemnn);
    \draw[->] (riemnn) to node[above] {} (idsnn);
    \draw[->] (idsp) to node {} (idsnn);
%    \draw[->] (riemp) to node[below,left] {$\alpha_{\mathrm{Riem}}$} (KO);
    \draw[->] (riemp) to node[below,left] {} (KO);
    \node (riemalphapfeil) at (\xriem,\yalpha) {$\adiff$};
%    \draw[->] (idsp) to node[below,right] {$\;\alpha_{\mathrm{Lor}}$}  (KO);
    \draw[->] (idsp) to node[below,right] {}  (KO);
    \node (loralphapfeil) at (\xlor,\yalpha) {$\oladiff$};
%    \draw[->,dashed] (riemnn) to node[above,left] {$\beta_{\mathrm{Riem}}$\;\;} (KO);
    \draw[->,dashed] (riemnn) to node[above,left] {\;\;} (KO);
    \node (riembetapfeil) at (\xriem,\ybeta) {$\bdiff$};
%    \draw[->,dashed,red] (idsnn) to node[above,right] {\;$\beta_{\mathrm{Lor}}$}  (KO);
    \draw[->,dashed,red] (idsnn) to node[above,right] {}  (KO);
   \node (lorbetapfeil) at (\xlor,\ybeta) {\red{$\olbdiff$}};
  \end{tikzpicture}
\end{center}
The dotted map $\bdiff$ is due to Schick and Wraith \cite{schick.wraith.2016}, its dottedness shall indicate here, that -- in general -- it is not defined
on all of $\Nnsc(M)$. The domain $\dom(\bdiff)$ includes the union of all path components in $\Nnsc(M)$ that contain metrics in $\Psc(M)$. All such metrics have an invertible Dirac operator. If $M$ does not admit a Ricci-flat metric with a parallel spinor, see \eg Examples~\ref{example.obstr.structured}, then we can achieve $\dom(\bdiff)=\Nnsc(M)$.

In the present article we discuss the existence of the map $\olbdiff$.
In the special case that initial data triples are topologically obstructed,
see Examples~\ref{examples.all.obstructed} and Corollary~\ref{cor:obstr.lightlike.id-triple} all pairs $(g,K)\in \Dec(M)$ have an invertible Dirac--Witten operator, and we obtain a map $\olbdiff$ defined on all of $\Omega\bigl(\Dec(M)\bigr)$. This includes all closed manifolds $M$ for which $\pi_1(M)$ is not virtually solvable of derived length at most~$2$. We consider it as plausible that this result can be extended to all manifolds whose fundamental group is not virtually abelian, see Remark~\ref{rem.initial.data.triples.with.DEC}. In other cases, \eg simply connected Calabi--Yau manifolds of dimension $8k+6$, we have $\Nnsc(M)\subsetneq\dom(\bdiff)$, so one should \emph{not} expect $\dom(\olbdiff)=\Omega\bigl(\Dec(M)\bigr)$ in general.
%arrows denote maps that they are only defined on subsets: this is for $\bdiff$ the union of all path components in $\Nnsc(M)$ that contain metrics in $\Psc(M)$, and some analogon for $\olbdiff$.

Initial data triples are also tightly connected to a Cauchy problem for parallel spinors on Lorentzian manifolds, recently studied by Baum, Leistner and Lischewski \cite{baum.leistner.lischewski:16,baum.leistner.lecture.notes:HH,leistner.lischewski:15,leistner.lischewski:19}, see also \cite{seipel:19}.
Obviously the work by Baum, Leistner, and Lischewski is tightly connected to Leistner's work on holonomy groups of Lorentzian manifolds~\cite{leistner.jdg2007}. If $(g,K,\phi)$ is an initial data triple, then we will show in Theorem~\ref{theorem.kernel-gives-initial-data} that  $(M,g,K,\phi)$ can be extended to a Lorentzian manifold $(\ol M,\ol g)$ with a parallel spinor $\bar \phi$ such that~$M$ is a spacelike hypersurface in $\ol M$ with induced $g$ and $K$, such that
$\bar\phi$ restricts to~$\phi$ on $M$.

Let us add two comments before we sketch the structure of the article. A particular case is $\pi_0(\Decstr(M))$, which is the borderline case, when our methods are still applicable. In particular we see that $\Decstr(M)$ is not connected in case $M$ has an index theoretic obstruction in $\KO^{-\dim M}(\pt)$ to a metric with positive scalar curvature. This can be interpretated as an index theoretical obstruction for some universe to evolve from a big bang to a big crunch. This is in analogy to similar results, derived from Gerhardt's theorem using a minimal spacelike hypersurface.
We explain this in Subsection~\ref{subsec.app.relativity}.

The other comment is that at many places we will refer to a method from \cite{ammann.kroencke.mueller:p19} to construct lightlike initial data triples.
There seem to be interesting relations to the article \cite{eichmair.galloway.mendes:2009.09527}.

\subsection{Structure of the article}
The article is structured as follows. In Section~\ref{sec.notation} we recall some basic notation, we fix some conventions and we introduce DEC and its strict version. We have said above that we need a Clifford-graded version of spinors on Lorentzian manifolds, and it remains interesting to see, how this new kind of spinors relates to complex spinors; this comparison is worked out in Section~\ref{sec.spinors.semi-riemann} on arbitrary semi-Riemannian manifolds. Restricting parallel spinors to hypersurfaces or extending them from hypersurface to manifolds is omnipresent in this article, in particular for spacelike hypersurfaces in Lorentzian manifolds; necessary tools for this setting are provided by Section~\ref{sec.parallel-spinors.hypersurfaces}. Sections~\ref{sec.kernel.DWo} and \ref{sec.homotopy.groups} are the mathematical core of the article. In Section~\ref{sec.kernel.DWo} we prove that harmonic Dirac--Witten spinors on closed manifolds with $(g,K)\in \Dec(M)$ provide initial data triples; then we study the underlying manifolds for initial data triples and we show that their fundamental groups are always virtually solvable.
Section~\ref{sec.homotopy.groups} is in large parts a summary of \cite{gloeckle:p2019}, however it includes the passage from $\Decstr(M)$ to $\Dec(M)$ which essentially relies on the results obtained in Section~\ref{sec.kernel.DWo}. At the end of Section~\ref{sec.homotopy.groups} we discuss the applications towards evolutions from big bang to big crunch, and compare this approach to Gerhardt's theorem.

\subsection*{Acknowledgements.}
We thank Greg Galloway for related discussions, during which he pointed us to Gerhardt's theorem, presented at the end of Subsection~\ref{subsec.app.relativity}.
Also thanks to Robert Bryant for drawing our attention to the nice argument, that~$S^n$ does not admit a metric with a parallel spinor for $n\geq 2$. We also thank the topology group in Regensburg for discussions about group (co)homology and the Borel conjecture which play a role in Subsection~\ref{subsec.ll-gen-ini-d-t-noncpct}. We thank Mattias Dahl for discussions about the connections of Definitions~\ref{def.dec} and~\ref{def.dec.ids}.

%%%%%%%%%%%%%%%%%%%%%%%%%%%%%%%%%%
\section{Notation and preliminaries}\label{sec.notation}
%%%%%%%%%%%%%%%%%%%%%%%%%%%%%%%%%%
In this subsection we fix some notation and recall some known facts.
%---------------------------
\subsection{Spacelike hypersurfaces in Lorentzian manifolds}
%--------------------------

We assume that $M$ is an $n$-dimensional space-like hypersurface in a time- and space-oriented Lorentzian manifold $(\olM,\olg)$ of dimension $n+1$. Let $g$ be the induced Riemannian metric on~$M$. Let $e_0$ be the future-oriented unit normal field of $M$ in~$\olM$. We define the Weingarten map $W:= -\ol\nabla e_0$. The associated second fundamental form will be defined as
$K=-W^\flat$, \ie
  $$\forall X,Y\in T_pM:\; K(X,Y)=-g(W(X),Y).$$
This implies $\ol\nabla_Y X= K(X,Y)e_0 + \nabla_Y X$ for $X \in TM$. Then $K=\frac12 (\maL_{e_0}\olg)|_{TM\otimes TM}$.

%Let $X\cdot \psi$ be the Clifford product of $X\in T_p\olM$ with $\psi\in\Sigma \olM$.
%On $M$, there is a second kind of Clifford multiplication  $X\bigcdot \psi$,
%related to the first one by the formula\footnote{Check which multiplcation satisfies which axioms}
%$$X\bigcdot \phi=ie_0\cdot X\cdot \phi.$$

We denote the standard basis of $\RR^{n,1}$ by $E_0,E_1,\ldots,E_n$ and the standard basis of $\RR^n$ by $E_1,\ldots,E_n$. The basis $(e_i)$ with $i\in \{0,\ldots,n\}$ resp.\  $i\in \{1,\ldots,n\}$ denotes a generalized orthonormal frame of $\olM$ resp.\ $M$, defined either in a point $p\in \olM$ resp.\ $p\in M$, locally on an open subset or globally on all of $\olM$ resp.\ $M$.

If $g$ is a Riemannian metric on a manifold $M$  and
 $K\in \Gamma(T^*M\otimes T^*M)$ a symmetric $2$-tensor, then we will say that
$(M,g,K)$ is a \emph{geometric hypersurface} in the time-oriented Lorentzian manifold $(\overline M,\ol g)$, if $M$ is a spacelike hypersurface and if, 
$g$ is the induced metric on $M$ and $K$ is the second fundamental form of $M$ in $\ol M$.

%---------------------------
\subsection{Dominant energy condition}\label{subsec.dominant.energy}
%---------------------------

For a Lorentzian manifold $(\ol M,\ol g)$ we define the Einstein curvature endomorphism $\Ein\in\End(T\ol M)$ as 
  $$\Ein(V) \definedas \mathrm{Ric}(V)- \frac12 \scal V, \quad\forall V\in T\ol M.$$

\begin{definition}\label{def.dec}
  We say that a Lorentzian manifold $(\ol M,\ol g)$ satisfies the dominant energy condition (DEC), if for all $x\in \ol M$ and all causal vectors $V,W\in T_x\ol M$ with the same time-orientation we have
    $$\olg(\Ein(V),W)\geq 0.$$
\end{definition}
Note that we did not require here, that $(\ol M,\ol g)$ carries some globally defined time- or space-orientation, although from now on we will assume that both types of orientations exist and we assume that every Lorentzian manifold comes with a fixed choice of such orientations.

Let again $M$ be a spacelike hypersurface of $\ol M$ with future-oriented
unit normal field $e_0$, and let $g$ and $K$ be defined as above. We then define the associated energy density $\rho=\olg(\Ein(e_0),e_0)\in C^\infty(M)$ and the momentum density $j=\olg(\Ein(e_0),\platz)\in \Omega^1(M)$. The equations by Gauß and Codazzi  (cf.\ \cite{BI}) imply
\begin{equation} \label{eq:CE}
	\begin{aligned}
		2\rho &= \scal^{g} + (\tr K)^2 - \|K\|^2 \\
		j &= \div K - \mathrm{d} \tr K.
	\end{aligned} 
\end{equation}

Now, for a moment we consider $M$ as an abstract manifold, not as  a submanifold of $\ol M$.
\begin{definition}\label{def.dec.ids}
An \emph{initial data set} on a manifold $M$ is a pair $(g,K)$ where $g$ is a Riemannian metric on $M$ and where $K$ is a symmetric smooth section of $T^*M\otimes T^*M$. Let $\Ini(M)$ be 
the set of all such pairs and we equip $\Ini(M)$ with the $C^\infty$-topology.
\footnote{For non-compact manifolds there are different versions of $C^\infty$-topology, \eg the strong and the weak version in the sense of \cite[Chap.~2, Sec. 1]{hirsch:76}, but we will only consider the compact case, when they all coincide.}

For any $(g,K)\in \Ini(M)$ we define $\rho$ and $j$ by equation \eqref{eq:CE}.

The subspace of those pairs $(g, K)$ satisfying the inequality $\rho \geq \|j\|$ everywhere will be denoted by $\Dec(M)$. We then say that $(g,K)$ satisfies
DEC. The subset of $(g,K)\in\Dec(M)$ which statisfies $\rho > \|j\|$, \ie the associated strict equation, is denoted by  $\Decstr(M)$, and we say that its elements \emph{strictly satisfy the dominant energy condition}.
\end{definition}

Let us explain the connection between Definitions~\ref{def.dec} and \ref{def.dec.ids}.
Let $M$ be a spacelike hypersurface of time-oriented Lorentzian manifold
$(\ol M,\ol g)$, and let $e_0$, $g$, $K$, $\rho$, and $j$ be defined as above by the embedding $M\subseteq \ol M$.
If $(\ol M,\ol g)$ satisfies DEC, then $(g,K)\in \Dec(M)$.

We conjecture -- but unfortunately we only have proofs in special cases -- that in some sense, a converse of this statement is true:
if $M$ is given, and $(g,K)\in \Dec(M)$, then there exist a
time-oriented Lorentzian manifold $(\ol M,\ol g)$ satisfying DEC 
and an embedding $M\hookrightarrow \ol M$ that induce the given $(g,K)$.

In the totally geodesic case, \ie $K=0$, DEC reduces to non-negative scalar curvature: $(g,0)\in \Dec(M)$ if and only if $\scal^g\geq 0$. And  $(g,0)\in \Decstr(M)$ if and only if $\scal^g > 0$.

%---------------------------
\subsection{Riemannian manifolds with parallel spinors and structured Ricci-flat metrics}\label{subsec.prelim.struc.ricciflat}
%---------------------------
Riemannian manifolds admitting a nowhere vanishing parallel spinor are Ricci-flat. % see \eg \cite{cahen.gutt.lemaire.spindel:86}
We say that a Ricci-flat Riemannian metric $g$ on $M$ is \emph{structured} if
the universal covering $\witi M$ of $M$ is spin and if -- with respect to the pull-back metric $\tilde g$ -- there is a nowhere vanishing parallel spinor on $\witi M$.
There are complete Ricci-flat Riemannian manifolds that are not structured (\eg the Riemannian Schwarzschild metric), but no closed Ricci-flat Riemannian manifold is known so far that is not structured. If $(M,g)$ is a closed structured Ricci-flat manifold, then a finite covering of $M$ already carries a parallel spinor.

Obviously closed structured Ricci-flat manifolds also satisfy the structure theorems for closed Ricci-flat manifolds by Cheeger--Gromoll \cite{cheeger.gromoll:71} and by Fischer--Wolf \cite{fischer.wolf:75}, see  \cite[Chap.~V, Theorem~3.11]{sakai:96} for a textbook style presentation. However, they are in fact much better understood than Ricci-flat metrics in general. As shown in \cite{wang_m:91} and \cite{dai.wang.wei:05}
structured Ricci-flat metrics on closed manifolds are stable. The parallel spinor implies that the holonomy group of $(M,g)$ is special, more precisely the restricted holonomy group is a product $G_1\times\cdots\times G_\ell$, acting diagonally on some orthogonal decomposition $T_xM=V_1\oplus\cdots\oplus V_\ell$ where
each $G_i$ is either the trivial group or one of groups $\SU(k_i)$ (if $2k_i=\dim V_i$), $\Sp(k_i)$ (if $4k_i=\dim V_i$), $G_2$ (if $7=\dim V_i$) and $\Spin(7)$ (if $8=\dim V_i$), see e.g.\ \cite{joyce:book}. The resulting product decomposition is also stable (in the sense of moduli spaces) under deformations \cite{kroencke:15}. The factors have smooth
(pre-)moduli space, due to Tian--Todorov and Joyce, and thus in combination with our knowledge about the deformations of products, we see that the (pre-)moduli of such metrics is smooth \cite{ammann.kroencke.weiss.witt:19}.

%---------------------------
\subsection{General notations and conventions}
%---------------------------
  
In this article all Hermitian scalar products in this article are complex
linear in the first entry and complex anti-linear in the second one.

Let $\Gamma_0$ be a discrete group. The \emph{derived series} is inductively defined as the commutator group  $\Gamma_{k+1}\definedas[\Gamma_k,\Gamma_k]$. The \emph{derived length} of $\Gamma$ is $\inf\{k\in \NN_0\mid \Gamma_k=1\}$. By definition, $\Gamma_0$ is solvable, if the derived length is finite. A group is called \emph{virtually solvable of derived length (at most) $k$} if it contains a subgroup $\Gamma_0$ of finite index, such that $\Gamma_0$ solvable of derived length (at most) $k$.

%%%%%%%%%%%%%%%%%%%%%%%%%%%%%%%%%%
\section{Spinors on semi-Riemannian manifolds}\label{sec.spinors.semi-riemann}
%%%%%%%%%%%%%%%%%%%%%%%%%%%%%%%%%%
\label{sec:Spinors}
In this section we recall and slightly extend known definitions and theorems
about spinors on (semi-)Riemannian manifolds. An important reference for this subject is the book \cite{baum:81} where complex spinors are considered.
However for our index theoretical considerations we will also need a $\Cl_{n,k}$-linear real version of spinors on semi-Riemannian manifolds,
developed in \cite{gloeckle:p2019} in the special case $k=1$.
These two approaches shall be compared.

%------------------------------------------------
\subsection{Spin structures and spin diffeomorphisms}
%------------------------------------------------

In this subsection we briefly recall how to construct spinor bundles on space- and time-oriented semi-Riemannian manifolds. 

At first, we assume that $N$ is an oriented manifold of dimension $m=n+k$. 
We denote by $P_{\GL^+}N$ the $\GL^+(m, \RR)$-principal bundle of positively oriented frames.
A topological spin structure is a reduction $P_{\widetilde{\GL^+}}N \to P_{\GL^+}N$ along the two-fold covering $\widetilde{\GL^+}(m, \RR) \to \GL^+(m, \RR)$.
In our article a \emph{spin manifold} always denotes an oriented manifold with a \emph{fixed} choice of a topological spin structure. If~$N_1$ and~$N_2$ are two spin manifolds of the same dimension $m$,
then a spin diffeomorphism consists of an orientation preserving diffeomorphism $f\colon N_1\to N_2$
together with a \emph{fixed} choice of a lift $S_f\colon P_{\widetilde{\GL^+}}N_1\to  P_{\widetilde{\GL^+}}N_2$ of $df\colon P_{\GL^+}N_1\to P_{\GL^+}N_2$,
$(v_1,\ldots,v_m)\mapsto (df(v_1),\ldots,df(v_m))$. In particular if $f:N\to N$ is a spin diffeomorphism, and if $M_f$ is the closed manifold obtained
from $N\times [0,1]$ by gluing $N\times \{0\}$ with $N\times \{1\}$ via $f$, then the lift $S_f$ defines a spin structure on $M_f$.

Now, assume that, additionally, $N$ carries a semi-Riemannian metric $h$ of signature $(n,k)$, $n+k=m$, and that space- and time-orientations
are fixed, that are compatible with the given orientation. 
The $\SO_0(n,k)$-principal bundle $P_{\SO_0}N$ is defined as the subbundle of $P_{\GL^+}N$ consisting of those generalized orthonormal frames $(e_1, \ldots, e_{n+k})$ for which $(e_1, \ldots, e_n)$ is positively oriented spacelike and $(e_{n+1}, \ldots, e_{n+k})$ is positively oriented timelike. Thereby, $\SO_0(n,k)$ is the identity component of $SO(n,k)$.
The topological spin structure defines a reduction $P_{\Spin_0}N \to P_{\SO_0}N$ along the two-fold covering $\Spin_0(n,k) \to \SO_0(n,k)$, by pull-back.

%------------------------------------------------
\subsection{Spinor bundles}
%------------------------------------------------
On a Riemannian spin manifold, the spinor bundles on $N$ are obtained by associating a (left) $\Cl_{n,k}$-module $\Sigma$ along the representation $\Spin_0(n,k) \hookrightarrow \Cl_{n,k} \to \End(\Sigma)$.
Then the spinor bundle $\Sigma N = P_{\Spin_0}N \times_{\Spin_0} \Sigma$ will be a (left) $\Cl(TN)$-module. The Levi-Civita connection induces a connection $\nabla$ on $\Sigma N$ with respect to which the $\Cl(TN)$-action is parallel.
Usually, the $\Cl_{n,k}$-modules $\Sigma$ we consider come with additional structures such as the structure of a complex vector space, a scalar product, a $\Ztwo$-grading or a right $\Cl_{n,k}$-action. In this case, $\Sigma N$ also carries such a structure and it will be $\nabla$-parallel. If the structure is compatible with the (left) $\Cl_{n,k}$-multiplication on $\Sigma$, the induced structure respects the multiplication by $\Cl(TN)$.

There are basically two such bundles we are interested in, which will be discussed in more detail in the next subsection. The first one arises when we take an irreducible (complex) representation $\Sigma_{n,k}$ of $\Cl_{n,k} \otimes_\RR \CC$. The corresponding spinor bundle $\Sigmairr N$ is called \emph{irreducible spinor bundle}. It is a complex vector bundle and carries a compatible non-degenerate Hermitian form. This is the spinor bundle that is most commonly used in semi-Riemannian geometry and mathematical physics, in particular in most of the literature on parallel spinors.

For the second one, the $\Cl_{n,k}$-module is taken to be $\Cl_{n,k}$ with left multiplication. The resulting spinor $\Sigma_{\Cl}N$ bundle is called \emph{$\Cl_{n,k}$-linear spinor bundle}. It has compatible non-degenerate bilinear form, $\Ztwo$-grading and right $\Cl_{n,k}$-action. It is often used in index theory. 

%------------------------------------------------
\subsection{Comparing different version of spinors}
%------------------------------------------------
As we intend to connect index theory with the study of parallel spinors, we need to compare the irreducible and the $\Cl_{n,k}$-linear spinor bundle.
We start off with the following classical results on the structure and representation theory of Clifford algebras.

\begin{proposition}[{cf. \cite[Thms.~5.7, 5.9]{lawson.michelsohn:89}}]
If $n+k$ is even, then there exists a unique irreducible (complex) representation $\rho \colon \CCl_{n,k} = \Cl_{n,k} \otimes_\RR \CC \to \End(\Sigma_{n,k})$. If $n+k$ is odd, then there are precisely two irreducible (complex) representations $\rho^+ \colon \CCl_{n,k}  \to \End(\Sigma^+_{n,k})$ and $\rho^- \colon \CCl_{n,k}  \to \End(\Sigma^-_{n,k})$, distinguished by whether $\omega_\CC = i^{\frac{n-k+1}{2}} E_1 E_2 \cdots E_{n+k}$ acts by $+\id$ or $-\id$.
\end{proposition}

\begin{proposition}[{cf. \cite[Thm~ 4.3, Table~I]{lawson.michelsohn:89}}]
The representations $\rho \colon \CCl_{n,k} \to \End(\Sigma_{n,k})$ and $\rho^+ \oplus \rho^- \colon \CCl_{n,k}  \to \End(\Sigma^+_{n,k}) \oplus \End(\Sigma^-_{n,k})$ are a complex linear isomorphisms if $n+k$ is even or odd, respectively.
\end{proposition}

Left and right multiplication turn $\CCl_{n,k}$ into a $\CCl_{n,k}$-bimodule. Similarly, post- and pre-compostion with $\rho^{(\pm)}(-)$ make $\End(\Sigma^{(\pm)})$ a $\CCl_{n,k}$-bimodule (and hence also $\End(\Sigma^+_{n,k}) \oplus \End(\Sigma^-_{n,k})$). The isomorphisms above turn into isomorphisms of $\CCl_{n,k}$-bimodules.

As an immediate consequence, we get well-behaved comparison isomorphisms between
 the spinor bundles $\Sigmairr^{(\pm)}N = P_{\Spin_0}N \times_{\Spin_0} \Sigma_{n,k}^{(\pm)}$ and $\Sigma_{\Cl}N = P_{\Spin_0}N \times_{\Spin_0} \Cl_{n,k}$:
\begin{corollary} \label{cor:BundleIsos}
The representations $\rho$ and $\rho^+ \oplus \rho^-$ induce complex linear vector bundle isomorphisms
\begin{equation} \label{eq:BundleIsos}
\begin{aligned}
		\Sigma_{\Cl} N \otimes_\RR \CC &\longrightarrow \Sigmairr N \otimes_\CC (\Sigma_{n,k})^* \\
		\Sigma_{\Cl} N \otimes_\RR \CC &\longrightarrow \Sigmairr^+N \otimes_\CC (\Sigma_{n,k}^+)^* \oplus \Sigmairr^-N \otimes_\CC (\Sigma_{n,k}^-)^*,
\end{aligned}
\end{equation}
respectively. Furthermore, these isomorphisms respect the $\CCl(TN)$-$\CCl_{n,k}$-bimodule structure and preserve the connection. 
\end{corollary}

We are also interested in scalar products. On $\End(\Sigma_{n,k}^{(\pm)})$ (and hence also on $\End(\Sigma^+_{n,k}) \oplus \End(\Sigma^-_{n,k})$), there is a canonical positive definite Hermitian scalar product. Namely, starting from an arbitrary scalar product on $\Sigma_{n,k}^{(\pm)}$, we may construct a scalar product that is invariant under multiplication by $E_1, E_2,  \ldots$ and $E_{n+k}$ by an averaging procedure. By Schur's lemma, a scalar product with this invariance property is unique up to a factor. Taking the induced scalar product on the dual (in the complex sense), we obtain a scalar product on $\End(\Sigma_{n,k}^{(\pm)}) = \Sigma_{n,k}^{(\pm)} \otimes_\CC (\Sigma_{n,k}^{(\pm)})^*$. This turns out to be independent of the normalizing factor. In fact, it is given by $\< A,B \> = \tr(A B^*)$, where the adjoint $(-)^*$ is taken with respect to the scalar product on $\Sigma_{n,k}^{(\pm)}$.

On $\CCl_{n,k}$ (and similarly for $\Cl_{n,k}$), we may describe a scalar product as follows: For $I = \{i_1 < \ldots < i_l\} \subseteq \{1, \ldots, n+k\}$, let $E_I = E_{i_1} \ldots E_{i_l} \in \CCl_{n,k}$. Then $(E_I)_{I \subseteq \{1, \ldots, n+k\}}$ is a basis of $\CCl_{n,k}$, and we require it to be orthogonal with $\< E_I, E_I \> = 2^{\lfloor \frac{n+k+1}{2} \rfloor}$ for all $I$.

\begin{lemma} \label{lem:ScalarProducts}
The $\CCl_{n,k}$-bimodule isomorphisms $\rho$ and $\rho^+ \oplus \rho^-$ are isometric.
\begin{proof}
	It only remains to show that the scalar product is preserved. Let us first consider the case, where $n+k$ is even. When $\< - , - \>$ is an $E_1$-, $E_2$- $, \ldots$ and $E_{n+k}$-invariant scalar product on $\Sigma_{n,k}$, then $\<v, \rho(E_I)^* \rho(E_I) w\> = \<E_I \cdot v, E_I \cdot w\> = \<v, w\>$ for all $v, w \in \Sigma_{n,k}$ and $I \subseteq \{1, \ldots, n+k\}$. Hence, $\tr(\rho(E_I)\rho(E_I)^*) = \tr(\rho(E_I)^*\rho(E_I)) = \tr(\id) = \dim(\Sigma_{n,k}) = 2^{\frac{n+k}{2}}$ and it remains to show that $\tr(\rho(E_I)\rho(E_J)^*) = 0$ for $I \neq J \subseteq \{1,\ldots, n+k\}$.
	
	Note, that $\rho(E_J)^*$ coincides with $\rho(E_J)$ up to a factor of $\pm 1$ and that $\rho(E_I)\rho(E_J) = \pm \rho(E_{I \triangle J})$, where $I \triangle J$ denotes the symmetric difference of the subsets $I, J \subseteq \{1, \ldots, n+k\}$. Thus it only remains to show that $\tr(\rho(E_I)) = 0$ for $I \neq \emptyset$.
	
	It is easy to see that, when $\emptyset \neq I \neq \{1,\ldots, n+k\}$, there is some $j \in \{1, \ldots, n+k\}$ with $E_I E_j = - E_j E_I$. Moreover, as $n+k$ is assumed even, this also true for $I = \{1,\ldots, n+k\}$ (and any $j$). From $\tr(\rho(E_I)) = \tr(\rho(E_I)\rho(E_j)\rho(E_j)^{-1}) = - \tr(\rho(E_j)\rho(E_I)\rho(E_j)^{-1})= -\tr(\rho(E_I))$ it follows that $\tr(\rho(E_I)) = 0$.
	
	If $n+k$ is odd, the same proof works (taking the representation $\rho^+ \oplus \rho^-$ instead of $\rho$) with two changes. Firstly, in this case $\tr(\id) = \dim(\Sigma_{n,k}^+) + \dim(\Sigma_{n,k}^-) = 2^{\frac{n+k+1}{2}}$.
Secondly, for $I = \{1,\ldots, n+k\}$  any $E_j$ commutes with $E_I$. However, in this case $\tr((\rho^+ \oplus \rho^-)(E_I))=0$ still holds as $E_I$ coincides with $\omega_\CC$ up to a power of $i$ and $\tr((\rho^+ \oplus \rho^-)(\omega_\CC))=\tr(\id_{\Sigma_{n,k}^+})+ \tr(-\id_{\Sigma_{n,k}^-}) = 0$.
\end{proof}
\end{lemma}

In general (if $n,k \neq 0$), the positive definite scalar product will not be invariant under left-multiplication by elements in $\Spin_0(n,k)$, \ie the identity component of $\Spin(n,k) \subseteq \CCl_{n,k}$. Thus it does not give rise to a scalar product on the spinor bundle. This can be fixed by considering
\begin{align} \label{eq:InnerProduct}
	\llangle \phi, \psi \rrangle \coloneqq \< \epsilon_k E_{n+1} \cdots E_{n+k} \cdot \phi, \psi \>
\end{align}
instead, where $\epsilon_k=1$ if $k \equiv 0,1 \mod 4$ and $\epsilon_k = i$ if $k \equiv 2,3 \mod 4$, see also  \cite[Ch.~1, Satz 1.11]{baum:81}).
The \emph{indefinite} Hermitian product $\llangle -,- \rrangle$ satisfies
\begin{align} \label{eq:ClMultInnerProd}
	\llangle X \cdot \phi, \psi \rrangle = (-1)^{k+1} \llangle \phi, X \cdot \psi \rrangle
\end{align}
for all $X \in \RR^{n,k}$ and, as a consequence, is $\Spin_0(n,k)$-invariant (cf.\ \cite[Ch.~1, Satz 1.12]{baum:81}). Thus it gives rise to a non-degenerate inner product on the spinor bundle, which will be automatically parallel.

In the case of real spinor bundles such as $\Sigma_{\Cl}N$, we may use the formula \eqref{eq:InnerProduct} without the $\epsilon_k$-factor to define a non-degenerate bilinear form. It will be symmetric for $k \equiv 0,1 \mod 4$ and anti-symmetric for $k \equiv 2,3 \mod 4$.
If we agree to produce a hermitian form on the complexification out of a symmetric bilinear form by sesquilinear extension and out of an anti-symmetric bilinear form by sesquilinear extension following multiplication by $i$, we obtain the following:
\begin{corollary} \label{cor:ScalarProducts}
The bundle isomorphisms \eqref{eq:BundleIsos} are isometric with respect to the Hermitian inner products $\llangle -,- \rrangle$.
\end{corollary}
Although this inner product is naturally defined on the spinor bundle, for certain arguments it is more useful to work with definite scalar products.
A splitting of $TN$ into a space- and a timelike subbundle, locally given by positively oriented timelike frames $(e_{n+1}, \ldots, e_{n+k})$, gives rise to such a scalar product by
\begin{align*}
	\< \phi, \psi \> = \llangle \epsilon_k^{-1} e_{n+k} \cdots e_{n+1} \cdot \phi, \psi \rrangle,
\end{align*}
or the same formula without the $\epsilon_k$ in the real case.
Note, however, that $\< -, - \>$ depends on the chosen splitting and is in general not parallel.

%%%%%%%%%%%%%%%%%%%%%%%%%%%%%%%%%%%%%%%%%%%%%%%%%%%%%%%
\section{Parallel spinors and spacelike hypersurfaces}\label{sec.parallel-spinors.hypersurfaces}
%%%%%%%%%%%%%%%%%%%%%%%%%%%%%%%%%%%%%%%%%%%%%%%%%%%%%%%

From now on $\Sigma \overline{M}$ denotes the spinor bundle over the Lorentzian manifold $(\overline{M},\ol{g})$, obtained by some (fixed) representation
$\rho$ of the Clifford algebra, so it can be  $\Sigma_{\Cl} \ol{M}$, $\Sigmairr\ol{M}$ or it can be obtained from some other reprensentation.  The statements of this section hold independently on which representation is chosen, the proofs do not depend on it. Alternatively, one might always think of the case $\Sigmairr\ol{M}$ and use Corollary~\ref{cor:BundleIsos} to translate the statements to the $\Sigma_{\Cl} \ol{M}$-case later, when necessary.

%------------------------------------------------------
\subsection{The hypersurface spinor bundle}
%------------------------------------------------------
\label{sec:HypSpinBun}
In this subsection we fix a Riemannian metric $g$ and
a symmetric $2$-tensor $K\in \Gamma(T^*M\otimes T^*M)$.
If $(M,g,K)$ is a geometric hypersurface in $(\overline M,\ol g)$, then
we will need to restrict spinors from $\overline{M}$
to  $M$. These spinors are then sections of the restricted bundle
$\Sigma \overline{M}_{|M}$, and we should consider this bundle from an intrinsic point of view. The bundle $\Sigma \overline{M}_{|M}$ carries several structures: a connection, a Clifford multiplication by vectors, two hermitian products and more. We will see that this restricted spinor bundle with all its structures can be constructed intrinsically and will be called $\ol\Sigma M$;
it depends only on $(M,g,K)$ and the spin structure on $M$.

This bundle has the property, that for any Lorentzian manifold $(\overline{M},\overline{g})$ with $M$ as spacelike hypersurface, induced spin structure, induced metric $g$ and induced second fundamental form $K$ we have $\Sigma \overline{M}_{|M}=\ol\Sigma M$.

At first, for a section $(Y,f)$ of $TM\oplus \underline{\RR}\to M$, where $\underline{\RR}$ denotes the trivial bundle, we define the connection
  $$\ol\nabla_X(Y,f)= \Bigl(\nabla_XY-f W(X),X(f)+ K(X,Y) f\Bigr).$$
This definition is taylored to achieve the following: If $(M,g,K)$ is a geometric hypersurface of $(\ol M,\olg)$, then $(Y,f)\mapsto Y+fe_0$ is a connection preserving isomorphism from
$(TM\oplus \underline{\RR},\ol\nabla)$ to $T\ol{M}_{|M}$, where $T\ol{M}_{|M}$ carries the Levi-Civita connection of $(\ol{M}, \ol{g})$. Motivated by this isomorphism, the elements $1\in \underline\RR$ will also be written as $e_0$.

The bundle $TM \oplus \underline\RR$ carries a bundle metric that mimics the metric $\ol g$ on $T\olM_{|M}$: It is defined by $g$ on $TM$, the negative definite standard metric on $\underline\RR$ and orthogonality of the sum.
Let $P_{\SO_0(n,1)}(M)$ be the $\SO_0(n,1)$-principal bundle of time- and space-oriented generalized orthonormal frames of $TM \oplus \underline\RR$. A spin structure on $M$ yields a $\Spin_0(n,1)$-principal bundle $P_{\Spin_0(n,1)}(M)$ together with an equivariant map
$\theta: P_{\Spin_0(n,1)}(M)\to P_{\SO_0(n,1)}(M)$. We define
$\ol\Sigma M:=P_{\Spin_0(n,1)}(M)\times_\rho \Sigma$,
where $\rho \colon \Cl_{n,1} \to \End(\Sigma)$ is a $\Cl_{n,1}$-representation.
This vector bundle carries several structures that will be relevant for us. From the connection $\ol\nabla$ we obtain connection-$1$-forms on the principal bundles and finally a connection on $\ol\Sigma M$, also called $\ol\nabla$.
A $\Spin_0(n,1)$-invariant (indefinite) scalar product on $\Sigma$ gives rise to a $\ol\nabla$-parallel indefinite scalar product $\llangle\platz,\platz\rrangle$ on  $\ol\Sigma M$. Associated to the Clifford multiplication of $\RR^{n,1}$ on $\Sigma$ we obtain a $\ol\nabla$-parallel Clifford multiplication
$$(TM\oplus \underline\RR,\ol\na)\otimes (\ol\Sigma M,\ol\na) \to (\ol\Sigma M,\ol\na).$$
Restricted to $e_0$ we obtain an involution
$e_0\in\End(\ol\Sigma M)$, which is however not $\ol\nabla$-parallel; we have $\ol\nabla_Xe_0=-W(X)$.
The bundle $\ol\Sigma M$ also carries a positive definite scalar product
which can either be obtained by the formula
$\<\phi,\psi\>=\llangle e_0\cdot \phi,\psi\rrangle$
or as the scalar product associated to the positive definite scalar product
on $\Sigma$ using only frames $A\in P_{\Spin_0(n,1)}(M)$ for which $E_0$ is
the first vector of $\theta(A)$. Again this positive definite scalar
product is not $\ol\na$-parallel.
As a consequence of \eqref{eq:ClMultInnerProd}, the involution $e_0$ is self-adjoint with respect to this definite scalar product.

Obviously, if  $(M,g,K)$ is a geometric hypersurface of $(\ol M,\olg)$, then
we have $\Sigma\ol M_{|M} = \ol\Sigma M$ with all the structures mentioned above.

If we carry out this construction for $(g,0)$ instead of $(g,K)$ we obtain an
isomorphic bundle, however with a different connection, denoted as $\nabla$.
The positive definite scalar product $\<\platz,\platz\>$ is $\nabla$-parallel, and obviously $\ol\Sigma M$ with this scalar product,
the induced Clifford multiplication and the connection $\nabla$ is isomorphic
to the classical spinor bundle defined in terms of the representation
$\rho|_{\Spin(n)}$. 

Following step by step  through our above construction
how $K$ affects the connection-$1$-forms on  $P_{\SO_0(n,1)}(M)$, $P_{\Spin_0(n,1)}(M)$, and finally for $\phi\in \Gamma(\ol\Sigma M)$ we obtain 
\begin{align} \label{eq:ConnBtwConns}
  \ol\nabla_X\phi = \nabla_X\phi+ \frac12 e_0 \cdot W(X)\cdot\phi.
\end{align}
When $\ol\Sigma M$ is a complex bundle, then we may define a second kind of Clifford multiplication by $Y\bigcdot \phi:=ie_0\cdot Y\cdot \phi$. Then \eqref{eq:ConnBtwConns} reads
\begin{align*}
	\ol\nabla_X\phi = \nabla_X\phi-\frac i2  W(X) \bigcdot \phi.
\end{align*}

%---------------------------------------------------------
\subsection{Dirac currents}
%--------------------------------------------------

Let $\ol\Sigma M$ be a hypersurface spinor bundle and $\phi \in \ol\Sigma_p M$, $p\in M$. Recall that $\langle e_0 \cdot X \cdot \phi,\, \phi \rangle$ is always real, for any $X \in T_pM$:
	\begin{align*}
		\langle e_0 \cdot X \cdot \phi,\, \phi \rangle = - \langle \phi,\, X \cdot e_0 \cdot \phi \rangle = \langle \phi,\, e_0 \cdot X \cdot \phi  \rangle = \overline{ \langle e_0 \cdot X \cdot \phi,\, \phi \rangle }.
	\end{align*}
Because of this, the Riemannian Dirac current $U_\phi\in T_pM$ is well-defined by demanding 
\begin{equation}\label{def-Riemann-Dirac-current}
g(U_\phi,X)=\<e_0 \cdot X \cdot \phi,\phi\>\quad \forall X\in T_pM.
\end{equation}
In the case of a complex hypersurface spinor bundle, this takes the more familiar form
\begin{equation*}
g(U_\phi,X)=-i \<X\bigcdot \phi,\phi\>\quad \forall X\in T_pM
\end{equation*}
cf.~\cite{baum.leistner.lischewski:16} and \cite{ammann.kroencke.mueller:p19}.

The Riemannian Dirac current is tightly connected with the Lorentzian Dirac current $V_\phi \in T_p\ol M=T_pM \oplus\RR$ defined by the condition
\begin{align}\label{def-Lorentz-Dirac-current}
	\ol g(V_\phi, X) = -\llangle X\cdot \phi,\phi\rrangle=- \<e_0 \cdot X \cdot \phi, \phi \> \quad \forall X\in T_p\olM.
\end{align}

It is straigthforward to check that
\begin{equation}\label{UV.equation}
V_\phi = -U_\phi + u_\phi e_0,
\end{equation}
where $u_\phi = \|\phi\|^2 \coloneqq \<\phi, \phi\>$.
In particular, $V_\phi \neq 0$ unless $\phi = 0$.
We calculate
%$\| V_\phi \cdot \phi \|^2 = -\ol g(V_\phi,V_\phi) \|\phi\|^2 = 0$ and hence $V%_\phi \cdot \phi = 0$.
\begin{align}
	\| V_\phi \cdot \phi \|^2 &= \|-U_\phi \cdot \phi\|^2 + 2\Re \<-U_\phi \cdot \phi, u_\phi e_0 \cdot \phi \>  +  \|u_\phi e_0 \cdot \phi\|^2 \nonumber\\
		&= g(U_\phi, U_\phi) \|\phi\|^2 - 2 u_\phi \Re(\<e_0 \cdot U_\phi \cdot \phi, \phi\>) + u_\phi^2 \|\phi\|^2\nonumber \\
		&= g(U_\phi, U_\phi) \|\phi\|^2 - 2 u_\phi g(U_\phi,U_\phi) - \ol g(u_\phi e_0, u_\phi e_0) \|\phi\|^2 \nonumber\\
		&= -\ol g(V_\phi,V_\phi) \|\phi\|^2.\label{Vphi-cdot-phi}
\end{align}
We have obtained the following well-known lemma. If $V_\phi$ is lightlike, then
\eqref{Vphi-cdot-phi} immediately provides $V_\phi \cdot \phi = 0$. 
\begin{lemma}\label{lem.lorentz-dirac.current.annih}
Let $\phi \in \ol\Sigma M$ with $\phi \neq 0$.  Then the vector $V_\phi$ is non-zero, causal (lightlike or timelike), and future-oriented. Moreover, if $V_\phi$ is lightlike, then $V_\phi \cdot \phi = 0$.
\end{lemma}

These definitions and result obviously also apply to section $\phi \in \Gamma(\ol\Sigma M)$ and vector fields $U_\phi\in\Gamma(TM)$ and  $V_\phi\in\Gamma(T\olM_M)$.

%---------------------------------------------------------
\subsection{The Cauchy problem for parallel spinors}\label{subsec.Cauchy.par.spin}
%---------------------------------------------------------

In this section we recall relevant parts of a series of recent articles
by Baum, Leistner and Lischewski \cite{baum.leistner.lischewski:16}, \cite{lischewski:15-preprint}, \cite{leistner.lischewski:19}, \cite{baum.leistner.lecture.notes:HH}. Parts of slightly modified results were proven in a simpler way by Seipel \cite{seipel:19} inspired by an oral communication by P. Chrusciel.

Let $(\overline{M}, \overline{g})$ be a time-oriented Lorentzian spin manifold with a spacelike hypersurface $M \subseteq \overline{M}$. Then $M$ inherits an induced spin structure and an induced Riemannian metric.
Assume, further, that $(\overline{M}, \overline{g})$ carries a non-trivial parallel spinor $\Phi$, then restriction to  $M$ yields a spinor $\phi = \Phi_{|M}$ with $\ol\nabla \phi =0$.
Equivalently, $\phi$ satisfies the so-called \emph{imaginary W-Killing spinor equation}
\begin{align}\label{gen.imag.killing}
	\nabla_X\phi  = -\frac{1}{2}e_0 \cdot W(X) \cdot \phi = \frac i2  W(X) \bigcdot \phi,
\end{align}
for all $X\in TM$, where again $W = -K^\sharp$ is obtained from the second fundamental form $K$ of the embedding.

\begin{definition}\label{def.initial.data.triples}
An \emph{initial data triple} $(g,K,\phi)$ on $M$ consists of a Riemannian metric $g$, a symmetric $2$-tensor $K$ and a nowhere vanishing spinor $\phi\in \Gamma(\ol\Sigma M)$ solving~\eqref{gen.imag.killing}.
We say that the triple  $(g,K,\phi)$ is lightlike (or timelike), if the Lorentzian Dirac current  $V_\phi$ -- defined by \eqref{def-Lorentz-Dirac-current} -- of $\phi \in \Gamma(\ol\Sigma M)$ is lightlike (or timelike) everywhere.
(By Lemma \ref{Lem:VphiParallel} below, if this holds in one point, then it is true on the whole connected component.)
\end{definition}
Thus any Lorentzian manifold with a parallel spinor induces an initial data triple on any spacelike hypersurface.

The Cauchy problem for parallel spinors asks for the converse of this:
Given a spin manifold $M$ along with an initial data triple $(g,K,\phi)$, is it possible to embed~$M$ into a time-oriented  Lorentzian spin manifold
$(\overline{M}, \overline{g})$ with parallel spinor $\Phi$ such that $g$, $K$,  and $\phi$ are induced on $M$ in the way described above?

\begin{lemma} \label{Lem:VphiParallel}
Assume that $V_\phi$ is the Lorentzian Dirac current of a $\ol\nabla$-parallel spinor $\phi \in \Gamma(\ol\Sigma M)$ along~$M$ (\ie \eqref{gen.imag.killing} holds on~$M$). Then $V_\phi$ is $\ol\nabla$-parallel along~$M$. Equivalently,
\begin{equation} \label{eq:VectorConstr}
\begin{aligned}
	\nabla U_\phi &= -u_\phi W \\
	du_\phi &= K(U_\phi, -).
\end{aligned}
\end{equation}
\begin{proof}
  Recall that the indefinite inner product $\llangle \Phi, \Psi \rrangle = \<e_0 \cdot \Phi, \Psi \>$ is $\ol\nabla$-parallel, cf.\ the discussion following \eqref{eq:InnerProduct}.  Thus for any $X \in \Gamma(T\ol{M}_{|M})$ and $Y \in TM$, we have
\begin{align*}
	\ol g(\ol\nabla_Y V_\phi, X) &= \del_Y \ol g(V_\phi,X) - \ol g(V_\phi, \ol\nabla_Y X) \\
							&= -\del_Y \bigl\llangle X \cdot \phi, \phi \bigr\rrangle + \bigl\llangle (\ol\nabla_Y X) \cdot \phi,\phi \bigr\rrangle \\
							&= - \bigl\llangle X \cdot \ol\nabla_Y \phi, \phi \bigr\rrangle - \bigl\llangle X \cdot \phi, \ol\nabla_Y \phi \bigr\rrangle =0.
\end{align*}
This implies that $V_\phi$ is $\ol\nabla$-parallel along $M$. The system of equations \eqref{eq:VectorConstr} is a simple reformulation of $\ol\nabla V_\phi=0$ using \eqref{UV.equation}.
\end{proof}
\end{lemma}

This lemma allows for the following strategy for solving the Cauchy problem for parallel spinors:
First, for the data $(g, K, U_\phi, u_\phi)$ subject to the constraint \eqref{eq:VectorConstr}, solve the (analogous) Cauchy problem for parallel vector fields.
Then try to parallely extend the spinor onto the obtained Lorentzian manifold.

This program has been carried through by Baum, Leistner and Lischewski \cite{lischewski:15-preprint}, \cite{leistner.lischewski:19}, \cite{baum.leistner.lischewski:16} in the case, where the Lorentzian Dirac current $V_\phi$ is lightlike everywhere.
Note that the spin contraints considered by these authors typically consist of \eqref{gen.imag.killing} along with 
\begin{equation} \label{eq:SpinConstr2}
	\begin{aligned}
		U_\phi \bigcdot \phi &= i u_\phi \phi \\
		u_\phi &= \|U_\phi\|.
	\end{aligned}
\end{equation}
These two equations, however, are just equivalent to the condition that $V_\phi = -U_\phi + u_\phi e_0$ is lightlike, as $V_\phi \cdot \phi = 0$ in this case, see Lemma~\ref{lem.lorentz-dirac.current.annih}.

\begin{corollary}\label{cor:contraints.valid}
If $(g,K,\phi)$ is a light initial data triple, then the Dirac current $U_\phi$ and the function $u_\phi:=\|\phi\|^2$ satisfy \eqref{eq:SpinConstr2}.
\end{corollary}

The case of timelike $V_\phi$ does not play a major role in the work by Baum, Leistner and Lischewski, as their motivation came from special holonomy, and
in this case the holonomy group coincides with the holonomy group of the Riemannian factor: the parallel timelike vector field implies that the Lorentzian manifold is locally a product of a Riemannian manifold with $(\RR,-dt^2)$.
%So they do not provide new  no new holonomy group appear: the holonomy has a product structure of a Riemannian holonomy group with a trivial line bundle.

As it does, however, play an important role here, we are interested in solving the Cauchy problem in this case as well.
Whereas the original work of Baum, Leistner, Lischewski involves sophisticated techniques for solving the occuring partial differential equations, it was later observed by Chruściel that a solution of the Cauchy problem can in fact be explicitly written down.
This argument appeared first in \cite{seipel:19} and it fortunately immediately extends to the timelike case. Note that in the following we turn covariant tensors on $M$, \ie sections of $(T^*M)^{\otimes \ell}$ into covariant tensors on $M\times \RR$ by pullback under the projection $M\times \RR\to M$. 

\begin{theorem} \label{thm:SolCauchyP}
Let $(g,K)$ be an initial data set, $U \in \Gamma(TM)$ and $u \in C^\infty(M)$ a positive function such that \eqref{eq:VectorConstr} hold. Then
\begin{align*}
\overline{g} &= (\|U\|^2-u^2)\, dt^2 - dt \otimes g(U,-) - g(U,-) \otimes dt +g
\end{align*}
defines a Lorentzian metric on $M \otimes \RR$, where $t$ denotes the $\RR$-coordinate, with the following properties:
\begin{itemize}
\item $M \times \{t\}$ is a spacelike hypersurface for any $t \in \RR$.
	\item The induced initial data set on $M \times \{t\}$ is $(g,K)$, when the time-orientation (determining the sign of $K$) is chosen such that $\del_t$ is future-pointing.
	\item $\del_t = -U + u e_0$, where $e_0$ denotes the future-pointing unit normal on $M \times \{t\}$.
	\item $\del_t$ is $\overline{\nabla}$-parallel.
\end{itemize}
Moreover, if $\phi \in \Gamma(\ol\Sigma M)$ is a nowhere vanishing\footnote{It follows from $\ol\nabla \phi = 0$ that this is the case as long as for any component $M_1$ of  $M$ we have $\phi|_{M_1}\not\equiv 0$.} spinor on $M$  satisfying \eqref{gen.imag.killing}, then the Lorentzian manifold $(M \times \RR, \overline{g})$ for $U = U_\phi$ and $u = u_\phi$ admits a $\overline{\nabla}$-parallel spinor extending $\phi$ and whose Dirac flow is $\del_t$.
\end{theorem}

Again the theorem holds independently on which definition of ``spinor'' is chosen. Note that we implicitly assume that $M\times \RR$ carries the unique spin structure whose restriction to $M\cong M\times\{0\}$ is the given spin structure on $M$.

\begin{proof}
  We start off by showing that $e_0 \coloneqq \frac{1}{u}\bigl(U + \del_t\bigr)$ is orthonormal on $M = M \times \{t\} \subseteq M \times \RR$.
  
From
	\begin{align*}
		\overline{g}(e_0,\platz) &= \left(\|U\|^2 - u^2\right) \frac{1}{u}\, dt - \frac{1}{u} g(U,\platz) - \frac{\|U\|^2}{u}\, dt + \frac{1}{u} g(U,\platz) = -u\, dt,
	\end{align*}
	it immedately follows that $\overline{g}(e_0,X)=0$ for all $X \in TM$ and $\overline{g}(e_0,e_0) = -1$.
	As, moreover, the restriction of $\overline{g}$ to $M \times \{t\}$ obviously yields the Riemannian metric $g$, we obtain that $\overline{g}$ is a Lorentzian metric and $M \times \{t\}$ is spacelike.
	
	We have to show that the induced second fundamental form on $M \times \{t\}$ is $K$.
	One way to calculate the second fundamental form is by the formula
        $\frac{1}{2} (\mathcal{L}_{e_0} \ol{g})|_{TM\otimes TM}$,
        where $e_0 = \frac{1}{u}(U + \del_t)$ is the future-directed unit normal of the foliation $(M \times \{t\})_{t \in \RR}$.
	For a vector field $X$ of $M \times \{t\}$, we denote by the same symbol its constant extension to $M \times \RR$, \ie the extension with $\mathcal{L}_{\del_t} X = 0$.
	In this case,
	\begin{align*}
		\mathcal{L}_{e_0} X = -[X, e_0]
	&= \frac{\del_X u}{u^2} \bigl(U + \del_t\bigr) - \frac{1}{u}\bigl[X, U + \del_t\bigr] \\
	&= \frac{\del_X u}{u} e_0 + W(X) + \frac{1}{u} \nabla_U X 
	\end{align*}
	using $\nabla_X U = -uW(X)$.
	Thus, we obtain for $X, Y \in \Gamma(TM)$
	\begin{align*}
		(\mathcal{L}_{e_0} \ol{g})(X, Y)
		&= \del_{e_0} (g(X,Y)) - g\bigl(\mathcal{L}_{e_0} X,Y\bigr) - g\bigl(X,\mathcal{L}_{e_0}Y\bigr) \\
		&= \frac{1}{u}\del_{U} (g(X,Y)) -  g\Bigl(W(X)+ \frac{1}{u}\nabla_U X, Y\Bigr) - g\Bigl(X, W(Y) +  \frac{1}{u}\nabla_U Y\Bigr)\\
		&= 2K(X,Y)
	\end{align*}
	as claimed.
	
	Now that we have calculated the second fundamental form, it is easy to show that $\del_t$ is parallel with respect to the Levi-Civita connection of $\overline{g}$:
	First of all, when $X \in TM$, then
	\begin{align}
		\ol\nabla_X \del_t = \ol\nabla_X (ue_0 - U)
		&= (\del_X u) e_0 + u \ol\nabla_X e_0 - \nabla_X U - K(X,U)e_0 \nonumber\\
		&= K(X,U) e_0 - uW(X) + uW(X) - K(X,U)e_0 = 0.\label{eq:partialt.par}
	\end{align}
	Using this, we obtain furthermore
	\begin{align*}
		\overline{g}(\ol\nabla_{\del_t} \del_t, X) &= \del_t \overline{g}(\del_t, X) - \overline{g}(\del_t, \ol\nabla_{\del_t}X) \\
		&= -\del_t g(U, X) - \overline{g}(\del_t, \mathcal{L}_{\del_t} X)
			- \overline{g}(\del_t, \ol\nabla_X \del_t) \\
                &= 0+0+0 = 0
	\end{align*}
	for a constant extension of some $X \in \Gamma(TM)$.	
	Together with
	\begin{align*}
		\overline{g}(\ol\nabla_{\del_t} \del_t, \del_t) = \frac{1}{2} \del_t \overline{g}(\del_t, \del_t) &= \frac{1}{2} \del_t \bigl(\|U\|^2-u^2\bigr) = 0,
	\end{align*}
	this implies that $\ol\nabla \del_t = 0$.
	It remains to show the claim concerning the parallel spinor.
	For this, note that the constant extension of a vector field $X$ is parallel in $t$-direction, as \eqref{eq:partialt.par} implies that $\ol\nabla_{\del_t} X= \ol\nabla_{\del_t} X-\ol\nabla_X \del_t  = \mathcal{L}_{\del_t} X= 0$.
        Thus, if $(e_0,e_1,\ldots,e_n)$ is a positively space- and time-oriented
	  generalized orthonormal frame
          of $(T_{(x,0)} (M \times \RR),\olg_{(x,0)})$, then we can identify it with the same kind of frame at $(x,t)$, denoted as $(e_0,e_1,\ldots,e_n)_t$, and the path
          $t\mapsto (e_0,e_1,\ldots,e_n)_t\in P_{\SO_0(n,1)}(M \times \RR,\olg)|_{(x,t)}$ is parallel. We obtain an isomorphism of $P_{\SO_0(n,1)}$-principal bundles preserving the  connection-$1$-form
          \begin{eqnarray*}
            P_{\SO_0(n,1)}(M \times \RR)_{|M\times \{0\}} \times \RR &\longrightarrow &P_{\SO_0(n,1)}(M \times \RR,\olg)\\
             \bigl((e_0,e_1,\ldots,e_n),t\bigr) &\longmapsto&  (e_0,e_1,\ldots,e_n)_t
          \end{eqnarray*}
          This map yields a canonical isomorphism of principal bundles with connection-$1$-form between 
          $P_{\SO_0(n,1)}(M \times \RR,\olg)$ and the pull-back of  $P_{\SO_0(n,1)}(M \times \RR)_{|M\times \{0\}}$ under the projection map $M\times \RR\to M\times\{0\}$.
          This canonical isomorphism lifts to the spin structure, and then to the spinor bundle. We obtain an isomorphism
          $$E:\Gamma(\ol\Sigma M) = \Gamma(\Sigma \ol M_{|M}) \to \{\phi\in\Gamma(\Sigma\overline M)\mid \ol\nabla_{\del_t} \Phi =0\}$$
          which is a right inverse to restriction, and which satisfies $E(\ol\nabla_X \psi)= \ol\nabla_XE( \psi)$ for all $X\in TM$. For $\psi=\phi$, we see that $E(\phi)$ is a parallel extension of $\phi$ to $M\times \RR$.

          The statement about the Dirac flow follows from $V_\phi = -U_\phi + u_\phi e_0 = \del_t$.
\end{proof}

%%%%%%%%%%%%%%%%%%%%%%%%%%%%%%%%%%
\section{The kernel of the Dirac--Witten operator}\label{sec.kernel.DWo}
%%%%%%%%%%%%%%%%%%%%%%%%%%%%%%%%%%

The main goal of the article is to detect non-trivial homotopy groups in $\Dec(M)$ using index theory. For this it is of crucial importance to find sufficient criteria for~$M$ and~$(g,K)\in \Dec(M)$ for the Dirac--Witten operator $\Dirac^{(g,K)}$ to be invertible.

%-------------------------------------
\subsection{Dirac--Witten operators and Schrödinger--Lichnerowicz formula}
%-------------------------------------

Let us summarize some facts about the Dirac--Witten operator which essentially go back to \cite{witten:81}. For a coordinate free presentation we refer to \cite{parker.taubes:82} which provides sufficiently detailed proofs. The Schrö\-ding\-er--Lich\-ne\-ro\-wicz type formula also plays a major role in \cite{hijazi.zhang:03} and \cite{gloeckle:p2019}.

All references above only treat irrducible complex spinors.
The facts below also hold for other types of spinors, as the proofs do not rely on the particular choice of spinor module.
For $\Cl_{n,1}$-linear spinors, which are of a main interest to us, this is also an immediate consequence of Corollaries~\ref{cor:BundleIsos} and \ref{cor:ScalarProducts}.

Let $\ol\Sigma M$ be a hypersurface spinor bundle of $(M, g, K)$ as explained in Section \ref{sec:HypSpinBun}.
The Dirac-Witten operator $\olDirac$ of $\ol\Sigma M$ is defined similarly its Dirac operator $\Dirac$, but using the connection $\ol\nabla$ instead of $\nabla$, \ie by the local formula
\begin{align*}
	\olDirac \phi = \sum_{i=1}^n e_i \cdot \ol\nabla_{e_i} \phi,
\end{align*}
where $(e_1, \ldots, e_n)$ is a local orthonormal frame of $TM$.
A short calculation using \eqref{eq:ConnBtwConns} provides that Dirac-Witten and Dirac operator are related by the formula
\begin{align} \label{eq:ConnBtwDiracs}
	\olDirac \phi = \Dirac \phi - \frac12 \tr(K) e_0 \cdot \phi,
\end{align}
keeping in mind that $\tr(K) = -\tr(W)$ by our sign conventions.
From this formula, we can see that $\olDirac$ shares the principal symbol of $\Dirac$, and thus is a generalized Dirac operator. 
Moreover, as $e_0$ is self-adjoint w.\,r.\,t.~the positive definite scalar product, $\olDirac$ is formally self-adjoint.
Thus many functional analytic properties of Dirac operators, \eg discreteness of the spectrum, also apply to $\olDirac$. 
It is hence also not surprising that the Dirac--Witten operator satisfies a Schrö\-ding\-er--Lich\-ne\-ro\-wicz type formula, which reads
	\begin{align} \label{eq:SchrL}
	\olDirac^2 &= \overline\nabla^*\overline\nabla + \frac12 (\rho-e_0 \cdot j^\sharp \cdot), \\
	\intertext{with} \nonumber
	2\rho &= \scal + (\tr K)^2 - \|K\|^2 \\ \nonumber
	j &= - \mathrm{d} (\tr K) + \div K.
	\end{align}
This formula can be either deduced from the Schrödinger-Lichnerowicz formula of~$\Dirac$ using \eqref{eq:ConnBtwConns} and \eqref{eq:ConnBtwDiracs}, or checked by a straightforward calculation, which is done in \cite[Sec.~3]{parker.taubes:82}.

%------------------------------------------------
\subsection{From the kernel of the Dirac--Witten operator to initial data triples}
%------------------------------------------------
\label{sec:KernelDW}

The goal of this paragraph is to investigate the kernel of the Dirac-Witten operator on a closed manifold $M$. This will contribute to our efforts in this and the following subsections to find obstructions against the existence of non-zero harmonic Dirac-Witten spinors, \ie spinors $\phi\not\equiv 0$ with $\olDirac \phi = 0$.

\begin{theorem}\label{theorem.kernel-gives-initial-data}
Let $M$ be a closed spin manifold of dimension $n$. We assume that the initial data set $(g,K)$ satisfies the dominant energy condition, i.e. $\rho\geq |j|$.
Then for any $\phi\in \ker\olDirac$ the triple $(g,K,\phi)$ satisfies the constraint equation \eqref{gen.imag.killing} of the Cauchy problem for a parallel spinor.
Moreover, $(\rho e_0 - j^\sharp) \cdot \phi = 0$ in this case.
\end{theorem}
Thus, under the assumtions of the theorem, $(g,K,\phi)$ is an initial data triple, as defined in Definition~\ref{def.initial.data.triples}.

\begin{proof}
We calculate for an arbitrary spinor $\phi$
  $$\<\left(\rho-e_0\cdot j^\sharp\right)\cdot\phi,\phi\>\geq (\rho -\|e_0\|\,\|j^\sharp\|)\|\phi\|^2\geq 0.$$
  For a harmonic Dirac--Witten spinor, the Schrödinger--Lichnerowicz formula \eqref{eq:SchrL} then implies
  $$0=\int_M\<\ol\Dirac^2\phi,\phi\>\geq\int_M\<\ol\nabla^*\ol\nabla\phi,\phi\>.$$
As $M$ is assumed to be closed, this yields $\ol\nabla \phi=0$ via partial integration,
which in turn is equivalent to \eqref{gen.imag.killing}.
The second assertion then follows immediately from $0=\olDirac^2 \phi = \ol\nabla^* \ol\nabla \phi + \frac 12 (\rho- e_0 \cdot j^\sharp \cdot) \phi$.
\end{proof}

\begin{lemma} \label{Lem:DiracCurrent}
Assuming the setup of Theorem \ref{theorem.kernel-gives-initial-data}, let  $u_\phi:=\|\phi\|^2$ and let $U_\phi$ be the Riemannian Dirac current of $\phi$, see \eqref{def-Riemann-Dirac-current}. Then $\rho e_0 - j^\sharp$ and $V_\phi \coloneqq u_\phi e_0 - U_\phi$ are linearly dependent.
%\cbernd{B2Jonathan: (nicht zu erledigen). Ich habe den leichten Eindruck, dass dieses Aussage nur unzureichend bisher genutzt wurde; vor allem im lichtartigen Fall.}
\end{lemma}
\begin{proof}
It suffices to show that $\rho U_\phi = u_\phi j^\sharp$.
For any $X \in TM$, we calculate
	\begin{align*}
		g(\rho U_\phi, X)
		&= \rho \Re \left(\langle e_0 \cdot X \cdot \phi,\, \phi \rangle \right) = \Re \left( \langle X \cdot \phi,\, \rho e_0 \cdot \phi \rangle \right) \\
		&= \Re \left( \langle X \cdot \phi,\, j^\sharp \cdot \phi \rangle \right) = \frac12 \left( \langle X \cdot \phi,\, j^\sharp \cdot \phi \rangle + \langle j^\sharp \cdot \phi,\, X \cdot \phi \rangle \right) \\
		&= -\frac12 \left(\langle j^\sharp \cdot  X \cdot \phi,\, \phi \rangle + \langle X \cdot j^\sharp \cdot \phi,\,  \phi \rangle \right) = g(j^\sharp, X) \|\phi\|^2 \\
		&= g(u_\phi j^\sharp, X)
	\end{align*}
and the assertion follows.
\end{proof}

From now on, we assume that $M$ is connected.
Otherwise, the following arguments apply componentwise.
Unless $V_\phi \equiv 0$, which is the case if and only if $\phi \equiv 0$, it follows from Lemma~\ref{Lem:VphiParallel} that $V_\phi$ is either lightlike everywhere or timelike everywhere.

\begin{corollary} \label{cor:TlSpinorVac}
If $M$ is a closed connected spin manifold, $(g, K)$ an initial data set on $M$ subject to the dominant energy condition and $\phi \in \ker \olDirac$ with $V_\phi$ timelike everywhere, then $\rho \equiv 0 \equiv j$.
\end{corollary}
\begin{proof}
From the previous lemma, we get that $\rho e_0 - j^\sharp$ is either $0$ or timelike, being a multiple of $V_\phi$.
On the other hand, as $\phi$ is nowhere vanishing, it follows from $(\rho e_0 - j^\sharp) \cdot \phi = 0$ that $\rho e_0 - j^\sharp$ is either $0$ or lightlike.
So $\rho e_0 - j^\sharp = 0$.
\end{proof}

We get the following obstructions to the existence of nowhere vanishing Dirac-Witten harmonic spinors.
\begin{theorem}\label{thm:WdecInv-old}
Let $M$ be a closed connected spin manifold with a timelike initial data triple $(g, K,\phi)$. 
Then $M$ carries a Ricci-flat metric with a parallel spinor.
\end{theorem}

\begin{proof}
We embed~$M$ into the Lorentzian manifold $(M \times \RR, \ol g)$ from Theorem~\ref{thm:SolCauchyP}; note that \eqref{eq:VectorConstr} is satisfied because of \eqref{gen.imag.killing}, see Lemma~\ref{Lem:VphiParallel}.
Recall, that $(M \times \RR, \ol g)$ carries a parallel spinor $\ol \phi$ extending $\phi$ and that its Dirac flow is given by the (parallel) vector field $\del_t$.
We denote by $T$ the parallel vector field with  $\ol g(T, T) =-1$ obtained by normalizing $\del_t$.
We now define $h$ to be the Riemannian metric on $M$ such that the canonical projection $\pi \colon M \times \RR \to M$ is a semi-Riemannian submersion, \ie $\ol g(X, Y) = -\ol g(T, X) \ol g(T, Y) + h(d_p \pi(X), d_p \pi(Y))$ for all $p \in M \times \RR$ and $X, Y \in T_p(M \times \RR)$.

We seek to define a parallel spinor on $(M, h)$.
First of all, if $(T, e_1, \ldots, e_n)$ is a space- and time-oriented generalized $\ol g$-orthonormal frame, then $(d\pi(e_1), \ldots, d\pi(e_n))$ is an $h$-orthonormal frame.
Phrased differently, the $\SO(n)$-reduction of $P_{\SO_0(n,1)}(M \times \RR) \to M \times \RR$ defined by $T$ is canonically isomorphic to the pullback $\pi^*P_{\SO(n)}M \to M \times \RR$, \ie $(P_{\SO(n)}M) \times \RR \to M \times \RR$.
Passing to double covers, any spin structure $P_{\Spin_0(n,1)}(M \times \RR) \to P_{\SO_0(n,1)}(M \times \RR)$ induces a spin structure on the $\SO(n)$-reduction, and this corresponds uniquely to a spin structure on $(M, h)$.
To summarize, the principal bundles discussed so far fit into the following commutative diagram, where all (proper) squares are pullbacks:
\begin{equation*}
	\begin{tikzcd}
		P_{\Spin_0(n,1)}(M \times \RR) \dar & (P_{\Spin(n)}M) \times \RR \lar \dar \rar & P_{\Spin(n)}M \dar \\
		P_{\SO_0(n,1)}(M \times \RR) \drar & (P_{\SO(n)}M) \times \RR \lar \dar \rar & P_{\SO(n)}M \dar \\
		& M \times \RR \rar{\pi} & M
	\end{tikzcd}
\end{equation*}
In particular, a spinor bundle of $(M \times \RR, \ol g)$ may be obtained by pulling back a corresponding spinor bundle of $(M,h)$, and we shall henceforth identify $\ol \Sigma(M \times \RR) \cong (\ol \Sigma M) \times \RR$.

As $T$ is $\ol \nabla$-parallel, the Levi-Civita connection $\ol \nabla$ of $\ol g$ induces a connection on $(P_{\SO(n)}M) \times \RR \to M \times \RR$ and this connection coincides with the pullback of the Levi-Civita connection $\nabla$ of $h$.
The latter statement follows from the fact that for any $Y \in \Gamma(TM)$ the unique $\pi$-related vector field $\ol Y \in \Gamma(T^\perp) \subset \Gamma(T(M \times \RR))$ satisfies $d\pi(\ol\nabla_X \ol Y) = \nabla_{d\pi(X)} Y$  for all $X \in T(M \times \RR)$, which in the case $\ol g(X,T)=0$ is \cite[Lem.~7.45~(3)]{oneill:83} and holds for $X = T$ as $d\pi(\ol\nabla_T \ol Y)= d\pi([T,\ol Y]) + d\pi(\ol\nabla_{\ol Y}T)= [d\pi(T),d\pi(\ol Y)] + 0 = 0$.
It follows immediately that the connections induced by the Levi-Civita connections also coincide on the spin level.
In particular, after passing to associated bundles, a spinor $\ol \psi \in \Gamma((\ol \Sigma M) \times \RR)$ is a pulled-back spinor $\psi \in \Gamma(\ol \Sigma M)$ if and only if $\ol \nabla_T \ol \psi = 0$ and in this case $\ol \nabla_{X} \ol \psi = \nabla_{d\pi(X)} \psi$.

Applying this to $\ol \phi$ yields the desired parallel spinor on $(M,h)$. 
It should be noted that the spin structure constructed on $(M, h)$ and the spin structure we started with on $(M, g)$ correspond to the same topological spin structure on $M$.
\end{proof}

\begin{corollary}\label{cor:WdecInv-old}
Let $M$ be a closed connected spin manifold, $(g, K)$ an initial data set on $M$ subject to the dominant energy condition and $\phi \in \ker \olDirac \setminus \{0\}$. 
If $V_\phi$ is lightlike, then $b_1(M) \neq 0$. If $V_\phi$ is timelike, then $M$ carries a Ricci-flat metric with a parallel spinor.
\end{corollary}

\begin{proof}
If $V_\phi$ is lightlike, then $U_\psi^\flat$ is a nowhere vanishing closed $1$-form on $M$, thus $[U_\psi^\flat]$ is a non-trivial element in $H^1_{dR}(M)$.

If  $V_\phi$ is timelike, then the statement follows immediately from Theorems~\ref{theorem.kernel-gives-initial-data} and~\ref{thm:WdecInv-old}.
\end{proof}

Thus we have reduced the question for topological obstructions to timelike initial data triples to the well-studied problem  of determining obstructions to have a metric whose restricted holonomy is a product of group $\SU(k_1)$, $\Sp(k_2)$, $G_2$ and $\Spin(7)$, see Subsection~\ref{subsec.prelim.struc.ricciflat} for details.
In the lightlike case, we will derive further conclusions in the following subsections.

\begin{examples}\label{example.obstr.structured}
Let us give some examples of closed manifolds with obstructions to Ricci-flat metrics with nowhere vanishing parallel spinors. Due to Theorem~\ref{thm:WdecInv-old} this also provides obstructions to timelike initial data triples.
Some of these examples also feature a zero first Betti number so that there is no non-trivial initial data triple at all according to Corollary~\ref{cor:WdecInv-old}.
  \begin{enumerate}[{\rm(1)}]
 \item \label{example.obstruc.para.spin.i}
 If $M$ is a closed $n$-dimensional Ricci-flat manifold, then it follows from the Cheeger-Gromoll splitting theorem that its fundamental group 
contains a free abelian subgroup of finite index and rank at most $n$, see \eg  \cite[Chap.~V, Theorem~3.11]{sakai:96}.
 In particular, $\pi_1(M)$ is virtually abelian and has polynomial growth of degree at most $n$.
This yields an obstruction to a timelike initial data triple. E.g.\ manifolds that admit a metric of negative sectional curvature have an exponentially growing fundamental group and are thus ruled out.
\item\label{example.obstruc.para.spin.ii}
  Let $k\in \NN$. For  $x,y\in \RR^k$, $z\in \RR$ consider
   the matrix $H_{x,y,z}$, the $2k+1$-dimensional Heisenberg group $\maH_{2k+1}$ and the lattice $L_{2k+1}$ as follows:
   $$H_{x,y,z}:=\begin{pmatrix}1 & x_1& x_2& \cdots & x_k & z\\ 0 & 1 & 0 &\cdots&0 & y_k\\
   0 & 0 & 1 & 0 & & y_{k-1}\\
   \vdots & \vdots & 0 & 1 & & \vdots\\
   \vdots & \vdots & \vdots &  &\ddots & y_1\\
   0& 0 & 0 & \cdots &0 & 1
   \end{pmatrix},$$
   $$\maH_{2k+1}:=\left\{H_{x,y,z}\mid x,y\in \RR^k,z\in \RR\right\},\quad L_{2k+1}:=\left\{H_{x,y,z}\mid x,y\in \ZZ^k,z\in \ZZ\right\}.$$
   The $2k+1$-dimensional Heisenberg manifold $\mathbb{H}_{2k+1}\definedas L_{2k+1}\backslash\maH_{2k+1}$ has fundamental group $L_{2k+1}$ which grows as a polynomial of degree $2k+2>2k+1=\dim \mathbb{H}_{2k+1}$. Thus there are no Ricci-flat metrics on $\mathbb{H}_{2k+1}$. This manifold has first Betti number $b_1(\mathbb{H}_{2k+1})=2k$. 
 \item\label{example.obstruc.para.spin.iii}
   The map $H_{x,y,z}\mapsto H_{-x,-y,z}$ defines a group automorphism of order $2$ and we can consider the associated semi-direct product
   $$\Gamma:=L_{2k+1}\rtimes \ZZ/2\ZZ.$$
Consider a manifold $M$ of dimension in $\{ 4,5\ldots,2k+1\}$  with $\pi_1(M)=\Gamma$. Again the fundamental group of $M$ grows as a polynomial of order $2k+2$, thus no Ricci-flat metric exists. 
Because of $\Gamma/[\Gamma,\Gamma]=  (\ZZ/2\ZZ)^{2k+1}$, we also have $b_1(M)=0$.
\item\label{example.obstruc.para.spin.iv}
  Further obstructions come from special holonomy. If a closed spin manifold $M$ has a timelike initial data triple, and hence, also a metric with parallel spinor, then a finite cover of $M$ is isometric to a Riemannian product of manifolds with trivial holonomy, with holonomy $\SU(m)$ (Calabi-Yau manifold), with holonomy $\Sp(m)$ (hyper-Kähler manifold), $G_2$ and $\Spin(7)$. Thus, this covering has a nontrivial parallel $1$-forms (in the case of a trivial factor), a parallel $2$-form (in the case of an $\SU(m)$-factor or an $\Sp(m)$-factor, a parallel $3$-form (in the case of a $7$-dimensional $G_2$-factor, or a parallel $4$-form (in the case of an $8$-dimensional $\Spin(7)$-factor).
  In particular, all spheres $S^n$, $n\geq 2$ do not carry a timelike initial data triple.
  \item \label{example.obstruc.para.spin.v}
If $M$ is a closed connected spin manifold of dimension $n = 4k$ with $|\hat{A}(M)| > 2^{2k-1}$, then $M$ does not carry a Ricci-flat metric with parallel spinor.
Arguing by contradiction, we assume that there exists a metric $g$ with $\scal^g \geq 0$.
Then partially integrating the Schrödinger-Lichnerowicz formula yields that any harmonic spinor $\phi$ on $(M,g)$ is parallel: $0=\|\Dirac\phi\|_{L^2}^2 \geq \|\nabla \phi\|_{L^2}^2$.
This is true for any spinor bundle, but we restrict our attention to the complex irreducible one.
In this case, the Atiyah-Singer index theorem states that $\hat{A}(M) = \ind{\Dirac} = \dim \ker(\Dirac_{|\Sigma^+ M}) - \dim \ker(\Dirac_{|\Sigma^-M})$ and the ineqality above yields that either $\Gamma(\Sigma^+ M)$ or $\Gamma(\Sigma^- M)$ must contain more than $2^{2k-1}$ parallel spinors.
This is a contradiction, as the rank of the half spinor bundles is $2^{2k-1}$.
This yields a bunch of examples, as the $\hat{A}$-genus is additive under the direct sum operation.
For instance, if $M$ is the connected sum of two or more copies of a $K3$-surface with the same orientation, then we have $|\hat{A}(M)|\geq 4$, and thus the assumption is satisfied.
Note that as the $K3$-surface is simply connected, we again have $b_1(M) = 0$.
  \end{enumerate}
\end{examples}

%-----------------------------------------
\subsection{Examples of lightlike (generalized) initial data triples}\label{subsec.example-ll-gen-ini-d-t-gen}
%----------------------------------------------

In this and the following subsections we want to analyze lightlike initial data triples, and we will try to derive necessary conditions for their existence. For this the following modification of Definition~\ref{def.initial.data.triples} will be helpful.

\begin{definition}\label{def.gen.initial.data.triples}
  Let $M$ be a connected manifold whose universal covering $\witi M$ is spin.
  A \emph{generalized initial data triple} $(g,K,\phi)$ on $M$ consists of a Riemannian metric $g$, a symmetric $2$-tensor $K$ and a nowhere vanishing spinor $\phi\in \Gamma(\ol\Sigma \witi M)$ on the universal covering solving \eqref{gen.imag.killing} for the pullbacks of $g$ and $K$ to $\witi M$, with the condition that the Riemannian Dirac current $U_\phi$ is the pullback of a vector field on $M$.
 Again, we say that the triple is \emph{lightlike} (or \emph{timelike}), if the Lorentzian Dirac current is lightlike (or timelike) everywhere.
\end{definition}

Our goal is to obtain necessary conditions for a closed manifold $M$ to admit a lightlike generalized initial data triple. For a better understanding of our proof, it is helpful to have some examples of lightlike generalized initial data triples in mind.

\begin{example}\label{example.lightlike.idt.eins}
  %Obviously, one immediately asks for examples of manifold $B_1$ and $\maT$ satisfying the above conclusions.
  Let $(W,g_0)$ be an $n$-dimensional euclidean vector space, and let $V$ be a hyperplane with unit normal vector $N$, viewed as a parallel vector field on~$W$.  Let~$\bigcdot$ be the Clifford multiplication of a spinor bundle of $(W, g_0)$, given by a representation that admits a symmetric involution $e_0$, anti-commuting with Clifford multiplication by vectors in $W$. Thus, a further Clifford multiplication is given by $X\cdot\psi:= -ie_0 \cdot X\bigcdot \psi$.
  
There is a non-zero parallel spinor field $\phi_0$ with $N\bigcdot \phi_0=i\phi_0$. With the calculation $g(U_{\phi_0},X)=\<X\bigcdot \phi_0,N\bigcdot \phi_0\>=g(X,N)\,\|\phi_0\|^2$ we see $U_{\phi_0}=\|\phi_0\|^2 N$.
As a result $\phi_0$ solves \eqref{gen.imag.killing} with $K=0$ and thus $(g_0,0,\phi_0)$ is a lightlike initial data triple on $W$. Let $\Gamma$ be a lattice in $W$. Then the intial data triple $(g_0,0,\phi_0)$ factors to an initial data triple $(g,0,\phi)$ on $M\definedas W/\Gamma$.
Assume that $w_1,\ldots, w_n$ is a $\ZZ$-basis of $\Gamma$, $w_n\notin V$, and write
$w_i=v_i+s_iw_n$ with $v_i\in V$ and $s_i\in \RR$. If $s_1,\ldots,s_{n-1}$ are all rational,
then the image of $V$ in $M$ is a torus, otherwise it is a dense leaf of some (flat) foliation of codimension $1$. If $s_1,\ldots,s_{n-1},s_n=1$ are rationally linearly independent, then the leaves are simply connected and thus isometrically immersed euclidean spaces of dimension $n-1$.

In the rational case we may choose positive, smooth, $\Gamma$-invariant and $V$-invariant
functions $f_1,f_2:W\to \RR$, then $(f_1g,K,f_2\phi)$ is again a lightlike initial data triple of~$M$ for some suitable $K_0$, see  \cite{ammann.kroencke.mueller:p19}.  
\end{example}

\begin{example}\label{example.heisenberg.comb}
  Recall that for $x,y,z\in \RR$ the matrix $H_{x,y,z}$,  the Heisenberg group $\maH_3$  and the lattice $L_3$ were defined in Example~\ref{example.obstr.structured}~\eqref{example.obstruc.para.spin.ii}.
  For a real number $\alpha$ we define $s\colon\maH_3\to \RR$, $s(H_{x,y,z})= y+\alpha x$,  which is a Lie group homomorphism and whose level sets define a left-invariant codimension $1$ foliation of $\maH_3$, and this foliation descends to a foliation of $\Gamma\backslash \maH_3$ for any lattice $\Gamma$ in $\maH_3$, \eg $\Gamma=L_3$.
We want to construct lightlike generalized initial data triples on
$\HH_3=L_3\backslash \maH_3$.

At first we equip $\maH_3$ with a left-invariant metric $g$. The leaves of the foliation carry a flat metric, and their second fundamental form is left-invariant as  well, in particular it is parallel along the leaves. 
We will now show that there exist a symmetric 2-tensor $K$ and a spinor $\phi$ such that $(g,K,\phi)$ defines a lightlike initial data triple on $\maH_3$, the Riemannian Dirac current $U_\phi$ is proportional to $\grad(s)$ and $(g,K,\phi)$ descends to a generalized initial data triple on $\HH_3$.

As a technical interplay we also consider the lattice $\Gamma_\alpha\definedas\lgen H_{1,-\alpha,0},H_{0,1,0} \rgen$, \ie the group generated by  $H_{1,-\alpha,0}$ and $H_{0,1,0}$, which also includes
their commutator $H_{0,0,1}$; we have $\Gamma_0=L_3$.
As we have $s(H_{1,-\alpha,0})=0$ and $s(H_{0,1,0})=1$, the function $s$ defines a well-defined torus bundle $\bar s \colon \Gamma_\alpha\backslash \maH_3\to \RR/\ZZ$.
Let $Q^\alpha=\bar s^{-1}([0])=\lgen H_{1,-\alpha,0},H_{0,0,1} \rgen\backslash s^{-1}(0)$, which is a flat torus and a Lie group, and let $\iota: Q^\alpha\to \Gamma_\alpha\backslash \maH_3$ be the inclusion.\footnote{Note that $\bar s \colon \Gamma_\alpha\backslash \maH_3\to \RR/\ZZ$ is \textbf{not} principal bundle for the group $Q^\alpha$: it has a section and principal bundles with a section are trivial.} 
Let $N$ be one of the two unit normal fields of the foliation, and let $\Psi_t$ denote the flow along $N$. Then $s(\Psi_t(p))=s(p)+\beta t$ for some $\beta\neq 0$, and the action of the abelian Lie group $s^{-1}(0)$ preserves $s^{-1}(t)$.
Define $g_t:=\iota^*\Psi_t^*(g)$, where~$g$ also denotes the associated metric on  $\Gamma_\alpha\backslash \maH_3$.
Then $(\lgen H_{1,-\alpha,0},H_{0,0,1} \rgen\backslash\maH_3,g)$ is isometric to $(Q^\alpha\times \RR,g_t+dt^2)$, and we will identify them. The manifolds and the isometry are invariant under left multiplication by $s^{-1}(0)$. The derivative $\frac{d}{dt}|_{t=t_0}g_t$ is non-zero, invariant by the  $s^{-1}(0)$-action and thus parallel along $Q^\alpha\times\{t_0\}$ for the intrisic connection, in particular it is $g_{t_0}$-divergence free.

We choose spin structures, such that non-trivial parallel spinors exist on all tori. There is a parallel spinor $\phi_0$ on $Q^\alpha$ with $N\bigcdot\phi_0=i\phi_0$.
Then
\cite[Main Constructions 15 to 17]{ammann.kroencke.mueller:p19}  provides a unique extension of $\phi_0$ to a spinor $\phi$ on $Q^\alpha\times \RR$ with $\nabla_N\phi=0$. One checks $N\bigcdot \phi=i\phi$.
It was proven in \cite{ammann.kroencke.mueller:p19} that the spinor $\phi$ is  $Q^\alpha$-invariant, thus parallel along each leaf $Q^\alpha\times \{t_0\}$ with respect to the intrinsic connection. We obtain a lightlike initial data triple $(g,K,\phi)$ on $Q^\alpha\times\RR$ for some suitable $K$ with $U_\phi=\|\phi_0\|^2N$ and $u_\phi=\|\phi_0\|^2$.
By pullback we obtain a lightlike initial data triple on $\maH_3$, invariant under the action of $H_{x,-\alpha x,z}$
but we do not expect it to be invariant under $H_{0,y,0}$. It yields a lightlike \emph{generalized} initial data triple on $\HH_3$. The construction in \cite{ammann.kroencke.mueller:p19} is such that the left invariance of $g$ and the left invariance of the foliation imply the left invariance of $K$ -- with proper identifications $K$ coincides with the second fundamental form of the leaves. Thus the tensor $K$ descends to~$\HH_3$.
We expect -- although we have no proof -- that a suitable modification of the metrics $g_t$ in the class of flat metrics on tori yields non-generalized triples on $\HH_3$, \ie triples where the spinor is defined on $\HH_3$.

The level sets of $s$ give a codimension-$1$ foliation of $\HH_3$. The leaves are $2$-tori if $\alpha$ is rational. For irrational $\alpha$ all leaves are dense and diffeomorphic to $S^1\times\RR$.

In this example $\pi_1(\HH_3)$ is isomorphic to $\ZZ^2\rtimes \ZZ$, where
the action of $\ZZ$ on $\ZZ^2$ is given by the matrix
  $$A_0:=\begin{pmatrix}1 & 1 \\ 0&  1 \end{pmatrix}.$$

Let us note without proof that one may perturb these examples to lightlike generalized initial data triples, without the symmetries above, and such that the second fundamental form of the leaves is no longer parallel, both in the case of rational and irrational $\alpha$. However, the induced metric on the leaves remains flat. Modifications by functions $f_1$ and $f_2$ similar to the modifications in Example~\ref{example.lightlike.idt.eins} are also possible, provided $\alpha$ is rational; in this modifications $f_1$ and $f_2$ are defined on the space of leaves, which is diffeomorphic to $S^1$.
\end{example}

\begin{example}

  Consider $S_{k,A}:=(\ZZ^k\rtimes \ZZ)\backslash (\RR^k\times \RR)$, where we assume that the action of $1\in \ZZ$ in the semidirect product is given by a matrix $A\in \SL(k,\ZZ)$. Any curve $\RR\to \Mod_0(\ZZ^k\backslash \RR^k)$ $t\mapsto [g_t]$ in the moduli space of flat tori, with $A^*[g_{t+1}]=[g_t]$ defines a lightlike generalized intial data triple $(g,K,\phi)$ on $S_{k,A}$, where see \cite[Main Constructions 15--17]{ammann.kroencke.mueller:p19}. We choose the representatives $g_t$ for $[g_t]$ such that $\div^{g_t}\frac{d}{dt} g_t=0$. On  $((\ZZ^k\backslash \RR^k)\times \RR,g)$ we define the metric $g_t+dt^2$. A family of isometries $(\ZZ^k\backslash \RR^k,g_t)\to(\ZZ^k\backslash \RR^k,g_{t+1})$, homotopic to $A$ define a $\ZZ$-action on $(\ZZ^k\backslash \RR^k)\times\RR$ which may be turned into the deck transformation group for  $(\ZZ^k\backslash \RR^k)\times\RR\to S_{k,A}$. And again, the Dirac current $U_\phi$ of $\phi$ satisfies $U_\phi^\flat=c\, dt$ for some constant $c$. In the case $k=2$, $A=A_0$, and $g_t$ as in Example~\ref{example.heisenberg.comb}, these examples coincide with the case $\alpha=0$ constructed in Example~\ref{example.heisenberg.comb}. The modifications by periodic functions $f_1(t)$ and $f_2(t)$
  are also possible in this example, though already covered by  \cite{ammann.kroencke.mueller:p19}. The leaves, given by integration of $\ker U_\phi^\flat$ are the tori with constant $t$, in particular they are compact.

  Under some special conditions one may ``tilt'' these examples and deform them to  lightlike generalized initial data triples, such that the leaves are non-compact. Obviously $\pi_1(S_{k,A})=\ZZ^k\rtimes \ZZ$.

  Note that $S_{k,A}$ is the mapping torus of the diffeomorphism $\ZZ^k\backslash \RR^k\to\ZZ^k\backslash \RR^k$ given by $A$.
  One also may replace $\ZZ^k$ by a Bieberbach group $\Gamma$ acting $\RR^k$ such that  $\Gamma\backslash\RR^k$ is spin with parallel spinor, see \cite{pfaeffle:00} for examples, and for suitable spin diffeomorphisms. This version of the construction provides examples that contain $\ZZ^k\rtimes \ZZ$ as a finite index subgroup.
  \end{example} 

\begin{remark}\label{rem:further.methods}
  A further method to obtain new examples is as follows. We start with a lightlike initial data triple $(g,K,\phi)$ on some manifold $M$. We solve the associated Cauchy problem, see Subsection~\ref{subsec.Cauchy.par.spin}, and we obtain a globally hyperbolic Lorentzian spin manifold $(\ol M,\bar g)$, with a lightlike parallel spinor such that $M$ is a Cauchy surface for~$\ol M$ with the given~$(g,K,\phi)$ induced. Now, if~$M_0$ is another Cauchy surface for~$\ol M$, then the geometry and the parallel spinor on $(\ol M,\bar g)$ induce a lightlike  initial data triple on $M_0$ as well --  and by pulling back with a diffeomorphism -- a lightlike  initial data triple $(g_0,K_0,\phi_0)$ on $M$. In general, the diffeomorphism $M\to M_0$ cannot be chosen such that $(g_0,K_0,\phi_0)$ equals to $(g,K,\phi)$. One may break symmetries acting on $M$ by this construction.
  
The same construction works for lightlike generalized initial data triples with the obvious modifications.
\end{remark}  

\begin{summary*}
  We have seen that there are many examples of generalized initial data triples on closed manifolds. The fundamental group of all of them is virtually solvable of derived length at most $2$. We will show in the following subsections that this is indeed a necessary condition.  
\end{summary*}

%-----------------------------------------
\subsection{Lightlike initial data manifolds and other notation}\label{subsec.ll-gen-ini-d-t-gen}
%-----------------------------------------

\begin{definition}\label{def.initial.data.manifold}
  We say that an $n$-dimensional closed connected spin manifold $M$ is a \emph{lightlike initial data manifold of rank $\ell\in\{0,1,\ldots, n-1\}$} if there is a simply-connected closed $(n-\ell-1)$-dimensional manifold $P$, an oriented spin diffeomorphism $f\colon P\to P$, a path $[0,1]\ni t\mapsto h_t$ of structured Ricci-flat metrics on $P$ with $h_0=f^*h_1$, and  a matrix $A\in \SL(\ell,\ZZ)$ such that $M$ is spin diffeomorphic to  
  $$\quotientspace{P\times \bigl(\Torus[\ell]\bigr)\times [0,1]}{(p,x,0)\sim \bigl(f(p),Ax,1\bigr).}$$
\end{definition}
In the definition we assume that the torus $\Torus[\ell]$ carries the spin structure, such that the spin structure has a trivialization that is parallel for the euclidean metric. Note that every $\ell$-dimensional torus with a spin structure has a double cover that is spin-diffeomorphic to  $\Torus[\ell]$ with this spin structure. 
Recall that a Ricci-flat metric on $P$ is called structured if the universal cover of $P$ is spin and carries a non-trivial parallel spinor.
In the definition above, simply-connectedness of $P$ implies that a metric is structured Ricci-flat on $P$ if and only if it carries a non-trivial parallel spinor.

In particular, the fundamental group of any  lightlike initial data manifold $M$ of rank $r$ is of the form $\pi_1(M)=\ZZ^\ell\rtimes \ZZ$ and its first
Betti number satisfies $b_1(M)= 1+\dim \ker (A-1)$.

The following proposition is an immediate consequence of \cite[Main Constructions~15--17]{ammann.kroencke.mueller:p19}.

\begin{proposition}
Any lightlike initial data manifold carries a lightlike generalized initial data triple  $(g,K,\phi)$.
\end{proposition}

We conjecture that the converse of this statement is true up to a finite covering.

\begin{conjecture}\label{conj:WdecInv-new}
Let $M$ be an $n$-dimensional closed connected spin manifold, $(g, K,\phi)$ a  lightlike generalized initial data triple.
Then $M$ is finitely covered by a lightlike initial data manifold.
\end{conjecture}

We will explain how to prove that the conjecture holds in a special case.

In fact, suppose that $(g, K,\phi)$ is a  lightlike generalized initial data triple. By assumption the Riemannian Dirac current of $\phi$ is the pullback of some $U\in \Gamma(TM)$. Then $U^\flat$ is a closed $1$-form, and the kernel of $U^\flat$ integrates to a foliation $\maF$ of codimension~$1$. Let $\Phi_t: M\to  M$ be the flow of the  the vector field $U/\|U\|^2$. This flow maps leaves of $\maF$ diffeomorphically to other leaves, and it acts transitively on the space of leaves. In particular, either all leaves are  compact or all leaves are non-compact.
We will confirm the conjecture in the case of compact leaves, see Theorem~\ref{thm:WdecInv-new-cpct}. We expect that triples with  non-compact leaves can be deformed to triples with compact leaves, but a full proof is still missing.

Before,  we will discuss the case of compact and non-compact leaves separately,
we will introduce some joint notation. Let $\pi:\witi M\to M$ be the universal covering. Let $s:\witi M\to \RR$ be a smooth function  with $ds=\pi^*(U^\flat)$, and set $\witi Q_\sigma:=s^{-1}(\sigma)$.
Let $\witi\Phi_t:\witi M\to \witi M$ be the flow of the pullback of the vector field $U/\|U\|^2$. Obviously, it is a lift of $\Phi_t$ and satifies
$s(\witi\Phi_t(x))=s(x)+t$. We obtain a diffeomorphism $\witi Q_0\times\RR\to\witi  M$, $(x,t)\mapsto \witi\Phi_t(x)$. Thus $\witi M$ and $\witi Q_0$ are homotopy equivalent, in particular $\witi Q_0$ is simply-connected.

The leaves of $\maF$ are thus given by $Q_\sigma:=\pi(\witi Q_\sigma)$, $\sigma\in \RR$.
Let $\tilde g_\sigma$ and $g_\sigma$ be the induced metrics on $\witi Q_\sigma$ and $Q_\sigma$, respectively.
According to Leistner and Lischewski \cite[Theorem 4]{leistner.lischewski:19}
$g_\sigma$ is a structured Ricci-flat metric on~$Q_\sigma$. As the $1$-form $\pi^*(U^\flat)$ is invariant under the deck transformation group of $\witi M\to M$, there is a group homomorphism $S\colon \pi_1(M)\to \RR$, satisfying
$s(\gamma\cdot x)=S(\gamma)+s(x)$ for any $\gamma\in \pi_1(M)$ and $x\in \witi M$, where $\pi_1(M)$ is identified with the deck transformation group by fixing a base point in $\witi M$.
The image of $S$ is never $\{0\}$, as this would imply that $s$ is a pullback from $M$ and thus has a maximum.
If the image of $S$ is a discrete set, say $\sigma_0\ZZ$, then $s$ factors to a well-defined map $M\to \RR/\sigma_0\ZZ$, and thus the leaves are compact.
Conversely, if a leaf $Q_\sigma$ is compact, then for small $\ep>0$ the map $Q_\sigma\times (-\ep,\ep)\to M$, $(x,t)\mapsto \Phi_t(x)$ is injective, and thus its lift  $\witi Q_\sigma\times (-\ep,\ep)\to \witi M$, $(x,t)\mapsto \witi \Phi_t(x)$ will not hit the preimage of $Q_\sigma$ other than in $\witi Q_\sigma$. Hence the image of $S$ is disjoint from $(-\ep,\ep)\setminus \{0\}$. We have thus seen that the image of $S$ is discrete if and only if the leaves are compact.

%---------------------------------------
\subsection{Lightlike generalized initial data triples with compact leaves}\label{subsec.ll-gen-ini-d-t-cpct}
%---------------------------------------

We now consider the case of a compact leaf.

\begin{theorem}\label{thm:WdecInv-new-cpct}
  Let $M$ be an $n$-dimensional closed connected spin manifold with a lightlike generalized initial dats triple $(g, K,\phi)$.
 Assume that the leaves of the associated foliation $\maF$ -- defined at the end of Subsection~\ref{subsec.ll-gen-ini-d-t-gen}  --  are compact. Then:
\begin{enumerate}[{\rm (1)}]
    \item\label{item-one} There is a spin diffeomorphism $F\colon Q_0\to Q_0$, and some $\sigma_0>0$  such that $M$ is spin diffeomorphic to  
 $$\quotientspace{Q_0\times [0,\sigma_0]}{(q,0)\sim \bigl(F(q),\sigma_0\bigr).}$$
 The pulled-back metrics $\hat g_\sigma$ on $Q_0 \times \{\sigma\}$ define a path of structured Ricci-flat metrics on $Q_0$ with $F^*\hat g_{\sigma_0}=\hat g_0$.
In particular, we have $\pi_1(M)\cong\pi_1(Q_0)\rtimes \ZZ$, where the action of $\ZZ$ on $\pi_1(Q_0)$ is defined by 
the map $\pi_1(F):\pi_1(Q_0)\to \pi_1(Q_0)$.
  \item\label{item-two} $M$ is finitely covered by a lightlike initial data manifold.
  \item\label{item-three} If $(g, K,\phi)$ is a  lightlike initial dats triple, \ie if additionally $\phi$ is a spinor on~$M$, then the metrics $\hat g_\sigma$ on $Q_0$ carry a non-zero parallel spinor.
\end{enumerate}
\end{theorem}

\begin{proof}
  In case that the leaves are compact, then the image of $S$ is a discrete subgroup of $\RR$, and let $\sigma_0>0$ be the positive real number with $\image S=\sigma_0\ZZ$. The map $Q_0\times [0,\sigma_0]\to M$, $(x,t)\mapsto \Phi_t(x)$ yields a spin diffeomorphism as claimed in \eqref{item-one}
-- yet still without the property that $F = \Phi_{\sigma_0}$ preserves the basepoint (so that $\pi_1(F)$ is well-defined). This can be fixed by considering $(x, t) \mapsto \Phi_t \circ \Psi_t^{-1}(x)$ instead, for a suitable family of spin diffeomorphisms $\Psi_t$ to be defined in the proof of \eqref{item-two} below.
We now apply
\cite[Theorem~4]{leistner.lischewski:19} which is possible as Equation (1.8) in  \cite{leistner.lischewski:19}  -- which essentially agrees with our Equation \eqref{eq:SpinConstr2} -- is satisfied, see Corollary~\ref{cor:contraints.valid}. It follows from there that
$g_\sigma$ and thus $\hat g_\sigma$ is structured Ricci-flat, and with the stronger assumption in \eqref{item-three} these metrics carry a non-zero parallel spinor.
Thus \eqref{item-one} and \eqref{item-three} are proven.

It remains to prove \eqref{item-two}. As $(Q_0,g_0)$ is structured Ricci-flat and compact
its universal covering $(\witi Q_0,\tilde g_0)$
is isometric to $(P,h_0)\times (\RR^\ell,\geucl)$,
where $(P,h_0)$ is a simply-connected compact structured Ricci-flat manifold with a finite isometry group $\Isom(P,h_0)$. We will fix a base-point $p$ of $\wihat Q_0$, mostly suppressed in notation, in order to identify deck transformations of $\wihat Q_0$ and $\wihat M$ with elements in $\pi_1(Q_0)$ and $\pi_1(M)$. We determine a smooth family of spin diffeomorpisms $\Psi_t:Q_0\to Q_0$, $\Psi_0=\id_{Q_0}$, with lift  $\witi\Psi_t:\witi Q_0\to \witi Q_0$, $\witi\Psi_0=\id_{\witi Q_0}$,  such that $\witi\Psi_{\sigma_0}(p)=\witi\Phi_{\sigma_0}(p)$.
We choose $F\definedas \Phi_{\sigma_0} \circ \left(\Psi_{\sigma_0}\right)^{-1}$
whose lift  $\witi F\definedas \witi\Phi_{\sigma_0} \circ \left(\witi\Psi_{\sigma_0}\right)^{-1}$ is base point preserving. Then the induced map $\pi_1(F)$ is an automorphism of $\pi_1(Q_0)$. Note that $\witi F$ of $F$ is a deck transformation and thus gives rise to some $\beta_F\in \pi_1(M)$.
We obtain $\pi_1(M)=\pi_1(Q_0)\rtimes \ZZ$ where $\beta_F$ generates $\ZZ$.

We have $\pi_1(Q_0)\subset \Isom(\witi Q_0)=\Isom(P,h_0)\times (\RR^\ell\rtimes \O(r))$
as isometries of $\witi Q_0$ have to map the flat factor to the flat factor and the factor with the nowhere flat metric $h_0$ to itself as well. The projection  $\Pi_2(\pi_1(Q_0))\in\RR^\ell\rtimes \O(\ell)$ of $\pi_1(Q_0)$ acts discretely and cocompactly on  $\RR^\ell$, and thus it is a
 cristallographic group. 
 Hence,  $\Pi_2(\pi_1(Q_0))$ contains a (unique) maximal free abelian group $A_0$ of rank $\ell$ of finite index whose deck transformations act isometrically on $\RR^\ell$ by translations, and we know that there is a short exact sequence $1\to A_0\to \Pi_2(\pi_1(Q_0))\to T\to 1$ for some finite group~$T$. The translations in $A_0$ actually come with a lift to the spin structure; let $A_1$ be the subgroup in $A_0$ of those elements that have a trivial lift; it either is of index~$1$ or~$2$. We define $\Gamma_0:=\Pi_2^{-1}(A_1)$, which is a normal subgroup of finite index in $\pi_1(Q_0)$.
 As $\Isom(P,h_0)$ is finite, there is a finite index subgroup $\Gamma_1$ in $\Gamma_0$, normal in $\pi_1(Q_0)$,
 such that $\Gamma_1$ acts trivially on $P$.

Let $\iota$ be the  (finite) index of $\Gamma_1$ in $\pi_1(Q_0)$.
We do not know whether $\pi_1(F)$ preserves the subgroup $\Gamma_1$.

Abelian subgroups of rank $\ell$ are contained in $\Isom(P,h_0)\times A_0$, and thus there are only finitely many.
Thus there are finitely many free abelian subgroups of index $\iota$ in $\pi_1(Q_0)$.
 The group $\pi_1(F)$ acts on set of free abelian subgroups of index $\iota$.

Thus there is a number $r>0$ with $\pi_1(F)^r(\Gamma_1)=\Gamma_1$. We set
  $$\Gamma_2\definedas \Gamma_1\cap\pi_1(F)^{-1} (\Gamma_1)\cap\cdots\cap \pi_1(F)^{-(r-1)} (\Gamma_1),$$
which is of finite index in $\Gamma_1$ and thus a $\pi_1(F)$-invariant finite index subgroup of $\pi_1(Q_0)$. As a consequence the semi-direct
product $\Gamma_2\rtimes \ZZ$ given by this $\pi_1(F)$-action is a finite index subgroup of $\pi_1(M)$.

Now $M'\definedas\witi M/(\Gamma_2\rtimes \ZZ)$ is a finite covering of $M$ satisfying
all properties of $M$ mentioned so far, but with the additional
properties that its leaf $Q'_0\definedas \witi Q_0/\Gamma_2$ is isometric to the product of $(P_0,h_0)$ and a flat torus of dimension $\ell$, and that the leaf carries a non-trivial parallel spinor. Pulling back with the corresponding flow $\Phi_t'$ we obtain a family of structured Ricci-flat metrics $g_\sigma'$
on $Q'_0$, and  because of results from \cite{ammann.kroencke.weiss.witt:19}
there are spin diffeomorphisms $\Psi_\sigma:P\times (\Torus[\ell])\to Q'_0$,
such that $\Psi_\sigma^*g_\sigma'$ is the product metric on $(P,h_\sigma)\times (\Torus[\ell],\<\platz,\platz\>_\sigma)$, where $\<\platz,\platz\>_\sigma$ is a smooth family of scalar product on the vector space $\RR^\ell$.

The manifold $M'$ is obtained as
$$\quotientspace{P\times (\Torus[\ell])\times [0,\sigma_0]}{(q,0)\sim \bigl(F'(q)),\sigma_0\bigr)}$$
for some spin diffeomorphism $F'\colon P\times (\Torus[\ell])\to P\times (\Torus[\ell])$ with $F'^*\left(h_{\sigma_0}\times \<\platz,\platz\>_{\sigma_0}\right)=h_0\times \<\platz,\platz\>_0)$ -- not necessarily preserving the basepoint.
As the metrics $h_\sigma$ are nowhere flat, this implies that $F'$ is of the form
$f\times A$  where $f$ is a spin diffeomorphism of $P$ with $f^*h_{\sigma_0}=h_0$ and where $A\in \SL(\ell,\ZZ)$ is a linear diffeomorphism of $\Torus[\ell]$,  as requested in Definition~\ref{def.initial.data.manifold}. Statement \eqref{item-two} then follows.
\end{proof}

In particular, the fundamental group of a closed manifold carrying a lightlike generalized initial data triple with compact leaves is virtually solvable of derived length at most $2$.

%--------------------------------------
\subsection{More results for all lightlike generalized initial data triples}\label{subsec.ll-gen-ini-d-t-compact-noncpct}
%--------------------------------------

In this subsection we will analyze closed manifolds $M$ with
a lightlike generalized initial data triple with non-compact leaves. We will
show that also in this case the fundamental is virtually solvable of derived length $2$. The situation is in some sense more rigid, \eg the denseness of the leaves implies that all fibers are isometric.

We assume further -- but unfortunately we cannot yet prove it -- that such triples with non-compact leaves can be deformed to ones with compact leaves, and thus Properties \eqref{item-one} and \eqref{item-two} of Theorem~\ref{thm:WdecInv-new-cpct} would follow as well.

Before specializing to the non-compact case, we prove a Proposition that holds both for compact and non-compact leaves.
\begin{proposition}\label{prop:Qtilde.struc}
  We assume that $\maF=(Q_\sigma)_{\sigma\in \RR}$ is a foliation of a closed connected manifold~$M$ obtained from a generalized initial data triple with induced Riemannian metrics~$g_\sigma$ and universal coverings $\witi Q_\sigma$, as defined at the end of Subsection~\ref{subsec.ll-gen-ini-d-t-gen}.
  Then, there is a number $\ell\in \{0,1,\ldots,\dim M-1\}$ such that for any $\sigma\in \RR$ the manifold
  $(\witi Q_\sigma,\ti g_\sigma)$ is isometric to $(P,h_\sigma)\times (\RR^\ell,\geucl)$ where $(P,h_\sigma)$ is a simply-connected compact structured Ricci-flat manifold
  with a finite isometry group $\Isom(P,h_\sigma)$.
Furthermore, there is a numbers $k\in \NN_0$, $k\leq \ell$ such that for any $\sigma\in \RR$ the manifold
$(Q_\sigma,g_\sigma)$ is isometric to $(R,\hat h_\sigma)\times (\RR^k,\geucl)$ with some compact spin manifold $R$.
%where $\<\platz,\platz\>_\sigma$ is a family of scalar products on the vector space $\RR^k$.
The Riemannian manifold $(R,\hat h_\sigma)$ is finitely covered by the Riemannian product of  $(P,h_\sigma)$ and some flat torus $(L_\sigma\backslash \RR^{\ell-k}, \geucl)$, where $L_\sigma$ is a lattice in $\RR^{\ell-k}$. If the leaves of $\maF$ are non-compact, then we may arrange the isometries such that $h_\sigma$, $\hat h_\sigma$ and $L_\sigma$ do not depend on $\sigma$.
\end{proposition}

\begin{proof}
In the case of compact leaves, the proposition immediately follows from
Theorem~\ref{thm:WdecInv-new-cpct} with $k=0$. We thus assume that the leaves are non-compact. As a result every leaf is dense.  Again, we may
apply \cite[Theorem 4]{leistner.lischewski:19} as \cite[(1.8)]{leistner.lischewski:19}  resp.\ \eqref{eq:SpinConstr2} will be satisfied, see Corollary~\ref{cor:contraints.valid}. Thus the induced metric $g_\sigma$ on any leaf is a structured Ricci-flat metric.

For every $\sigma\in \RR$ determine the maximal number $\ell_\sigma$ such that  $(\witi Q_\sigma,\ti g_\sigma)$ is isometric to the Riemannian product
$(P_\sigma,h_\sigma)\times (\RR^{\ell_\sigma},\geucl)$ where $(P_\sigma,h_\sigma)$ is a complete simply connected Ricci-flat manifold. Because of the Cheeger--Gromoll splitting theorem  $(P_\sigma,h_\sigma)$ does not contain a line. Recall that a line is defined as a distance minimizing geodesic, parametrized by arclength and defined on $\RR$.  If $P_\sigma$ is not a point, $h_\sigma$ must not be flat, and as Ricci-flat manifolds are analytic, it is non-flat on every non-empty open subset. Thus for every number $\rho>0$ and every point $p\in \witi Q_\alpha$, the number $\ell_\sigma$ coincides with the maximal number such that there is an isometry of an open neighborhood of $p$ in $\witi Q_\alpha$ to some Riemannian product $(P',h')\times \bigl(B_\rho(p,\RR^{\ell_\sigma}),\geucl\bigr)$, mapping $p$ to some $(p',0)$.

By density of the leaf $\witi Q_\sigma$, a flat disc $\bigl(B_\rho(p,\RR^{\ell_\sigma}),\geucl\bigr)$ splits off in any point in \emph{any} leaf, and thus the number $\ell=\ell_\sigma$ does not depend on $\sigma$. It follows that $\dim P_\sigma=n-\ell-1$. 
We now argue that $P_\sigma$ is compact for any $\sigma$. So let us assume that $P_\tau$ is not compact for some $\tau$ and thus of infinite diameter.
% folgendes nicht loeschen!
%Using the diffeomorphism $\witi\Phi_\sigma$ the non-compactness of $P_0$ follows -- as the diffeomorphisms will not perserve teh product decomposition, a convenient way is to consider $H_{n-\ell-1}(\witi Q_\sigma,\ZZ)=H_{n-\ell-1}(P_\sigma,\ZZ)$ which vanishes if and only if  $P_\sigma$ is non-compact. \cbernd{Können wir uns dieses Argument sparen?}
%% Reference: Massey, A basic course in algebraich topology, Springer Chap. XIV, Exercse 4.1
%%Thus let us assume that $P_0$ is non-compact and thus of infinite diameter.
Then for all $L>0$ we have distance minimizing
geodesics $\gamma_L:[-L,L]\to P_\tau$, parametrized by arc-length.
The image of $(\dot\gamma_L(0),0)\in T_{(\gamma_L(0),0)}(P_\tau\times \RR^\ell)\subset T\witi Q_\tau$ in $TM$ is a unit length vector~$v_L$, tangent to the foliation $(Q_\sigma)_{\sigma\in \RR}$. Thus there is a sequence of $L_i\to \infty$, such that $v_{L_i}$ converges to some unit vector~$v$ for $i\to \infty$, and we have $v\in TQ_{\hat \sigma}$ for some $\hat \sigma\in \RR$.

Consider the geodesic $\gamma_\infty:\RR\to Q_{\hat \sigma}$ in $(Q_{\hat \sigma},g_{\hat \sigma})$ with $\dot\gamma_\infty(0)=v$.
Due to a standard convergence argument, it lifts to a line in $\witi Q_{\hat\sigma}$, and this line it is orthogonal to the $\RR^\ell$-component.  Applying the Cheeger splitting theorem, we conclude that  $(P_{\hat \sigma},h_{\hat \sigma})$ is a Riemannian product with a euclidean line. This is in contradiction to the maximality of $\ell$.

%For the first part of the Proposition it only remains to prove that $P_\sigma$ is diffeomorphic to $P_0$ \cbernd{später besser formuleiren}.

In general the diffeomorphisms $\Phi_t$ are no isometries, neither isometries of $M$ nor isometries from $Q_\alpha$ to $Q_{\alpha+t}$, and their lifts $\witi\Phi_t$ do not necessarily preserve the product structure.\footnote{We are able to construct examples, where it does not.} However the following lemma provides suitable isometries on the level of the universal coverings.

\begin{lemma}\label{lem.modified.diffeos}With the assumptions of the proposition and additionally assuming non-compact leaves, we obtain that
 for every $\sigma\in \RR$, $(P_\sigma,h_\sigma)$ is isometric to $(P_0,h_0)=:(P,h)$.
 There is a $1$-parameter group of isometries\/ $\Xi_\sigma:(P,h)\times (\RR^\ell,\geucl)\to (\witi Q_\sigma,\witi g_\sigma)$, depending smoothly on $\sigma\in \RR$.
\end{lemma}
The proof of the lemma will be given in the Appendix~\ref{appendix.lemma.isom}.

The group $\Gamma:=\ker S$ acts on $\witi Q_\sigma$, and we have  $Q_\sigma=\Gamma\backslash\witi Q_\sigma$.
Using the isometry $\Xi_\sigma$ we obtain an isometric, free and properly discontinuous action of $\Gamma$ on $P\times \RR^\ell$ which depends smoothly on $\sigma$. We want to split $\RR^{\ell}=V_1\oplus V_2$ orthogonally as affine euclidean spaces into a summand $V_1$ on which $\Gamma$ acts trivially, and into a summand $V_2$ with a cocompact action of $\Gamma$ on $P\times V_2$. This splitting may depend on $\sigma$.
We choose a maximal number $k=k(\sigma)$ such that $(Q_\sigma,g_\sigma)$ is isometric to the Riemannian product $(R_\sigma,\hat h_\sigma)\times (\RR^k,\geucl)$. We proceed similarly as above, and we will show first  that  $k$ does not depend on $\sigma$ and then that $R_\sigma$ is a compact manifold, whose diffeomorphism type  does not depend on $\sigma$ -- we then have $V_1\cong \RR^k$ and $R_\sigma=\Gamma\backslash (P\times V_2)$.

The number $k(\sigma)$ can be characterized as follows. Given $x\in Q_\sigma$ we consider maps $\iota_{x,\kappa}:\RR^\kappa\to Q_\sigma$, $0\mapsto x$, such that they are isometric embeddings in the sense of metric spaces. This implies that  $\iota_{x,\kappa}$ maps affine lines in $\RR^\kappa$ to lines (\ie length minimzing geodesics) in $(Q_\sigma,g_\sigma)$. If such a map $\iota_{x,\kappa}$ exists, then the Cheeger splitting theorem allows to split off a factor $\RR^\kappa$, and thus $\kappa\leq k(\sigma)$. Conversely, it is obvious, that a map $\iota_{x,\kappa}$ exists through $x$ with $\kappa=k(\sigma)$. Thus $k(\sigma)$ is the maximal such $\kappa$. We also see the map $\iota_{x,k}$ through any $x\in M$ is unique up to rotations.

Now for any $\tau\in \RR$ and $y\in Q_\tau$
consider a sequence of $x_i\in Q_\sigma$ with $x_i\to y$ in $M$. Choose $\iota_{x_i,k(\sigma)}$ as above, and after passing to a subsequence, they will converge to an isometric embedding $\iota_{y,k(\sigma)}$ through $y$, and this implies $k(\tau)\geq k(\sigma)$ for all $\sigma,\tau\in \RR$, \ie $k$ is constant.

If there is a $\sigma \in \RR$ with $R_\sigma$  non-compact, then we may obtain a line in the limit orthogonal to the $\RR^k$-factor, similar to before and thus a contradiction. The factors $R_\sigma$ may be identied as the leaves of a foliation $\FF_R$ given by the images of the orthogonal complements to the images of map $\iota_{x,k}$. All leaves are compact, and by shifting in the $\RR^k$-direction or by using the flow $\Phi_t$, we obtain isomorphisms of the fundamental group of the leaves. Standard  techniques
about foliations \cite{lawson:77}, \cite[Chap. II]{moore.schochet:06}
imply that all leaves are diffeomorphic and that this is foliation is an $R$-manifold bundle.
%As $\pi_1(R_\sigma)=\Gamma$ does not depend on $\sigma$ and as all leaves are compact, they are all diffeomorphic. 
%\cjonathan{Ist das ein Standardschluss in der Theorie von Blätterungen? Gibt es eine Referenz dazu?}

The orthogonal decomposition $\RR^\ell=V_1\oplus V_2$, with $V_1\cong \RR^k$ a priori depends $\sigma$. However as the action is smooth in $\sigma$ the decomposition is smooth, and thus $\Xi_\sigma$ can be rearranged such that $\RR^\ell=\RR^{\ell-k}\oplus \RR^k$ is the standard decomposition.

As $\gamma\in \Gamma$ acts isometrically on $P\times \RR^{\ell-k}$ it has to preserve its product structure, \ie $\gamma$ acts as
$\rho^\sigma(\gamma)=\rho_P(\gamma)\times \rho_2^\sigma(\gamma)$, where $\rho_P(\gamma)\in \Isom(P,h)$ and $ \rho_2^\sigma(\gamma)\in\Isomaff(V_2)$.
Due to the finiteness of $\Isom(P,h)$, $\rho_P(\gamma)$ is constant in $\sigma$, 
the subgroup $\{\gamma\in \Gamma\mid \rho_P(\gamma)=\id\}$ has finite index, and it is a cristallographic group which acts freely, properly discontinuously and cocompactly on $\RR^{\ell-k}$, and thus its maximal abelian subgroup $\Gamma_1$ is isomorphic to $\ZZ^{\ell-k}$, acts by translations (depending on $\sigma$) on $\RR^{\ell-k}$ and is of finite index in $\Gamma$. Obviously for different values of $\sigma$ the actions of $\Gamma_1$ are affinely conjugated to each other and each such conjugation extends to a conjugation of the actions of the full groups $\Gamma$. Soon, we will show that these conjugations may be chosen in $\O(\ell-k)$ and thus also $(R_\sigma,\hat h_\sigma)$ are isometric, and we write $(R,h):=(R_0,h_0)$. Furthermore this implies that we may choose the maps $\Xi_\sigma:P\times \RR^{\ell-k}\times \RR^k\to \witi Q_\sigma$ equivariantly for a fixed cristallographic isometric action of $\Gamma$ on $P\times \RR^{\ell-k}$. Thus, $\Xi_\sigma$ is the lift of a family of isometries $\xi_\sigma: (R,\hat h)\times (\RR^k,\geucl)\to (Q_\sigma,g_\sigma)$ depending smoothly on $\sigma$.

It remains to show that the translations are conjugated in $\O(\ell-k)$. Note that $(R_\sigma,\hat h_\sigma)$ is isometric to $\rho^\sigma(\Gamma)\backslash (P\times \RR^{\ell-k})$, and thus its minimal covering by a product of $P$ with a torus is isometric to $\rho^\sigma(\Gamma_1)\backslash (P\times \RR^{\ell-k})$. 
One the other hand, along $Q_\sigma$ all the $R$-factors are isometric. For $t\in \image S$ we have  $Q_\sigma=Q_{\sigma+t}$, and thus $(R_{\sigma+t},\hat h_{\sigma+t})$ is isometric to  $(R_\sigma,\hat h_\sigma)$,
and hence their minimal coverings by  a product of $P$ and a torus as well. We have derived  that  $\rho_2^{\sigma+t}(\Gamma_1)\backslash \RR^{\ell-k}$ and $\rho_2^\sigma(\Gamma_1)\backslash \RR^{\ell-k}$ are isometric. Thus there are matrices $B_{\sigma,t}\in \O(\ell-k)$, $t\in \image S$, $\sigma\in \RR$ with 
$$\rho_2^{\sigma+t}(\Gamma_1)=B_{\sigma,t}\cdot \rho_2^\sigma(\Gamma_1) \cdot (B_{\sigma,t})^{-1}.$$

On the other hand, $\rho_2^\sigma$ depends smoothly on $\sigma$ and thus there are matrices
$A_\sigma\in \GL(\ell-k,\RR)$, $A_0=1$, smooth in $\sigma$, such that
$$\rho_2^\sigma(\gamma)=A_\sigma\cdot \rho_2^0(\gamma) \cdot (A_\sigma)^{-1} \quad\forall\gamma\in\Gamma_1.$$

In combination we see that conjugation with $A_\sigma\cdot (A_{\sigma+t})^{-1}\cdot B_{\sigma,t}$ preserves the lattice $\rho_2^\sigma(\Gamma_1)$.
Now every lattice $L$ has the property: there is an $\epsilon(L)>0$ such that: if $F$ is a matrix with $FLF^{-1}=L$ and $|\<F(X),F(Y)\>-\<X,Y\>|\leq \epsilon(L)\|X\|\,\|Y\|$, then $F$ is orthogonal. Thus, there is an $\epsilon_\sigma>0$, such that for $|t|<\epsilon_\sigma$ we have 
$A_\sigma\cdot (A_{\sigma+t})^{-1}\cdot B_{\sigma,t}\in \O(\ell-k)$, thus 
$A_\sigma\cdot (A_{\sigma+t})^{-1}\in \O(\ell-k)$.
By a compactness argument we obtain $A_\sigma\in \O(\ell-k)$ for all $\sigma\in \RR$.

We finally define the $(\ell-k)$-dimensional Riemannian torus $T:=\Gamma_1\backslash \RR^{\ell-k}$ by the action of one of the isometric $\rho_2^\sigma$ actions.
As $\Gamma_1$ has finite index in $\Gamma$, we get a finite covering $P\times T\to R=\Gamma\backslash (P\times \RR^{\ell-k})$.
\end{proof}

\begin{remarks}
  \begin{enumerate}[{\rm (1)}]
  \item   If $t\in \image S$, then $Q_\sigma=Q_{\sigma+t}$. Thus $\xi_\sigma$ and $\xi_{\sigma+t}$
    are isometries between the same manifolds, however they do not coincide.
  \item In general not every automorphism of $\Gamma$ will preserve $\Gamma_1$, \eg we can construct examples with $\Gamma\cong\ZZ^k\times (\ZZ/2\ZZ)$, in which case there are $2^k$ maximal free abelian subgroups, and $\Aut(\Gamma)$ acts transitively on the set of such subgroups. However every automorphism of $\Gamma$ which preserves the product structure $\Isom(P,h)\times \Isomaff(\RR^{\ell-k})$ will preserve $\Gamma_1$ as it is the unique maximal free abelian subgroup in the Bieberbach group $\Gamma\cap(\{1\}\times \Isomaff(\RR^{\ell-k}))$.
\item The isometry group of $(R,\hat h_\sigma)$ is compact and fits into a short exact sequence
         \begin{align} \label{eq:G0Ses} 0\to \ZZ^m\backslash\RR^m\to \Isom(R,\hat h_\sigma)\to \wihat\Gamma_\sigma\to 1, \end{align}
where $\wihat\Gamma_\sigma$ is a finite group and $m\leq \ell-k$ is the dimension of the space of Killing vector fields on $(R, \hat h_\sigma)$.
\item Later on, we will see that, if $k$ is defined as above, then the rank of the free abelian group $\image(S)$ equals to $k+1$.  
\end{enumerate}
\end{remarks}

After fixing a base point, deck transformations may be identified with elements in $\pi_1(M)$. Using the notation of the proof of the proposition we
obtain a short exact sequence 
$$0\lto \Gamma \lto \pi_1(M) \xrightarrow{\;\, S\;\,}\image S\lto 0.$$
One can show with group theoretical methods, that this implies that $\pi_1(M)$ is virtually solvable.%
\ifdraft

\textbf{Draftversion:} For a proof we refer to Appendix~\ref{app.group}.
\else
We will omit a proof as we will obtain a stronger version of this statement later in the non-compact case, and as it is already shown in the compact case.
\fi

%--------------------------------------
\subsection{Further results in the case of non-compact leaves}\label{subsec.ll-gen-ini-d-t-noncpct}
%--------------------------------------

We want add some further structural results which only hold in the case of non-compact (i.e. dense) leaves.
We write $\hat h$ instead of $\hat h_\sigma$, $h$ instead of $h_\sigma$, and $L_\sigma=L$ from now on.

We have seen that the maps $\xi_\sigma$ constructed above provide an isometry from the product  $(R,\hat h)\times \RR^k$ onto the leaf $Q_\sigma$ of $\maF$, and furthermore it is easy to see that
  $$(R,\hat h)\times \RR^k\times \RR\to M,\quad (x,y,\sigma)\mapsto \xi_\sigma(x,y)$$
is a smooth covering. Thus the family of  submanifolds of $(M,g)$ that are isometric to $(R,\hat h)$ and that are contained in $Q_\sigma$ for some $\sigma$ form a smooth foliation $\FF_R$ of~$M$, and obviously this is a manifold bundle over some closed manifold $B$. The manifold $B$ may be defined as the space of leaves of $\FF_R$; any leaf of $\FF_R$ is of the form $\xi_\sigma(R\times\{y\})$ for some $(y,\sigma)\in \RR^k\times \RR$, and the differential structure is such that
\begin{equation}\label{eq:cov.B}
  \RR^k\times \RR\to B,\quad (y,\sigma)\mapsto \xi_\si(R\times\{y\})
\end{equation}
is a covering by local diffeomorphisms.

We also study this from the perspective of principal bundles.
Let $\maT$ be the space of isometric embeddings $(R,\hat h) \to (M, g)$ with the property that their image lies in some leaf $Q_\sigma$.
(Note that this implies that its image has an orthogonal complement in this $Q_\sigma$ that is isometric to $(\RR^k,\geucl)$.)
The group $G\definedas \Isom(R,\hat h)$ acts on $\maT$ by precomposition, and one easily sees that this action turns $\maT$ into a principal bundle over the closed manifold $B$ discussed above. The evaluation map $\maT\times R\to M$ yields a diffeomorphism from $\maT\times_G R$ to $M$.

Let $G_1$ be the minimal open subgroup of $G$ such that the principal bundle $\maT$ can be reduced to $G_1$, in other words: $\dim G=\dim G_1$, there is a $G_1$-principal bundle $\maT_1\to B$ such that $\maT=\maT_1\times_{G_1}G$, and $G_1$ is minimal among subgroups with these properties. We have $M=\maT_1\times_{G_1} R$.
As $G_1$ has a finite number of connected components,
and as its identity component is $G_0\cong \ZZ^m\backslash\RR^m$ (cf.~\eqref{eq:G0Ses}), we see that $B'\definedas \maT_1/G_0$ is a finite covering of $B$. The minimality of $G_1$ implies that
$B'$ is connected, and $\maT_1\to B'$ is a principal bundle for the group~$G_0$.
Furthermore $M':=\maT_1\times_{G_0} R$ is a connected finite covering of~$M$.
Similarly as we defined~$S$ for~$M$, we can define~$S'$ for~$M'$, and  $\image S'$ is then a finite index subgroup of $\image S$.
The universal coverings of $B$ and $B'$ obviously coincide and are diffeomorphic to $\RR^{k+1}$, see \eqref{eq:cov.B}. We will write~$\witi B$ for this covering.
The group $\image S'$ is the deck transformation group of $\witi B\to B'$, and thus it acts cocompactly, freely and properly discontinuously on $\witi B\cong \RR^{k+1}$. Thus,~$B'$ is a closed manifold with fundamental group $\ZZ^r$, with $r=\rank \image  S' = \rank \image S$, with a contractible universal covering.
Hence,~$B'$ is a classifying space for $\ZZ^r$, and thus it is homotopy equivalent to the $r$-dimensional torus. It follows $r=k+1$.
(Note that in the case $k\geq 4$ the Borel conjecture -- solved for $\ZZ^{k+1}$ -- and in the case $k=2$ the solution of Thurston's geometrization conjecture imply that $B'$ is homeomorphic to a $(k+1)$-dimensional torus, the case $k\leq 1$ being trivial, and we conjecture that for any $k$ the manifold $B'$ is diffeomorphic to such a torus.)

The group $G_0$ has a finite connected covering $\wihat G_0$ (say, with $a$ sheets) for which the action lifts along $P\times (L\backslash\RR^{\ell-k}) \to R$ to a free action.
\footnote{As the following theorem should be a statement about spin manifolds, we add here: if the induced spin structure on $L\backslash \RR^m$ does not amit a parallel spinor, then we pass to a double cover where a parallel spinor exists.} We will also assume $L\backslash \RR^{\ell-k}$ to be equipped with the euclidean metric.
  Obviously $\wihat G_0\cong \ZZ^m\backslash \RR^m$ acts trivially on $P$ as the isometry group of $P$ is discrete.
As $B'$ is homotopy equivalent to a torus we can pass to the finite covering $B'':=(a\ZZ^{k+1})\backslash\witi B$, and then the structure group of $\maT_1$ reduces to $\wihat G_0$, \ie there is a  $\wihat G_0$-principal bundle $\maT''\to B''$ such that $\maT''\times_{\wihat G_0} G_0$ is isomorphic to the pullback of $\maT_1\to B'$.
Then
$$M'':=\maT''\times_{\wihat G_0}\bigl(P\times (L\backslash\RR^{\ell -k})\bigr)=P\times \maT''\times_{\wihat G_0}(L\backslash\RR^{\ell -k})$$
is a finite covering of $M'$ and thus of $M$. In particular $M''$ is a product of $P$ and a principal
bundle with fiber $L\backslash\RR^{\ell -k}$ over a manifold $B''$ homotopy equivalent to a torus of dimension $k+1$. Within the lift $Q_\sigma''$ of $Q_\sigma$ to $M''$, the product with $P$ is a Riemannian product. 

Further, we have an isomorphism of Lie groups
$$\leftsmallquotientspace{L}{\RR^{\ell-k}}=\bigl(\leftsmallquotientspace{\ZZ^m}{\RR^m}\bigr)\,\times \,\bigl(\leftsmallquotientspace{\ZZ^{\ell-k-m}}{\RR^{\ell-k-m}}\bigr)$$
such that $\wihat G_0=\ZZ^m\backslash \RR^m$ acts on the first factor by left multiplication. This leads to the product decomposition
w$$\maT''\times_{\wihat G_0}\bigl(L\backslash\RR^{\ell -k}\bigr)= \maT''\,\times \,\Bigl(\leftsmallquotientspace{\ZZ^{\ell-k-m}}{\RR^{\ell-k-m}}\Bigr),$$
however this is no longer Riemannian  product, even not in $Q_\sigma''$.

We conclude that, $\pi_1(M)$ contains the subgroup $\pi_1(M'')$ which is of finite index and which fits into a short exact sequence $0\to \ZZ^{\ell - k}\to\pi_1(M'')\to  \ZZ^{k+1}\to 0$.

We summarize what we have proved so far, writing $M_1$ instead of $M''$ and $B_1$ instead of $B''$, and $\maT_1:=\maT''\times_{\wihat G_0} (L\backslash\RR^{\ell -k})$.

\begin{theorem}\label{thm:WdecInv-light-nc.one}
  Assume that a closed connected $n$-dimensional manifold~$M$ carries a generalized initial data triple  $(g,K,\phi)$, and assume that the associated leaves are non-compact.
  Then there is a finite covering $M_1\to M$, some numbers $0\leq k\leq\ell<n$, a closed spin
  manifold $B_1$, a lattice $L$ in $\RR^{\ell-k}$, an $(L\backslash\RR^{\ell -k})$-principal bundle $\pi\colon\maT_1\to B_1$, a closed simply-connected  manifold $(P,h)$ with a parallel spinor, and a foliation $(\maG_\sigma)_{\sigma\in \RR}$
 of codimension $1$ on $B_1$ -- given as the images of level sets of a submersion $s_B:\witi B\to \RR$ --, a flat longitudinal metric on the foliation $(\maG_\sigma)_{\sigma\in \RR}$, and a flat Riemannian metric on the preimages $\pi^{-1}(\maG_\sigma)\subseteq\maT_1$  of any leaf $\maG_\sigma$ such that: 
  \begin{itemize}
  \item The fibers $\maT_1\to B_1$ with the induced metrics are all isometric to  $\bigl(L\backslash \RR^{\ell -k},\geucl\bigr)$ -- the map of $L\backslash \RR^{\ell -k}$ to any orbit is an isometry --, and the action of  $\bigl(L\backslash \RR^{\ell -k},\geucl\bigr)$ is isometric. The map  $\pi^{-1}(\maG_\sigma)\to \maG_\sigma$ is a Riemannian submersion for each $\sigma$.
    \item The dimension of $B_1$ coincides with the $\ZZ$-rank of $\image S$. We define $k:=\dim B_1-1$. And  $B_1$ is homotopy equivalent to a $(k+1)$-dimensional torus.
    \item All leaves $\maG_\sigma$ of $B_1$ are simply connected and thus isometric to euclidean~$\RR^k$. As a consequence $\witi B$ is diffeomorphic to $\RR^{k+1}$.
    \item $M_1$ is diffeomorphic to $\maT_1\times P$.
    \item Let $\witi\maT$ be the pullback of $\maT_1$ to the universal covering $\witi B$ of $B_1$.
      Then the map $s:\witi M= \witi\maT \times P\to \RR$ is the composition of the projection $\witi M=\witi\maT \times P\to \witi\maT\to \witi B$ and the map $s_B$. Thus leaves on $M_1$ are diffeomorphic to a product of  $\pi^{-1}(\maG_\sigma)$ and $P$.
    \item The metric on each leaf of $M_1$ is the Riemannian product of the given metric on  $\pi^{-1}(\maG_\sigma)$ and the metric $h$ on $P$.
    \item The map $s_B$ induces a map $S_1:\pi_1(B_1)\to \RR$, whose image acts as deck transformations for $\witi B\to B_1$,
 and thus we have $\pi_1(B_1)=\image(S_1)$, (obviously this action is only continuous if $\image(S_1)$ carries the  discrete topology; the action preserves the longitudinal metric on $B_1$).
    \item On $\witi B$ the group $\image(S_1)$ acts freely on the space of leaves.
  \end{itemize}
  In particular, the fundamental group $\pi_1(M)$ is virtually solvable of derived length at most~$2$. In more detail: Its finite index subgroup $\pi_1(M_1)$  fits into a short exact sequence
  $0\to \ZZ^{\ell - k}\to \pi_1(M_1)\to \ZZ^{k+1}\to 0$.
\end{theorem}

In fact the number $\ell$ use here is the one
introduced in Proposition~\ref{prop:Qtilde.struc}.

Note that we only made statements about the Riemannian metrics $g_\sigma$ on the leaves $Q_\sigma$, its coverings and projections. The results may not be extended to the full metric $g$ on $M$, as the vector field $U$ and the function $u$ may not be  compatible with the product structure $\wihat T_1\times P$, examples may be obtained by techniques similar to the ones mentioned at the end of Remark~\ref{rem:further.methods}.
However the lifts of $U^\flat$ and of the functions $s$ and $S$ are compatible.\medskip

\begin{corollary}\label{conj:WdecInv-lightlike}
  Let $M$ be a closed connected manifold, with a lightlike generalized initial data triple $(g, K,\phi)$. Then $b_1\neq 0$ and  $\pi_1(M_1)$ is virtually solvable of derived length at most $2$.
\end{corollary}

Let us compare the theorem to the examples given in Subsection~\ref{subsec.example-ll-gen-ini-d-t-gen}.

\begin{example}[Example~\ref{example.lightlike.idt.eins} continued]
  We use the notion of Example~\ref{example.lightlike.idt.eins}, and we assume without loss of generality $\|\phi_0\|=\|\phi\|\equiv 1$. Then $\pi_1(M)=\Gamma$ and $S(w_j)=s_j S(w_n)$, and $s$ is the projection to the $\RR$-factor of $\witi M=W=V\times \RR$ with $N=(0,1) \in V \times \RR$. We obtain $\ell=n-1$, and thus $P$ is a point. Then up to the factor $S(w_n)$, $\image S$ is the group generated by $s_1,\ldots,s_{n-1},s_n=1$, its rank being  $k+1$ by definition. By possibly changing the basis of $\Gamma$ we may achieve $w_1,\ldots,w_{\ell-k} \in V$, \ie  $s_1=\cdots=s_{\ell-k}=0$. Then $(R,\hat h)$ from Proposition~\ref{prop:Qtilde.struc} is the $(\ell-k)$-dimensional torus $\lgen w_1,\ldots,w_{\ell-k}\rgen\backslash \spann(w_1,\ldots,w_{\ell-k})$.
  Here $\lgen w_1,\ldots\rgen$ denotes the subgroup generated by $w_1,\ldots$, while  $\spann(w_1,\ldots)$ is the generated linear subspace of $V$. For the space $\witi B$ we may take the quotient space $W/\spann(w_1,\ldots,w_{\ell-k})$, and let $\Gamma_B$ be the image of $\lgen w_{\ell-k+1}, \ldots,w_n\rgen$ in $\witi B$. We may take $B_1=\Gamma_B\backslash\witi B$ and for the total space of $\maT_1$ we choose $M=M_1$. The leaves $\maG_\sigma$ are given by the image of $V\times \{\sigma\}$ in $B_1$. The map $s_B$ is obtained by factoring $s$ through the projection $W\to \witi B$. Obviously $S_1=S$.
\end{example}

\begin{example}[Example~\ref{example.heisenberg.comb} continued]\label{example.heisenberg.irrat}
  We use the notation of Example~\ref{example.heisenberg.comb}, and we consider the case $\alpha\notin\QQ$. We may take $M_1=M=\HH_3=L_3\backslash \maH_3$. Every $(Q^\alpha \times \{\sigma\},g_\sigma)$ is isometric to $\RR\times (\RR/L\ZZ)$, where $L$ is the length of one of the loops  $\gamma_{x,y}: [0,1]\to \maH_3$, $z\mapsto H_{x,y,z}$, with respect to the left-invariant metric. Obviously $L$ does not depend on $x$ and $y$. Thus the total space of $\maT_1$ is $\HH_3$, and we have $k=1$ and $\ell=2$, and $P$ is a point. For $B_1$ we may take $\Torus[2]$, we interpret  $\Torus[2]$ as the space of the circles $[\gamma_{x,y}]$, \ie it is the quotient of $\HH_3$ by the $S^1$-action, given by the central action of $H_{0,0,z}$, $z\in \RR$.
The map $\pi:\maT_1\to B_1$ is given by $[H_{x,y,z}]\to [\gamma_{x,y}]$.

Then $\maG_{\sigma}$ is a line of irrational slope $-\alpha$ in $\Torus[2]$, $\maG_\sigma=\left\{[\gamma_{x,\sigma-\alpha x}]\mid x\in \RR\right\}$. Furthermore $Q^\alpha \times \{\sigma\}$ is $\pi^{-1}(\maG_\sigma)$. We have discussed $\pi_1(\HH_3)=L_3=\ZZ^2\rtimes \ZZ$, which leads to a short exact sequence $0\to \ZZ^2\to L_3\to \ZZ\to 0$. This is \textbf{not} the short exact sequence given by the theorem; the theorem yields $0\to \ZZ\to L_3\to \ZZ^2\to 0$, this is a way to write $L_3$ as a central extension of $\ZZ^2$ by $\ZZ$.
\end{example}

\begin{remark}
At the end, we should mention that if we apply the above theorem to a closed $n$-dimensional manifold $M$, then not all numbers $\ell\leq n-1$ may arise. Some numbers $\ell$ are obstructued as the above conditions imply that a structured Ricci-flat metric exists on a closed simply-connected $(n-\ell-1)$-dimensional manifold. As such metrics do not exist in dimensions
$1,2,3,5,9$, the case $\ell\in \{n-2,n-3,n-4,n-6,n-10\}$ may not arise. However, for all other integers $\ell\in \{0,1,\ldots,n-1\}$ we have examples.
\end{remark}

\begin{remark}\label{rem.initial.data.triples.with.DEC}
  In these Subsections~\ref{subsec.ll-gen-ini-d-t-gen},  \ref{subsec.ll-gen-ini-d-t-cpct}, \ref{subsec.ll-gen-ini-d-t-compact-noncpct}, and~\ref{subsec.ll-gen-ini-d-t-noncpct} we never used the fact that $(g,K)$ satisfies the dominant energy condition (DEC). We expect further obstructions from this condition. It seems plausible to us, that all lightlike inital data triples satisfying DEC are finitely covered by the product of a simply-connected closed manifold with a parallel spinor and a torus.
\end{remark}

%%%%%%%%%%%%%%%%%%%%%%%%%%%%%%%%%%%%%%%
\subsection{Conclusions}
  %Examples of manifolds without (generalized) initial data triples}
%%%%%%%%%%%%%%%%%%%%%%%%%%%%%%%%%%%%%%%%

We will now apply the above results in order to see that the Dirac-Witten operator is often invertible when we only assume the dominant energy condition and not the strict dominant energy condition, as the existence of initial data triples is topologically obstructued.

\begin{examples}\label{examples.all.obstructed}
We provide some further examples of closed spin manifolds with an obstruction to initial data triples.
\begin{enumerate}
\item Let $M$ be a closed spin manifold whose fundamental group is not virtually solvable. Then Corollary~\ref{conj:WdecInv-lightlike} tells us that $M$ cannot carry a lightlike initial data triple. Moreover, as $\pi_1(M)$ is not virtually abelian, Example~\ref{example.obstr.structured} \eqref{example.obstruc.para.spin.i} and Corollary~\ref{cor:WdecInv-old} show that we also do not have a timelike initial data triple.
\item Let $M$ be a closed spin manifold which admits a metric of non-positive curvature $g$. Recall that Wolf's conjecture, proved by S.T. Yau \cite[Corollary~1]{yau:71}, says that if $\pi_1(M)$ is virtually solvable, then $g$ is flat. This is a converse to the well-known fact that if $g$ is flat, then $\pi_1(M)$ is a Bieberbach group and thus virtually abelian. So let us assume that $g$ is non-flat. Then $\pi_1(M)$ is not virtually solvable and thus $M$ cannot carry an initial data triple. 
%\item
%  The sphere $M=S^n$ does not have a structured Ricci-flat metric, and thus no timelike initial data triple, see Example~\ref{example.obstr.structured} \eqref{example.obstruc.para.spin.iv}
%  Because of $b_1(S^n)=0$  lightlike initial data triples are not possible either.
%  %see Bryant's comment in \url{https://mathoverflow.net/questions/68094/ricci-flat-metric-on-n-sphere?noredirect=1&lq=1}.
\end{enumerate}
\end{examples} 

\begin{corollary} \label{cor:obstr.lightlike.id-triple}
  Suppose that $M$ is a closed connected spin manifold that has a
  topological obstruction to an initial data triple, for example if one of the following holds:
  \begin{itemize}
  \item $b_1(M)=0$ and there are obstructions to a Ricci-flat metric with a parallel spinor, see e.g.\ Examples~\ref{example.obstr.structured}  \eqref{example.obstruc.para.spin.iii}, \eqref{example.obstruc.para.spin.iv} and \eqref{example.obstruc.para.spin.v}.
  \item $M$ is one of the examples in Examples~\ref{examples.all.obstructed}.
  \end{itemize}
    Then for every $(g,K)$ satisfying the dominant energy condition the Dirac--Witten operator is invertible.
\end{corollary}

%%%%%%%%%%%%%%%%%%%%%%%%%%%%%%%%%%%%%%%%%%%%%%%%%%%%%
\section{Homotopy groups of $\Decstr(M)$ and  $\Dec(M)$}\label{sec.homotopy.groups}
%%%%%%%%%%%%%%%%%%%%%%%%%%%%%%%%%%%%%%%%%%%%%%%%%%%%%
In this section we want to review and then slightly extend the results on the space $\Decstr(M)$ of initial data sets $(g, K)$ subject to the strict dominant energy condition that were obtained by the second author \cite{gloeckle:p2019}.
As stated in the introduction, the goal of this extension is to get rid of the strictness assumption, \ie to obtain topological information about $\Dec(M)$.
In particular, we will see that this works when $\pi_1(M)$ is not virtually solvable.

More precisely, in this section we will assume that $M$ is a closed manifold, $M\neq \emptyset$. For the definition of the space of initial data pairs, and strict and non-strict dominant energy condition we refer to Subsection~\ref{subsec.dominant.energy}. In particular, in Definition~\ref{def.dec.ids}, we introduced the spaces $\Decstr(M)\subseteq \Dec(M)\subseteq \Ini(M)$.
The main purpose of the aforementioned article is to construct and detect non-trivial elements in $\pi_k(\Decstr(M))$ for certain $k \in \NN$.  
We will then consider the maps $\pi_k(\Decstr(M)) \to \pi_k(\Dec(M))$ induced by inclusion.
Our analysis from Section~\ref{sec.kernel.DWo} will provide sufficient criteria for the non-trivial elements to survive.

%----------------------------------------------------
\subsection{Initial data sets and positive scalar curvature}\label{subsec.ini.data.pairs}
%----------------------------------------------------
The construction of candidates for non-trivial elements in $\pi_k(\Decstr(M))$ in \cite{gloeckle:p2019} relies on a link between $\Decstr(M)$ and $\Psc(M)$, the space of positive scalar curvature metrics, which is a subspace of the $C^\infty$-space $\Met(M)$ of all metrics on $M$.

This link comes from considering initial data sets $(g, K)$, where $K$ is a purely constant trace tensor, \ie $K = \tau g$ for some $\tau \in \RR$.
Such a pair satisfies the dominant energy condition in the strict sense $\rho > \|j\|$ if and only if $\scal^g > -n(n-1) \tau^2$.
Hence, if we define
\begin{align*}
	\tau \colon \Met(M) &\lto \RR \\
	g &\lmapsto \sqrt{\frac{1}{n(n-1)}\min\{0,-\min_{x \in M} \scal^g(x)\}} +C
\end{align*}
for some $C > 0$, then $\bigl(g, \pm\tau(g) g\bigr) \in \Decstr_{\pct}(M)$, where the subscript $\pct$ indicates that we are considering the subspace of $\Decstr(M)$, where $K$ is a purely constant trace tensor.
Note that $\tau$ is continuous and thus gives rise to a continuous map of pairs
\begin{align*}
	\phi \colon \Bigl(\Met(M) \times I,\; \bigl(\Psc(M) \times I\bigr) \cup \bigl(\Met \times \del I\bigr)\Bigr) &\lto \Bigl(\Ini_{\pct} (M),\; \Decstr_{\pct}(M)\Bigr) \\
	(g, t) &\lmapsto \bigl(g, t\tau(g)g\bigr),
\end{align*}
where $I = [-1,1]$.
Although $\phi$ depends on the choice of the constant $C > 0$, its homotopy class is easily seen to be independent of this choice. Moreover, we have the following:
\begin{theorem}
The map $\phi \colon \big(\Met(M) \times \del I\bigr) \cup \bigl( \Psc(M) \times I\bigr) \overset{\phi}{\longrightarrow} \Decstr_{\pct}(M)$ is a homotopy equivalence.
\end{theorem}
\begin{proof}
	First of all, we note that $\Met(M) \times \RR \to \Ini_{\pct}(M), (g, t) \mapsto \bigl(g, t\tau(g)g\bigr)$ is a homeomorphism (as we have assumed $M \neq \emptyset$).
	Under this homeomorphism $\Decstr_{\pct}(M)$ corresponds to the subspace
          $$U = \Bigl\{(g,t) \,\Bigm|\, \scal^g + n(n-1)\tau(g)^2 t^2 > 0 \Bigr\} \subseteq \Met(M) \times \RR\,$$
        and we have to show that the inclusion $\bigl(\Met(M) \times \del I\bigr) \cup \bigl(\Psc(M) \times I \bigr)\to U$ is a homotopy equivalence.
	
	This follows from the existence of a continuous function $f \colon \RR^2 \setminus (\RR_{\leq 0} \times \{0\}) \to I$ with $f(s,t) =\sign(t) \in \del I$ for all $s \leq 0, t \neq 0$:
	Given such a function, a homotopy inverse of the inclusion is given by
	\begin{align*}
		U &\lto \bigl(\Met(M) \times \del I\bigr) \cup \bigl(\Psc(M) \times I\bigr) \\
	 	(g, t) &\lmapsto \bigl(g, f(\min_{x \in M} \scal^g(x), t)\bigr).
	\end{align*}
	The sign condition for $f$ guarantees that convex combination (affecting only the $\RR$-variable) yields the required homotopies.
	It is elementary to check that
	\begin{align*}
	f \colon \RR^2 \setminus (\RR_{\leq 0} \times \{0\}) &\lto I \\
	(s, t) &\lmapsto \begin{cases}
		\frac{t}{\sqrt{t^2+s^2}} & s > 0 \\
		\sign(t) & s \leq 0
	\end{cases}
	\end{align*}	
	is continuous and thus a function with the properties mentioned above.
\end{proof}

As $\Met(M)$ is contractible, the canonical projection map
$$\bigl(\Met(M) \times \del I\bigr) \cup \bigl(\Psc(M) \times I\bigr) \to \Susp \Psc(M)$$
to the suspension of $\Psc(M)$ is a homotopy equivalence.
For instance, a homotopy inverse may be given by the explicit formula
	\begin{align*}
	\Susp \Psc(M) &\lto \bigl(\Met(M) \times \del I\bigr) \cup \bigl(\Psc(M) \times I\bigr) \\
	[g,t] &\lmapsto
	\begin{cases}
	\bigl((-2t-1) h + 2(1+t)g, -1 \bigr) & t \in [-1,-\frac12]\\
	(g, 2t) & t \in [-\frac12 , \frac12]\\
	\bigl((2t-1) h + 2(1-t)g, 1 \bigr) & t \in [\frac12, 1],
	\end{cases}
	\end{align*}
where $h \in \Met(M)$ is some chosen basepoint.
Note that we follow the convention that $\Susp \emptyset \cong \del I$, and then the homotopy inverse maps $t\in \del I$ to $(h,t)$.
Combining this discussion with the previous theorem,
we obtain:
\begin{corollary}
  The canonical homotopy equivalence $\bigl(\Met(M) \times \del I\bigr) \cup \bigl(\Psc(M) \times I\bigr) \to \Susp \Psc(M)$ and the map $\phi$ from above induce a homotopy equivalence
  $$\Phi \colon \Susp \Psc(M) \to \Decstr_{\pct}(M),$$
which is defined independently of our choices up to homotopy.
\end{corollary}
This implies that $\Phi_\ast \colon \pi_k\bigl(\Susp \Psc(M), [h,1]\bigr) \to \pi_k\Bigl(\Decstr_{\pct}(M), \bigl(h, \tau(h)h\bigr)\Bigr)$ is an isomorphism for all $k$. Thus, we may consider the composition
\begin{align*}
	\pi_{k-1}\bigl(\Psc(M), h\bigr) &\lto \pi_k\bigl(\Susp \Psc(M), [h,1]\bigr) \cong \pi_k\Bigl(\Decstr_{\pct}(M), \bigl(h, \tau(h)h\bigr)\Bigr) \\
	&\lto \pi_k\Bigl(\Decstr(M), \bigl(h, \tau(h)h\bigr)\Bigr),
\end{align*}
	where the first map is the suspension homomorphism and the last one is induced by the inclusion $\Decstr_{\pct}(M)\hookrightarrow \Decstr(M)$.
	We will show that this maps certain non-trivial elements in the homotopy groups of $\Psc(M)$ to non-trivial elements in the homotopy groups of $\Decstr(M)$.

	We are particularly interested in the case $k = 0$, which is a bit special.
	Notice that $\Susp \Psc(M)$ is path-connected unless $\Psc(M) = \emptyset$, in which case $\Susp \Psc(M) = \del I$ consists of precisely two points.

     Note that it follows from what we proved so far: if two metrics $g$ and $g'$ and real numbers $\tau,\tau'>0$ satisfy  $(g, \tau g),(g', \tau' g')\in \Decstr_{\pct}(M)$, then $(g, \tau g)$ and   $(g', \tau' g')$ are in the same path component of $\Decstr_{\pct}(M)$. And the same holds for $\tau,\tau'<0$.

	\begin{definition}
	  The path-component $C_+ \in \pi_0\bigl(\Decstr(M)\bigr)$ that contains all elements $\bigl\{(g, \tau g)\in \Decstr_{\pct}(M)\mid \tau \in \RR_{>0}\bigr\}$ is called \emph{component of expanding initial data}. 
	  Similarly, the path-component $C_- \in \pi_0\bigl(\Decstr(M)\bigr)$ that contains all $\bigl\{(g, \tau g)\in \Decstr_{\pct}(M)\mid \tau \in \RR_{<0}\bigr\}$ is the \emph{component of contracting initial data}.
	\end{definition}
	The discussion above shows that $C_+ = C_-$ if $\Psc(M) \neq \emptyset$, as $\Decstr_{\pct}(M)$ is path-connected in this case.
	In the case $\Psc(M) = \emptyset$, it is not clear whether $C_+ \neq C_-$ although they are induced by distinct path-components of $\Decstr_{\pct}(M)$.

	However, we will show that this is indeed the case, when the obstruction against positive scalar curvature metrics on $M$ is given by the so-called $\alpha$-index.

%----------------------------------------------------
\subsection{The $\alpha$-index and index difference for psc metrics}
%----------------------------------------------------
\label{sec:adiff}
Index theoretic methods play an important role for obstructing the existence of Riemannian metrics of positive scalar curvature as well as for detecting non-trivial homotopy groups of $\Psc(M)$.
The first task is for example carried out by the $\KO$-valued $\alpha$-index.
The index difference, a family version of the $\alpha$-index, was first used by Hitchin \cite{hitchin:74} to serve the second purpose.
Here, we briefly recall their construction, building on the framework laid out by Ebert in \cite{ebert:13}.

In the following, we additionally assume that the closed $n$-dimensional manifold~$M$ is spin.
For a fixed metric $g$, we work with its $\Cl_n$-linear spinor bundle $\Sigma_{\Cl} M = P_{\Spin(n)} M \times_{\Spin} \Cl_n$.
This bundle carries a right $\Cl_n$-action and we denote by $c \colon \RR^n \to \End(\Sigma_{\Cl} M)$ the corresponding Clifford multiplication.
The scalar product on $\Cl_n$ described in the course of Lemma \ref{lem:ScalarProducts} induces a bundle metric $\< -, - \>$ on $\Sigma_{\Cl} M$, with respect to which the right Clifford multiplication is skew-adjoint.
Moreover, the even-odd grading on $\Cl_n$, induced from $\RR^n \to \RR^n,\, v \mapsto -v$, gives rise to a $\Ztwo$-grading $\iota \in \End(\Sigma_{\Cl} M)$ on this spinor bundle.
The grading operator $\iota$ is self-adjoint and anti-commutes with the right Clifford multiplication.
Taking all this together, we obtain that the space of $L^2$-sections $H = L^2(\Sigma_{\Cl} M)$ is a ($\Ztwo$-graded) \emph{$\Cl_n$-Hilbert space} in the sense of \cite[Def.\ 2.1]{ebert:13}.
Moreover, it is \emph{ample}, meaning that it contains each irreducible ($\Ztwo$-graded) $\Cl_n$-module infinitely often.

The bundle $\Sigma_{\Cl} M$ carries a connection $\nabla$ induced from the Levi-Civita connection of $g$.
With respect to this, the right Clifford multiplication $c$, the bundle metric $\langle -, - \rangle$ and the grading operator $\iota$ are parallel.
Thus the associated Dirac operator~$\Dirac$, or rather its bounded transform $\frac{\Dirac}{\sqrt{1+ \Dirac^2}}$, is a \emph{$\Cl_n$-Fredholm operator} on $H$ (cf.\ \cite[Def.\ 2.6]{ebert:13}. Here and in the following we understand by $\Cl_{n,k}$-Fredholm operator that it also fulfills the extra technical condition for $n-k \equiv 3 \mod 4$ appearing in the definition of $\Fred^{n,k}(H)$. That $\Dirac$ also satisfies this condition needs a bit of extra thought.
We want to consider a suitable index of $\frac{\Dirac}{\sqrt{1+ \Dirac^2}}$ that takes into account both the grading and the $\Cl_n$-linear structure.

The \emph{$\KO$-index map} goes back to Atiyah and Singer \cite{atiyah.singer:69}.
For a compact space $X$, this is a natural map
\begin{align*}
	\ind \colon [X,\; \Fred^{n,k}(H)] \to \KO^{k-n}(X),
\end{align*}
where $\Fred^{n,k}(H)$ denotes the space of $\Cl_{n,k}$-Fredholm operators (with norm-topology) on the ample $\Cl_{n,k}$-Hilbert space $H$.
The main result of \cite{atiyah.singer:69} implies that the index map is a bijection.
Let $G^{n,k}(H)$ be the subspace of invertible operators in $\Fred^{n,k}(H)$.
Using that $G^{n,k}(H)$ is contractible (\cite[Lem.\ 2.8]{ebert:13}), the index map and the statement above can be extended to a relative setting:
\begin{theorem} \label{Thm:IndexMap}
For any compact CW-complex $(X,Y)$, there is a natural bijection
\begin{align*}
	\ind \colon [(X, Y),\; (\Fred^{n,k}(H), G^{n,k}(H)] \to \KO^{k-n}(X, Y).
\end{align*}
\end{theorem}

For a single operator $F \in \Fred^{n,k}(H)$, its \emph{$\KO$-valued index} is given by $\ind([F]) \in \KO^{k-n}(\pt)$, where, by abuse of notation, $F$ is identified with a map $\pt \to \Fred^{n,k}(H)$ in the obvious way.
\begin{definition}
	The \emph{$\alpha$-index} is defined to be
	$\alpha(M) = \ind\left(\left[\frac{\Dirac}{\sqrt{1+ \Dirac^2}}\right]\right) \in \KO^{-n}(\pt)$.
\end{definition}
\begin{remarks}
\begin{enumerate}
\item
Notice that this definition implicitly claims that $\alpha(M)$ is independent of the metric.
Indeed, as any two metrics can be joined by a continuous path, we may join the corresponding Dirac operators by a continuous path in $\Fred^{n}(H)$ and hence they define the same homotopy class of a map $\pt \to \Fred^{n}(H)$.
However, this argument is not yet completely rigorous, as also the bundle $\Sigma_{\Cl} M$ and hence $H = L^2(\Sigma_{\Cl} M)$ depends on the metric.
There are various ways to overcome this problem.
A very concrete one is presented in \cite{gloeckle:p2019}, where the method  of generalized cylinders due to Bär, Gauduchon and Moroianu \cite{baer.gauduchon.moroianu:05} is used to obtain explicit isomorphisms between the Hilbert spaces.

\item
If $M$ carries a metric of positive scalar curvature, then for this metric the Dirac operator is invertible.
It follows that $\frac{\Dirac}{\sqrt{1+ \Dirac^2}} \in G^n(H)$ and $\alpha(M) = 0$.
In particular, non-zero $\alpha$-index is an obstruction to the existence of psc metrics.

\item
Recall that
\begin{align*}
	\KO^{-n}(\pt) \cong
	\begin{cases}
		\ZZ & n \equiv 0,4 \mod 8 \\	
		\Ztwo & n \equiv 1,2 \mod 8 \\
		0 & \text{else}.
	\end{cases}
\end{align*} 
Hence, $\alpha(M)$ can only be non-zero if the dimension of $M$ is $0,1,2$ or $4$ mod~$8$.
\end{enumerate}
\end{remarks}

Now suppose that there exists a positive scalar curvature metric $h$ on $M$.
The purpose of the $\alpha$-index difference is to detect non-triviality of $\pi_k(\Psc(M), h)$ for some $k$..
First note that, as $\Met(M)$ is contractible, the boundary map $\pi_{k+1}(\Met(M), \Psc(M), h) \to \pi_k(\Psc(M), h)$ is an isomorphism. 
Assigning to every metric $g$ the Fredholm operator $F_g = \frac{\Dirac_g}{\sqrt{1+ \Dirac_g^2}}$ yields a continuous map of pairs $(\Met(M), \Psc(M)) \to (\Fred^n(H), G^n(H))$.
In particular, this provides a homomorphism $\pi_{k+1}(\Met(M), \Psc(M), h) \to \pi_{k+1}(\Fred^n(H), G^n(H), F_h)$.
Finally, forgetting the basepoint and applying the index map gives the composition
\begin{align*}
	\pi_{k+1}(\Fred^n(H), G^n(H), F_h)
	&\lto [(D^{k+1}, S^k),\; (\Fred^n(H), G^n(H))] \\
	 &\lto \KO^{-n}(D^{k+1}, S^k) \cong \KO^{-n-k-1}(\pt),
\end{align*} 
which is even bijective due to Theorem \ref{Thm:IndexMap} and contractibility of $G^n(H)$.

\begin{definition}
The \emph{$\alpha$-index difference} (or just \emph{index difference}) is defined to be the composition
\begin{align*}
	\adiff \colon \pi_k(\Psc(M), h) &\cong \pi_{k+1}\bigl(\Met(M), \Psc(M), h\bigr) \\
	&\lto \pi_{k+1}\bigl(\Fred^n(H), G^n(H), F_h\bigr) \cong \KO^{-n-k-1}(\pt)
\end{align*}
of the maps described above.
\end{definition}

\begin{remarks} \label{Rem:adiff}
\begin{enumerate}
\item
The map $\bigl(\Met(M), \Psc(M)\bigr) \to \bigl(\Fred^n(H), G^n(H)\bigr)$ required in the definition cannot be naively defined by $g \mapsto F_g$.
The problem is again that the operators $F_g$ act on different $\Cl_n$-Hilbert spaces.
To make sense of the assignment $g \mapsto F_g$, these Hilbert spaces have to be identified in a suitable way.
We refer once more to \cite{gloeckle:p2019} for more details.

\item
For $k \geq 1$, the index map $\pi_{k+1}\bigl(\Fred^n(H), G^n(H), F\bigr) \to \KO^{-n-k-1}(\pt)$ and thus also the $\alpha$-index difference are homomorphisms.

In fact, it is easy to derive from the definition of the index map that given $D \colon (X, Y) \to \bigl(\Fred^n(H), G^n(H)\bigr)$ and $D^\prime \colon (X, Y) \to \bigl(\Fred^n(H^\prime), G^n(H^\prime)\bigr)$, then $\ind([D \oplus D^\prime]) = \ind([D]) + \ind([D^\prime])$, where $D \oplus D^\prime$ denotes the block diagonal operator family on $H \oplus H^\prime$.
Furthermore, for $[D_1],\,[D_2] \in \pi_{k+1}\bigl(\Fred^n(H), G^n(H), F\bigr)$ and $[D_1^\prime],\,[D_2^\prime] \in \pi_{k+1}\bigl(\Fred^n(H^\prime), G^n(H^\prime), F^\prime\bigr)$ the identity
\begin{align*}
	 \bigl([D_1]*[D_2]\bigr) \oplus \bigl([D_1^\prime]*[D_2^\prime]\bigr) = \bigl[D_1 \oplus D_1^\prime\bigr] * \bigl[D_2 \oplus D_2^\prime\bigr]
\end{align*} 
holds.
Together, this implies that the binary operation on $\KO^{-n-k-1}(\pt)$ induced by the $\pi_{k+1}$-composition $*$ and the ordinary addition of $\KO^{-n-k-1}(\pt)$ satisfy the requirements of the Eckmann-Hilton argument and are thus equal.
\end{enumerate}
\end{remarks}

We conclude the section by stating some results connected with these invariants. 
The first one asserts that in the simply connected case, the $\alpha$-index is a complete obstruction to positive scalar curvature.
It is due to Stolz, building on work of Gromov and Lawson \cite{gromov.lawson:80}.
\begin{theorem}[\cite{stolz:92}]
If $M$ is simply connected and of dimension $n \geq 5$, then $\Psc(M) = \emptyset$ if and only if~$M$ is spin with $\alpha(M) \neq 0$.
\end{theorem}

The second and third results state that the $\alpha$-index difference detects that the homotopy groups of $\Psc(M)$ are quite rich in general.
It should be noted that although the final results are similar, the authors of \cite{crowley.steimle.schick:18} use very different constructions compared to the authors of \cite{botvinnik.ebert.randal-williams:14}] for establishing the non-triviality of the $\alpha$-index difference.
\begin{theorem}[\cite{crowley.steimle.schick:18}]
If $M$ is spin, of dimension $n \geq 6$ and $h \in \Psc(M) \neq \emptyset$, then $\adiff \colon \pi_k(\Psc(M), h) \to \KO^{-n-k-1}(\pt) \cong \Ztwo$ is split surjective for all $k \geq 0$ with $k +n+ 1 \equiv 1, 2 \mod 8$.
\end{theorem}

\begin{theorem}[\cite{botvinnik.ebert.randal-williams:14}]
If $M$ is spin, of dimension $n \geq 6$ and $h \in \Psc(M) \neq \emptyset$, then $\adiff \colon \pi_k(\Psc(M), h) \to \KO^{-n-k-1}(\pt) \cong \Ztwo$ and $\adiff \otimes \mathrm{id}_\QQ \colon \pi_k(\Psc(M), h) \otimes \QQ \to \KO^{-n-k-1}(\pt) \otimes \QQ \cong \QQ$
are surjective for all $k \geq 0$ with $n + k + 1 \equiv 1, 2 \mod 8$ or $n + k + 1 \equiv 0, 4 \mod 8$, respectively.
\end{theorem}
As the group structure of $\pi_0(\Psc(M))$ is unclear, the above statement should be understood in the case $k=0$ and $n+1\equiv 0 \mod 4$ in the sense that at least one element of $\pi_0(\Psc(M))$ is mapped to a non-trivial element in $\KO^{-n-k-1}(\pt)$.

%----------------------------------------------------
\subsection{The index difference for initial data sets strictly satisfying DEC}\label{subsec.index.diff.idp}
%----------------------------------------------------
We now explain how the construction from the previous section can be adapted to obtain an index difference for initial data sets strictly satisfying the dominant energy condition.
A comparison theorem between these index differences will then allow us to get interesting statements about the topology of $\Decstr(M)$.

For this purpose, we consider the $\Cl_{n,1}$-linear hypersurface spinor bundle on~$M$, \ie $\ol\Sigma_{\Cl} M = P_{\Spin(n)} M \times_\Spin \Cl_{n,1}$, associated to the metric $g$ on~$M$.
As the name already indicates, this bundle carries a right $\Cl_{n,1}$-action.
This Clifford multiplication $c \colon \RR^{n,1} \to \End(\ol\Sigma_{\Cl}M)$ is compatible with the (positive definite) scalar product $\<-, -\>$ induced by the scalar product on $\Cl_{n,1}$ (cf.\ discussion before Lemma \ref{lem:ScalarProducts}) in the sense that multiplication by the Riemannian basis vectors $E_1, \ldots E_n$ is skew-adjoint whereas multiplication by $E_0$ is self-adjoint.
Again, there is an even-odd grading operator $\iota \in \End(\ol\Sigma_{\Cl}M)$, which is self-adjoint and anti-commutes with $c$.
Thus the space of $L^2$-sections $\ol H = L^2(\ol\Sigma_{\Cl} M)$ gets the structure of an ample ($\Ztwo$-graded) $\Cl_{n,1}$-Hilbert space.

This can be further improved:
As in $\Cl_{n,1}$ the left multiplication with $E_0$ commutes with the left multiplication by $\Spin(n) \subset \Cl_{n,1}$, there is an induced operation on $\ol\Sigma_{\Cl}M$.
We will refer to this as left multiplication by $e_0$, as in the case where $M$ is a space-like hypersurface in a time-oriented Lorentzian manifold $N$, it corresponds to left multiplication with the unit normal $e_0$ on $\Sigma_{\Cl}N_{|M} \cong \ol\Sigma_{\Cl}M$, cf.\ Section \ref{sec:HypSpinBun}.
It is elementary to check that $c(e_{n+1})(\Psi) = e_0 \cdot \iota(\Psi)$, for $\Psi \in \ol\Sigma_{\Cl}M$, may be used to define an extension $c \colon \RR^{n+1,1} \to \End(\ol\Sigma_{\Cl}M)$ of the (right) Clifford multiplication, which is still compatible with $\<-, -\>$ and $\iota$.
Hence $\ol H$ has the structure of an ample ($\Ztwo$-graded) $\Cl_{n+1,1}$-Hilbert space.

The Levi-Civita connection of $g$ induces a connection $\nabla$ on $\ol\Sigma_{\Cl}M$, with respect to which the $\RR^{n+1,1}$-Clifford multiplication $c$, the scalar product $\<-, -\>$ and the grading operator $\iota$ are parallel.
We get an associated Dirac operator $\Dirac$, whose bounded transform defines a $\Cl_{n+1,1}$-Fredholm operator on $\ol H$.
However, we are more interested into the Dirac-Witten operator defined by $\ol\Dirac \Psi = \Dirac \Psi - \frac12 \tr(K) e_0 \cdot \Psi$, for $\Psi \in \Gamma(\ol\Sigma_{\Cl}M)$, which not only depends on the metric $g$, but also on the second component $K$ of an initial data set $(g, K)$.
The Dirac-Witten operator is only $\Cl_{n,1}$-linear, not $\Cl_{n+1,1}$-linear as left multiplication with $e_0$ does not commute with $c(e_{n+1})$ by definition.
Yet, it is still odd and formally self-adjoint, and hence the bounded transform $\frac{\ol\Dirac}{\sqrt{1+\ol\Dirac^2}}$ is a $\Cl_{n,1}$-Fredholm operator.
Moreover, the Schrödinger-Lichnerowicz type formula \eqref{eq:SchrL} shows that $\ol\Dirac$ is invertible if $(g, K) \in \Decstr(M)$.

Similarly to the situation in the previous subsection, it is possible to manufacture a continuous map of pairs $(\Ini(M), \Decstr(M)) \to (\Fred^{n,1}(\ol H), G^{n,1}(\ol H))$ that roughly spoken associates an initial data set $(g, K)$ the respective operator $\ol F_{g,K} = \frac{\ol\Dirac_{g, K}}{\sqrt{1+\ol\Dirac_{g, K}^2}}$.
Hence, using the induced map $\pi_{k+1}\bigl(\Ini(M), \Decstr(M), (h,L)\bigr) \to \pi_{k+1}\bigl(\Fred^{n,1}(\ol H), G^{n,1}(\ol H), \ol F_{h, L}\bigr)$, we may define the following.

\begin{definition}
	The \emph{$\ol\alpha$-index difference} is defined to be the composition
	\begin{align*}
		\oladiff \colon &\pi_k\bigl(\Decstr(M), (h, L)\bigr) \cong \pi_{k+1}\bigl(\Ini(M), \Decstr(M), (h,L)\bigr) \\
			&\lto \pi_{k+1}\bigl(\Fred^{n,1}(\ol H), G^{n,1}(\ol H), \ol F_{h, L}\bigr) \cong \KO^{-n-k}(\pt).
	\end{align*}
\end{definition}

The comments of Remark \ref{Rem:adiff} on the $\alpha$-index difference likewise apply for the $\ol\alpha$-index difference.
Note that we did not define an “$\ol\alpha$-index”.
The reason is that on any compact manifold $M \neq \emptyset$, there exists an initial data set that strictly satisfies the dominant energy condition, \eg $(g, \lambda g)$ for large $\lambda$, and hence the index of $\ol\Dirac$ is always zero.

\begin{theorem} \label{thm:CompAdiff}
For $h \in \Psc(M)$ and $k \geq 0$, the following diagram commutes:
\begin{equation*}
	\begin{tikzcd}
		\pi_k\bigl(\Psc(M), h\bigr) \rar{\mathrm{Susp}} \ar[dr, "\adiff"'] & \pi_{k+1}\bigl(\Susp\Psc(M), [h,1]\bigr) \rar{\Phi_*} & \pi_{k+1}\bigl(\Decstr(M), (h, \tau(h)h\bigr)) \ar[dl, "\oladiff"] \\
		& \KO^{-n-k-1}(\pt) &
	\end{tikzcd}
\end{equation*}
Moreover, if the basepoint of $\Decstr(M)$ is chosen to lie in $C_+ \in \pi_0(\Decstr(M))$, then $\oladiff(C_-) = \alpha(M) \in \KO^{-n}(\pt)$.
\end{theorem}
We only sketch the proof. For details we refer to \cite{gloeckle:p2019}.
\begin{proof}[Sketch of proof]
The proof essentially boils down to the following:
Suppose $(g_x)_{x \in D^{k+1}}$ is a continuous family of metrics with $g_x \in \Psc(M)$ for all $x \in S^k = \del D^{k+1}$.
Denote by $(\Dirac_x)_{x}$ the associated family of $\Cl_n$-linear Dirac operators.
Moreover, let $(\olDirac_{x,t})_{x,t}$ be the family of $\Cl_{n+1}$-linear Dirac-Witten operators associated to the family  $((g_x, t\tau(g_x)g_x))_{(x,t) \in D^{k+1} \times I}$ of initial data sets (with $(g_x, t\tau(g_x)g_x) \in \Decstr(M)$ for all $(x, t) \in \del(D^{k+1} \times I)$).
It is to show that the family indices of $(\Dirac_x)_x$ and $(\olDirac_{x,t})_{x,t}$ coincide.
For simplicity, we will not explicitely employ the clumsy bounded transforms here and argue only up to sign.

The $\Cl_n$-linear Dirac operators are defined on the space $H = L^2(\Sigma_{\Cl} M)$, but we may also consider the Dirac operators on $\ol H= L^2(\ol\Sigma_{\Cl} M)$, and -- as was remarked above -- these are $\Cl_{n+1,1}$-linear.
It is a fundamental property of the $\KO$-index that it is invariant under this “doubling procedure”, \ie $\ind((\Dirac_x)_x) \in \KO(\pt)$ is the same whether we consider the operators as $\Cl_n$-linear on $H$ or as $\Cl_{n+1,1}$-linear on $\ol H$.
Now another invariance of the $\KO$-index comes into play:
The Bott map produces from the family $(\Dirac_x)_{x \in D^{k+1}}$ of $\Cl_{n+1,1}$-linear operators a family $(\Dirac_x - t c(e_{n+1}) \iota)_{(x, t) \in D^{k+1} \times I}$ of $\Cl_{n,1}$-linear operators with the same family index (at least up to sign).
But, by definition, the Dirac-Witten operator for $(g_x, t\tau(g_x)g_x)$ is $\olDirac_{x,t} = \Dirac_x - t\frac{n\tau(g_x)}{2}e_0 \cdot = \Dirac_x - t\frac{n\tau(g_x)}{2} c(e_{n+1}) \iota$, and the claim follows by a simple rescaling in the second summand.
\end{proof}

Using the theorems cited at the end of Section \ref{sec:adiff}, we immediately get the following conclusions.
\begin{corollary}
If $M$ is spin with $\alpha(M) \neq 0$, then $C_+$ and $C_-$ are different path-components of $\Decstr(M)$.
In particular, when $M$ is simply connected of dimension $n \geq 5$ (and not necessarily spin), we have $C_+ = C_-$ if and only if $M$ carries a metric of positive scalar curvature.
\end{corollary}
\begin{corollary}
If $M$ is spin, of dimension $n \geq 6$ and $h \in \Psc(M) \neq \emptyset$, then $\oladiff \colon \pi_k\bigl(\Decstr(M), (h, \tau(h)h)\bigr) \to \KO^{-n-k}(\pt)$ is non-trivial for all $k \geq 1$ for which the target is non-trivial. Moreover, when the target is $\Ztwo$, it is split surjective.
\end{corollary}

We want to extend this to the initial data sets that satisfy DEC but not necessarily in the strict sense.
\begin{corollary} \label{cor:Wdec}Assume addtionally, that $M$ has a topological obstruction to the existence of an initial data triple as discussed in Corollary~\ref{cor:obstr.lightlike.id-triple}.
  %If additionally $b_1(M) = 0$ and $M$ does not admit a Ricci-flat metric with a parallel spinor,
  Then the two corollaries above also hold for $\Dec(M)$ instead of $\Decstr(M)$:
\begin{itemize}
	\item
	If $M$ is spin with $\alpha(M) \neq 0$, then the path-components $\tilde{C}_\pm \in \pi_0(\Dec(M))$ induced by $C_\pm$ are different.
In particular, when $M$ is simply connected of dimension $n \geq 5$ (and not necessarily spin), we have $\tilde{C}_+ = \tilde{C}_-$ if and only if~$M$ carries a metric of positive scalar curvature.
	\item
	  If $M$ is spin, of dimension $n \geq 6$ and $h \in \Psc(M) \neq \emptyset$, then the map $\oladiff$ extends to a homomorphism
          $$\pi_k(\Dec(M), (h, \tau(h)h))\to \KO^{-n-k}(\pt),$$
          that is non-trivial for all $k \geq 1$ for which $\KO^{-n-k}(\pt) \neq 0$. Moreover, when $\KO^{-n-k}(\pt) \cong \Ztwo$, then the homomorphism is split surjective.
\end{itemize}
\end{corollary}
\begin{proof}
  It should first be noted that the $\ol\alpha$-index difference factors as
     $$\pi_k\bigl(\Decstr(M), (h,L)\bigr) \lto \pi_k\bigl(\Ini^{\inv}(M), (h,L)\bigr) \lto \KO^{-n-k}(\pt).$$
Thereby, $\Ini^{\inv}(M)$ is meant to denote the subspace of initial data sets for which the Dirac-Witten operator is invertible and the first map is induced by the inclusion $\Decstr(M) \hookrightarrow \Ini^{\inv}(M)$.
The reason is that the only property of $\Decstr(M)$ that was needed to construct $\oladiff$ is that its elements possess an invertible Dirac-Witten operator.

Now, according to Corollary~\ref{cor:obstr.lightlike.id-triple} under the additional assumptions on $M$, $\Dec(M)$ includes into $\Ini^{\inv}(M)$.
Thus there is a factoriszation of inclutions $\Decstr(M) \hookrightarrow \Dec(M) \hookrightarrow \Ini^{\inv}(M)$ and the non-trivial elements of $\pi_k(\Decstr(M))$ detected by $\oladiff$ give rise to non-trivial elements of $\pi_k(\Dec(M))$.
\end{proof}

%----------------------------------------------------
\subsection{Application to general relativity}\label{subsec.app.relativity}
%----------------------------------------------------

From the perspective of physics, the most interesting part of Corollary \ref{cor:Wdec} is probably the one concerned with path-components of $\Dec(M)$.
Namely, it allows for an answer to the following (fairly vague) question:
When can a universe have both a big bang and a big crunch singularity?

More precisely, let $(\ol M, \ol g)$ be a globally hyperbolic Lorentzian manifold which is subject to the dominant energy condition, \ie $\Ein(V) \coloneqq \mathrm{Ric}(V)- \frac12 \scal V$ is past-causal for all future-causal $V \in T\ol M$.
We choose a foliation $\ol M \cong M \times \RR$ of $\ol M$ into spacelike hypersurfaces (diffeomorphic to $M$).
On every slice $M \times \{t\}$, $\ol g$ induces a metric $g_t$ and a second fundamental form $K_t$.
Hence, we obtain a family $(g_t, K_t)_{t \in \RR}$ of initial data sets on $M$.
%Let $e_0$ denote the future-directed unit normal on $M \times \{t\}$.
%  Using the definitions for $\rho_t$ und $j_t$ introduced in \eqref{eq:CE} we obtain $\Ein(e_0)=\mathrm{Ric}(e_0) - \frac12 \scal e_0 = -\rho_t e_0 + j_t^\sharp$, where we used in the last equation the equations by Gauß and by Codazzi.
%The dominant energy condition then implies that $\Ein(e_0) = -\rho_t e_0 + j_t^\sharp$ is past-causal.
%Thus we have $\rho_t \geq \|j_t\|$ and hence $(g_t, K_t) \in \Dec(M)$ for all $t \in \RR$.
As explained in Subsection~\ref{subsec.dominant.energy} the dominant energy condition for $(\ol M, \ol g)$ implies that on every slice the induced initial data set is subject to the dominant energy condition, \ie $(g_t, K_t) \in \Dec(M)$ for all $t \in \RR$.
A \emph{big bang} (with respect to the given foliation) may be characterized by demanding that in the limit $t \to -\infty$, the mean curvature $H_t = \frac1n \tr(K_t)$ of the initial data set $(g_t, K_t)$ uniformly converges to $\infty$.
Analogously, a \emph{big crunch} can be defined by $H_t \to -\infty$ for $t \to \infty$.
In particular, $H_t > 0$ holds for sufficiently small $t \in \RR$ in the case of a big bang and $H_t < 0$ for sufficiently large $t \in \RR$ in the case of a big crunch.
Note that an initial data set with $H_t > 0$ is in the component of expanding initial data $\tilde{C}_+$ and an initial data set with $H_t < 0$ is in the component of contracting initial data $\tilde{C}_-$.
Hence, if $(\ol M, \ol g)$ has both big bang and big crunch, then, for suitable choice of $t_- < t_+ \in \RR$, $[t_-, t_+] \to \Dec(M),\, t \mapsto (g_t, K_t)$ defines a continuous path starting in $\tilde{C}_-$ and ending in $\tilde{C}_+$, showing that $\tilde{C}_- = \tilde{C}_+$.
With this in mind, the first part of Corollary \ref{cor:Wdec} provides an obstruction to the existence of a space-time having both big bang and big crunch with Cauchy surface diffeomorphic to $M$.

This argument can be further improved.
To do so, we need the conservation theorem, which is proved in the book by Hawking and Ellis \cite{hawking.ellis:73}.

\begin{theorem}[{\cite[Sec.~4.3]{hawking.ellis:73}}]\label{thm.aus.h-w}
Let $(\ol M, \ol g)$ be a time-oriented Lorentzian manifold and assume it admits a temporal function, \ie a smooth function $\tau \colon \ol M \to \RR$ such that $\grad \tau$ is past-timelike.
Furthermore, let $U \subset \ol M$ be a compact submanifold with piecewise smooth boundary with $\dim U=\dim\ol M =n + 1$.
We assume that we have a piecewise smooth decomposition $\del U = \del_- U \cup \del_0 U \cup \del_+ U$ into the timelike boundary $\del_0 U$, the past non-timelike boundary $\del_- U$ and the future non-timelike boundary $\del_+ U$.
If $(\ol M, \ol g)$ satisfies the dominant energy condition and if the Einstein tensor $\Ein$ identically vanishes on $\del_- U$ and $\del_0 U$, then it identically vanishes on all of $U$.
In particular, if $(\ol M, \ol g)$ is globally hyperbolic, subject to the dominant energy condition and $\rho \equiv 0$ on a Cauchy surface, then it is a vacuum space-time.
\end{theorem}

Here, a space-time is called a \emph{vacuum space-time} or we say that it is \emph{vacuous} if $\Ein\equiv 0$. \ifdraft For a proof of this theorem see also Appendix~\ref{app.hawking.ellis}.\fi

We use it to prove our following theorem. 
\begin{theorem}
Let $(\ol M, \ol g)$ be a globally hyperbolic Lorentzian manifold with compact Cauchy surface $M$.
Assume that $(\ol M, \ol g)$ satisfies the dominant energy condition and has a foliation into spacelike hypersurfaces with big bang as well as big crunch.
If $\ol M$ is spin, then $\alpha(M) = 0$ or the foliation of $\ol M$ contains a leaf for which there exists a non-zero spinor $\phi$ such that $\phi$ along with the induced initial data set $(g, K)$ satisfies the constraint equation \eqref{gen.imag.killing}.
\footnote{Assuming $M$ connected, we can reexpress this by saying that $(g,K,\phi)$ is an initial data triple.}%In Fußnote verschoben, weil den Lesefluss störend.
  In particular, if addtionally $(\ol M, \ol g)$ is non-vacuous and simply connected of dimension $n +1\geq 6$, then $M$ admits a metric of positive scalar curvature.
\end{theorem}

\begin{proof}
As explained above, such a foliation gives rise to a map $\RR \to \Dec(M),\, t \mapsto (g_t, K_t)$ and it is possible to choose $t_- < t_+ \in \RR$ such that $H_{t_-} > 0$ and $H_{t_+} > 0$.
It is elementary to check that $(g_{t_-}, K_{t_-} + s g_{t_-})$ strictly satisfies the dominant energy condition for all $s > 0$.
Analogously, $(g_{t_+}, K_{t_+} - s g_{t_+})$ strictly satisfies the dominant energy condition for all $s > 0$.
Hence, by adding these paths at the ends of the restriction $[t_-, t_+] \to \Dec(M)$, we obtain a path $\gamma \colon [t_- -1, t_+ +1] \to \Dec(M)$ with endpoints in $\Decstr(M)$.
More precisely, the whole segment $\gamma([t_- -1, t_-))$ lies in the component $C_-$ and $\gamma((t_+, t_+ +1]) \subset C_+$.

We may assume $\alpha(M) \neq 0$, as otherwise we are done with the first claim.
Hence, by Theorem \ref{thm:CompAdiff} $\oladiff(C_-) = \alpha(M) \neq 0$.
Now recall from the definition, that $\oladiff$ may be computed as family index of the Dirac-Witten operators associated to any path in $\Ini(M)$ connecting $C_-$ with $C_+$.
Taking the path $\gamma$ from above, this implies that there is some $t_0 \in (t_- -1, t_+ +1)$ for which the Dirac-Witten operator associated to $\gamma(t_0)$ is not invertible.
As the end segments are contained in $\Decstr(M)$, this $t_0$ must be contained in $[t_-, t_+]$ and the first claim follows from Theorem \ref{theorem.kernel-gives-initial-data}.

For the remaining part, recall that due to Stolz's theorem, we may assume that~$M$ is spin with $\alpha(M) \neq 0$ as otherwise $M$ carries a metric of positive scalar curvature.
Hence, applying the first part for some foliation, we obtain a leaf $(M, g, K)$ along with a solution $\phi$ of the constraint equation \eqref{gen.imag.killing} for a parallel spinor.
As $M$ is simply connected, $b_1(M) = 0$ it follows from Corollary~\ref{cor:WdecInv-old} that $\phi$ cannot be lightlike.
Then by Corollary~\ref{cor:TlSpinorVac}, $\rho \equiv 0$ on $M$.
The conservation theorem then implies that $(\ol M, \ol g)$ is a vacuum space-time, contradicting the assumptions.
\end{proof}

It was pointed out to us by Greg Galloway that a more general version of this also follows from the following theorem due to Gerhardt \cite[Thm.~6.1]{gerhardt:83}.
\begin{theorem}[{\cite[Thm.~2.2]{gerhardt:00}}]
Let $(\ol M, \ol g)$ be a globally hyperbolic Lorentzian manifold with compact Cauchy surface.
Let $M_1$ and $M_2$ be two Cauchy surfaces forming the boundary of a region $\Omega$, $M_1$ lying in the past of $M_2$, and let $H_1$ and $H_2$ be their mean curvature functions, respectively.
Then for any function $f \in C^{0,\alpha}(\ol\Omega)$ with $H_2 \leq f_{|M_2}$ and $f_{|M_1} \leq H_1$, there exists a $C^{2,\alpha}$-Cauchy surface $M$ between $M_1$ and $M_2$ with mean curvature $H = f_{|M}$.
\end{theorem}
\begin{corollary}
Let $(\ol M, \ol g)$ be a globally hyperbolic Lorentzian manifold with compact Cauchy surface $M$.
Assume that $(\ol M, \ol g)$ satisfies the dominant energy condition and has a foliation with big bang as well as big crunch.
If $(\ol M, \ol g)$ is connected and non-vacuous, then $M$ admits a metric of positive scalar curvature.
\end{corollary}
\begin{proof}
  We apply Gerhardt's theorem for $M_1 = M_{t_-}$, $M_2 = M_{t_+}$ and $f \equiv 0$.
 Thus $M$ is a spacelike minimal hypersurface.
From the dominant energy condition, we get on $M$ that
\begin{align*}
	\scal \geq \scal - \|K\|^2 = \rho \geq 0.
\end{align*}
If $\scal \equiv 0$, then $\rho \equiv 0$ and $(\ol M, \ol g)$ is a vacuum space-time.
Hence, $\scal \geq 0$ with $\scal \not\equiv 0$, and it is well-known that this implies that there is a metric of positive sclar curvature on $M$.
Notice that this is true despite the reduced regularity assumptions.
For instance, the argument of Kazdan and Warner \cite{kazdan.warner:75} also applies for $C^2$-metrics.
\end{proof}

\subsection*{Concluding remark}
We obtain an obstruction from a big bang to a big crunch, similar to the one obtain via Gerhardt's theorem. In contrast to the approach via Gerhardt's theorem that uses minimal hypersurfaces, our proof uses index theory. This is a remarkable analogon to obstructions against positive scalar curvature metrics, \eg on a torus, where index theoretic methods compete with the minimal hypersurface method.

The analogy even goes further: in this article we have shown that the index theoretical methods can also be used to obtain informations about higher homotopy groups.

Our index theoretical method also has the advantage, that we can exclude with the same approach the evolution from the component of contracting initial data to the 
component of expanding initial data. It seems to us that this result cannot be obtained from Gerhardt's theorem.

\appendix
\newcommand{\Tay}[1]{\mathrm{TayDev}_{(#1)}^{k}}
\newcommand{\Taymap}[1]{\mathrm{\Tau}_{(#1)}^{k}}
\newcommand{\wTaymap}[1]{\witi{\mathrm{\Tau}}_{(#1)}^{k}}

\section{The Taylor development map for Ricci-flat metrics}\label{appendix.smoothness}
We develop a construction here, which might also be of interest, independent of our present article. The goal is to use the Taylor development of the metric in  order to associate to any closed Ricci-flat Riemannian manifold $(P,h)$ with finite fundamental group a covering map
$$\Taymap{P,h}:\maP_\O(P,h)\to \Tay{P,h}$$
from the orthnormal frame bundle of $(P,h)$ to some submanifold $\Tay{P,h}$  of some $\RR^r$ which behaves ``natural'' under isometries. 

Our construction relies on the following lemma.
\begin{lemma}
  Let $U$ be an open subset of $\RR^n$, and $f\colon U\to \RR^m$ an analytic function. Assume that for any $x,y\in U$, $x\neq y$ the Taylor series in $x$ and $y$ do not coincide. Then for every $x_0\in U$ there is a number $k$, such that
  $T_kf:=(f,df,d^2f,\ldots,d^k f):U_0\to \RR^N$ has an injective differential in $x_0$.
\end{lemma}

\begin{proof}
  Suppose $X\in \bigcap_{k = 0}^\infty \ker(d_{x_0}T_kf)$. Then the map $t\mapsto f(x_0+tX)-f(x_0)$ is an analytic map that vanishes of infinite order at $0$. It follows that $f$ is constant along $t\mapsto x_0+tX$.
  We can apply the same argument to
  $$t\mapsto \frac{\partial^{|\alpha|}f}{\partial x^\alpha}(x_0+tX)-\frac{\partial^{|\alpha|}f}{\partial x^\alpha}(x_0)$$
  for any multi-index $\alpha$. It thus follows, that all point along $t\mapsto x_0+tX$ have the same Taylor series. Thus $X=0$. 
\end{proof}

In the following, let $\maP_\O(P,h)$ be the principal bundle of $h$-orthonormal frames over the Riemannian manifold $(P,h)$,
and $\maP_{\GL}(P)$ the principal bundle of all frames over $P$. 
The universal Riemannian covering will be denoted by $(\witi P, \witi h) \to (P, h)$. 

\begin{construction}
  We assume that $(P,h)$ is a closed connected $m$-dimensional Ricci-flat Riemannian manifold with finite $\pi_1(P)$.  We associate to $(P,h)$ a natural number $k_0=k_0(P,h)$, and a submanifold $\Tay{P,h}$ of $\RR^r$  which depends on $k\geq k_0$ and where $r= (m^{k+3}-m^2)/(m-1)$. We also associate a covering map (in particular, it is a local diffeomorphism)
  $$\Taymap{P,h}: \maP_\O(P,h) \to \Tay{P,h}.$$ 
The map  $(P,h)\mapsto  (\Tay{P,h},\Taymap{P,h})$ is natural, smooth and almost injective in the following sense:
\begin{enumerate}
\item[Almost injectivity:] Two frames $E_0,E_1\in\maP_\O(P,h) $ with base points $p_0,p_1\in P$  are mapped to the same element in $\Tay{P,h}$, if and only if, there is an isometry from an open neighborhod of $p_0$ to an open neighborhood of $p_1$ whose differential maps $E_0$ to $E_1$. Note that any such local isometry lifts to a global isometry of $(\witi P, \witi h)$.
\item[Naturalilty:] If there is a (surjective) Riemannian covering $I\colon (P,h_P)\to (Q,h_Q)$, then $k_0(P,h_P)=k_0(Q,h_Q)$, and for $k\geq k_0$ we have $\Tay{P,h_P}=  \Tay{Q,h_Q}$
  and
  $\Taymap{P,h_P}=\Taymap{Q,h_Q}\circ \phi_*$, where $\phi_*$ is the map from
  $\maP_\O(P,h_P)\to\maP_\O(Q,h_Q)$ induced from $\phi$. 
\item[Smoothness:] If we have a smooth family $b\mapsto h_b$, $h\in B$ of such metrics on $P$, where $B$ is any parameter manifold, possibly with boundary, and if $k\geq k_0(P,h_b)$ for all $b\in B$ (which can be achieved \eg if $B$ is compact) then  $\Tay{P,h_b}$ and $\Taymap{P,h_b}$ depend smoothly on $b$. More precisely, $\bigcup_{b\in B} \left(\maP_\O(P,h_b)\times\{b\}\right)$ and $\bigcup_{b\in B} \left(\Tay{P,h_b}\times\{b\}\right)$ are  smooth submanifolds
  of $\maP_\GL(P)\times B$ and $\RR^r\times B$, and the maps $\Taymap{P,h_b}$ define a smooth map between these two submanifolds.    
\end{enumerate}
\end{construction}

Note that the Cheeger--Gromoll splitting theorem tells us that the compactness of $P$, its Ricci-flatness and the finiteness of $\pi_1(P)$ imply that $\Isom(\witi P,\witi h)$ and $\Isom(P,h)$ are finite.

\Description{of the construction}
Let $(P,h)$ satisfy the assumptions of the construction.
The metric $h$ is real-analytic in normal coordinates, thus the Taylor development converges.
  
   For each $x\in P$, and each $h$-orthonormal frame $E=(e_1,\ldots,e_m)$ we
   obtain  an isometry $I_E:\RR^m\to T_xP$, and let $\Taymap{P,h}(E)$ be the Taylor polynomial
   of order $k$ at $0$ of $I_E^*\left(\exp^{h}_x\right)^*h$. (We follow the convention that $\Taymap{P,h}$ also includes all coefficients of order $\leq k$.) 
   This Taylor polynomial $\Taymap{P,h}$ is a smooth map
   $$\Taymap{P,h}:\maP_\O(P,h)\lto \Bigl(\bigoplus_{j=0}^k(\RR^m)^{\otimes j}\Bigr)\otimes \RR^{m^2}=\RR^r$$
 for $r=m^2+m^3+\cdots +m^{k+2}=(m^{k+3}-m^2)/(m-1)$.  
   It is analytic with respect to the analytic structure on $\maP_\O(P,h)$ defined via $h$.

   From the previous lemma it follows, that for any $E\in \maP_\O(P,h)$ there is a $k_1(E)$, such that for all $k\geq k_1(E)$ the differential of  $\Taymap{P,h}$ in $E$ is injective. Thus, for such a $k$, the map  $\Taymap{P,h}$ is an immersion on a neighborhood of $E$. By an obvious compactness argument, we see that $k_1=k_1(E)$ can be chosen independently of $E$. 
   For $k\geq k_1$ we define $\Tay{P,h}$ as the image of  $\Taymap{P,h}$; we will show later that this set is a submanifold
   of $\RR^r$ for sufficiently large $k$.
   
   The Taylor polynomial of the metric is natural in the sense that if
   $\phi:(U,h_P)\to (V,h_Q)$ is an isometry defined on an open subset $U$ of a manifold $(P,h_P)$ as above, then $V$ is an open subset of a similar manifold $(Q,h_Q)$
   and
\begin{equation}\label{eq:nat.taymap}
  \Taymap{Q,h_Q}(\phi_*(E))= \Taymap{P,h_P}(E)\,\quad \forall E\in \maP_\O(U,h_P).
\end{equation}
When we apply this relation to $(Q,h_Q)=(P,h_P)=(P,h)$, then ``if''-part in the almost injectivity follows for any $k\in \NN$.
Such a locally defined isometry  lifts to an isometry of $\Isom(\witi P, \witi h)$, which is a discrete and thus finite group.
Note that there is a well-defined (smooth) covering $\pi_\Isom:\maP_\O(P,h)\to \Isom(\witi P,\witi h)\backslash\maP_\O(\witi P,\witi h)$. 
It follows for $k\geq k_1$ that the map $\Taymap{P,h}$ descends to an immersion 
$$\wTaymap{P,h}:\Isom(\witi P,\witi h)\backslash\maP_\O(\witi P,\witi h)\to \RR^r.$$
We will show that this map is injective for sufficiently large $k$. 
To show this assume that frames $E,E'\in \maP_\O(P,h)$ with base points $p,p'\in P$ are given, and we assume that for all $k\in \NN$, the Taylor series satisfy $\Taymap{P,h}(E)=\Taymap{P,h}(E')$. Then
   \begin{equation}\label{constr.isometry}
     \exp^{h}_{p'}\circ I_{E'}\circ (I_{E})^{-1}\circ \left(\exp^{h}_{p}\right)^{-1}
   \end{equation}
   defines a local isometry from a neighborhood of $p$ to a neighborhood of $p'$. Thus $\pi_\Isom(E)=\pi_\Isom(E')$.
   Thus for every $[E]\in \Isom(\witi P,\witi h)\backslash\maP_\O(\witi P,\witi h)$ there is a $k_0([E])\geq k_1$ such that
   $\left(\wTaymap{P,h}\right)^{-1}(\{\wTaymap{P,h}([E])\}=\{[e]\}$ for all $k\geq k_0([E])$. By a compactness argument this implies that
   $\wTaymap{P,h}$ is injective on a neighborhood of $[E]$, and a further compactness argument shows that $k_0=k_0([E])$ can be chosen independently on $[E]$. Then $\wTaymap{P,h}$ is an embedding,  $\Tay{P,h}=\image \wTaymap{P,h}$ a submanifold and the almost injectivity is proven.

  \jonathan{$\Taymap{P,h}$ is a covering map as it is the composition of the covering map $\pi_\Isom$ with a diffeomorphism.}
  Naturality follows from Equation~\eqref{eq:nat.taymap} and the smoothness is obvious from the construction.

\begin{proposition}
We assume that $P$ is a closed manifold with $\pi_1(P)$ finite, and that
$h_t$, $t\in (-\ep,\ep)$ is a smooth family of Ricci-flat metrics.
Assume that $h_t$, $t\in \RR$ is a smooth family of metrics, and $\phi_t\in \Diff(P)$ is continuous in $t$ with $\phi_0=\id_P$, and $\phi_t^*h_t=h_0$. Then $\phi_t$ is smooth in~$t$.
\end{proposition}

Again, the conditions imply that  $\Isom(\witi P,\witi g_t)$ and $\Isom(P,g_t)$ are finite for all $t$.

\begin{proof}
We have a smooth map
$$\Psi:\bigcup_{t\in (-\epsilon,\epsilon)} (\maP_\O(P,h_t) \times \{t\}) \to \Tay{P,h_0} \times (-\epsilon, \epsilon) ,\quad (E,t)\mapsto  (\Taymap{P,h_t}(E), t)$$
This map is a covering of smooth manifolds. Thus if we fix some $E_0\in \maP_\O(P,h_0)$, lifting yields a (unique) smooth path of frames $t\mapsto E_t\in \maP_\O(P,h_t)$ with $\Psi(E_t, t) = (\Taymap{P,h_0}(E_0), t)$ for all $t$. By a construction similar to the construction in \eqref{constr.isometry}, the maps $E_0\mapsto E_t$ come from isometries $\witi\phi_t\colon(\witi P,\witi h_0)\to (\witi P,\witi h_t)$ which smoothly depend on $t$. One can assume $\witi\phi_0=\id_{\witi P}$. By a continuity argument one sees that $\witi\phi_t$ is a lift of $\phi_t$. Thus the smoothness of $\phi_t$ follows.
\end{proof}

\section{Proof of Lemma~\ref{lem.modified.diffeos}}\label{appendix.lemma.isom}
The goal of this appendix is to provide a proof of Lemma~\ref{lem.modified.diffeos}. We proceed with the assumptions and notations the proof of Proposition~\ref{prop:Qtilde.struc}, in the special case that the leaves $Q_\sigma$ are non-compact. Let $p:=n-\ell-1=\dim P_\sigma$.

If we join the tangent spaces to the submanifolds of the form $P_\sigma\times \{v\}$ in $\witi Q_\sigma$ over all $v\in \RR^\ell$ and all $\sigma$ then we obtain a smooth foliation $\FF_P$ of $\witi M$ where all leaves are diffeomorpic to some $P_\sigma$, and as the leaves in the neighborhood of a simply-connected, compact leaf are diffeomorphic to that leaf, all leaves are diffeomorphic. For a given $x=(x_0,0)\in P_0\times \RR^\ell\cong \witi Q_0$, $\witi\Phi_t(x)$ will intersect exactly one leaf of $\FF_P$, denoted as $P(t)$, as we may assume $P_0=P(0)=:P$. For $t\in \image S$ we have  $P(t)\subset \witi Q_0$ and thus $P(t)$ and $P(0)$ are isometric for the induced metrics, let $\iota_t$, $t\in \image S$ be an isometry from $P(0)$ to $P(t)$.

We show that $P(\sigma)$ is isometric to $P(0)$ for any $\sigma\in \RR$.
Choose an orthonormal frame $E:=(e_1,\ldots,e_p)$ of $T_xP$.
Because of the denseness of $\image S$ there is a sequence $t_i\in \image S$ with $t_i\to \sigma$.
After passing to a subsequence  $d\iota_{t_i}(E):=\bigl(d\iota_{t_i}(e_1),\ldots,d\iota_{t_i}(e_p)\bigr)$ is a Cauchy sequence, and for this subsequence $\iota_{t_i}$ converges to an isometry $\iota_\sigma:P(0)\to P(\sigma)$.

Let $\maP_\O(\FF_P)$ be the $\O(p)$-principal bundle of all orthonormal frames of any tangent space of leaves of $\FF_P$ at any point of $\witi M$.
For any frame $\witi E\in \maP_\O(\FF_P)$ over $x\in \witi Q_\sigma$, let $\maT(\witi E)$ be the Taylor polynomial of sufficient order of the induced Riemannian metric on $P(\sigma)$ in normal coordinates given by the basis $\witi E$. As all $P_\sigma$ are isometric for the induced metric, we obtain a submersion
$\maT: \maP_\O(\FF_P)\to  \Tay{P,h}$. 

We extend the map $P(t)$ from above to a map which defines a bijection from  $\RR^\ell\times \RR$ to the space of all leaves of $\FF_P$.
For any $t\in \RR$ we choose an orthonormal basis $(f_1(t)\ldots,f_\ell(t))$ of the $\RR^\ell$-factor of $\witi Q_t$ at $\Phi_t(x)$, depending smoothly on $t$. For  $v=(v_1,\ldots,v_\ell)\in \RR^\ell$ we define
  $$\rho(v,t):=\exp_{\Phi_t(x)}^{\witi Q_t,\witi g_t}\Bigl(\sum_{j=1}^\ell v_jf_j(t)\Bigr).$$
This is an embedding of $\RR^\ell\times \RR\to \witi M$ which maps  $\RR^\ell\times \{t\}$ isometrically to a totally geodesic subspace  of $\witi Q_t$. The image of $\rho$ intersects each leaf of $\FF_P$ exactly once (and this intersection is orthogonal with complemetary dimensions). Let $P(v,t)$ be the leaf of $\FF_P$ running through $\rho(v,t)$.
For each  $(v,t)$ we have a covering $\tau_{v,t}:\maP_\O(P(v,t))\to  \Tay{P,h}$,
and they fit togther to a covering 
\begin{eqnarray*}
   \maP_\O(\FF_P) & \to &   \Tay{P,h} \times \RR^\ell\times \RR \\\
   \witi E\in \maP_\O(P(v,t)) & \mapsto & (\tau_{v,t}(\witi E),v,t).
\end{eqnarray*}
For the frame $E$ fixed above we consider the map $(v,t)\mapsto  (\tau_{v,t}(E),v,t)$, and we lift this map uniquely to 
$H:\RR^\ell\times\RR\to\maP_\O(\FF_P)$ with $H(0)=E$. For each $(v,t)$ there is a unique isometry $h_{v,t}$ from $P$ to $P(v,t)$ mapping $E$ to  $H(v,t)$.
We get a diffeomorphism
\begin{eqnarray*}
  P\times \RR^\ell\times \RR & \to & \witi M\\
  (x,v,t)& \mapsto & \Xi_t(x,v):=h_{v,t}(x).
\end{eqnarray*}
By construction $\Xi_t$ maps $P\times \RR^\ell\times\{t\}$ isometrically to $\witi Q_t$.

%\bibliographystyle{alpha}
%\bibliographystyle{acm}
%\bibliography{lorentzdec}

\begin{thebibliography}{10}

\bibitem{ammann.kroencke.mueller:p19}
{\sc Ammann, B., Kr{\"o}ncke, K., and M{\"u}ller, O.}
\newblock Construction of initial data sets for {L}orentzian manifolds with
  lightlike parallel spinors.
\newblock \arxiv{1903.02064}, 2019.

\bibitem{ammann.kroencke.weiss.witt:19}
{\sc Ammann, B., Kr{\"o}ncke, K., Weiss, H., and Witt, F.}
\newblock Holonomy rigidity for {R}icci-flat metrics.
\newblock {\em Math. Z. 291}, 1-2 (2019), 303--311.

\bibitem{andersson.dahl.galloway.pollack:MR3878142}
{\sc Andersson, L., Dahl, M., Galloway, G., and Pollack, D.}
\newblock On the geometry and topology of initial data sets with horizons.
\newblock {\em Asian J. Math. 22}, 5 (2018), 863--881.

\bibitem{atiyah.singer:69}
{\sc Atiyah, M.~F., and Singer, I.~M.}
\newblock Index theory for skew-adjoint {F}redholm operators.
\newblock {\em Pub. Math. IH\'ES 37\/} (1969), 5--26.

\bibitem{baer.gauduchon.moroianu:05}
{\sc B{\"a}r, C., Gauduchon, P., and Moroianu, A.}
\newblock Generalized cylinders in semi-{R}iemannian and spin geometry.
\newblock {\em Math. Z. 249\/} (2005), 545--580.

\bibitem{BI}
{\sc Bartnik, R., and Isenberg, J.}
\newblock The constraint equations.
\newblock In {\em The {E}instein equations and the large scale behavior of
  gravitational fields}. Birkh\"auser, Basel, 2004, pp.~1--38.

\bibitem{baum:81}
{\sc Baum, H.}
\newblock {\em Spin-Strukturen und Dirac-Operatoren {\"u}ber pseudoriemannschen
  Mannigfaltigkeiten}.
\newblock Teubner, 1981.

\bibitem{baum.leistner.lecture.notes:HH}
{\sc Baum, H., and Leistner, T.}
\newblock {L}orentzian geometry --- holonomy, spinors, and {C}auchy problems.
\newblock In {\em Geometric Flows and the Geometry of Space-time}, V.~Cortes,
  K.~Kr{ö}ncke, and J.~Louis, Eds., Tutorials, Schools, and Workshops in the
  Mathematical Sciences. Birkhäuser, 2018.

\bibitem{baum.leistner.lischewski:16}
{\sc Baum, H., Leistner, T., and Lischewski, A.}
\newblock {C}auchy problems for {L}orentzian manifolds with special holonomy.
\newblock {\em Differential Geom. Appl. 45\/} (2016), 43--66.

\bibitem{botvinnik.ebert.randal-williams:14}
{\sc Botvinnik, B., Ebert, J., and Randal-Williams, O.}
\newblock Infinite loop spaces and positive scalar curvature.
\newblock {\em Invent. Math. 209}, 3 (2017), 749--835.

\bibitem{cai.galloway:MR1846368}
{\sc Cai, M., and Galloway, G.}
\newblock On the topology and area of higher-dimensional black holes.
\newblock {\em Class. Quant. Grav. 18}, 14 (2001), 2707--2718.

\bibitem{cheeger.gromoll:71}
{\sc Cheeger, J., and Gromoll, D.}
\newblock The splitting theorem for manifolds of nonnegative {R}icci curvature.
\newblock {\em J. Diff.\ Geom.\ 6\/} (1971/72), 119--128.

\bibitem{crowley.steimle.schick:18}
{\sc Crowley, D., Schick, T., and Steimle, W.}
\newblock Harmonic spinors and metrics of positive curvature via the {G}romoll
  filtration and {T}oda brackets.
\newblock {\em J.\ Topol. 11}, 4 (2018), 1077--1099.

\bibitem{dai.wang.wei:05}
{\sc Dai, X., Wang, X., and Wei, G.}
\newblock On the stability of {R}iemannian manifold with parallel spinors.
\newblock {\em Invent. Math. 161}, 1 (2005), 151--176.

\bibitem{ebert:13}
{\sc Ebert, J.}
\newblock The two definitions of the index difference.
\newblock {\em Trans.\ Amer.\ Math.\ Soc. 369}, 10 (2017), 7469--7507.

\bibitem{eichmair.galloway.mendes:2009.09527}
{\sc Eichmair, M., Galloway, G., and Mendes, A.}
\newblock Initial data rigidity results.
\newblock \arxiv{2009.09527}, to appear in Comm. Math. Phys., 2020 (arXiv).

\bibitem{fischer.wolf:75}
{\sc Fischer, A., and Wolf, J.}
\newblock The structure of compact {R}icci-flat {R}iemannian manifolds.
\newblock {\em J. Diff.\ Geom. 10\/} (1975), 277--288.

\bibitem{galloway:MR2411473}
{\sc Galloway, G.}
\newblock Rigidity of marginally trapped surfaces and the topology of black
  holes.
\newblock {\em Comm. Anal. Geom. 16}, 1 (2008), 217--229.

\bibitem{galloway:MR3768955}
{\sc Galloway, G.}
\newblock Rigidity of outermost {MOTS}: the initial data version.
\newblock {\em Gen. Relativity Gravitation 50}, 3 (2018), Paper No. 3.

\bibitem{galloway.schoen:MR2238889}
{\sc Galloway, G., and Schoen, R.}
\newblock A generalization of {H}awking's black hole topology theorem to higher
  dimensions.
\newblock {\em Comm. Math. Phys. 266}, 2 (2006), 571--576.

\bibitem{gerhardt:83}
{\sc Gerhardt, C.}
\newblock {$H$}-surfaces in {L}orentzian manifolds.
\newblock {\em Comm. Math. Phys. 89}, 4 (1983), 523--553.

\bibitem{gerhardt:00}
{\sc Gerhardt, C.}
\newblock Hypersurfaces of prescribed mean curvature in {L}orentzian manifolds.
\newblock {\em Math. Z. 235}, 1 (2000), 83--97.

\bibitem{gloeckle:p2019}
{\sc {Gl{\"o}ckle}, J.}
\newblock {Homotopy of the space of initial values satisfying the dominant
  energy condition strictly}.
\newblock \arxiv{1906.00099}, 2019.

\bibitem{gromov.lawson:80}
{\sc Gromov, M., and Lawson, M.}
\newblock The classification of simply connected manifolds of positive scalar
  curvature.
\newblock {\em Ann. Math. 111}, 3 (1980), 423--434.

\bibitem{hawking.ellis:73}
{\sc Hawking, S.~W., and Ellis, G. F.~R.}
\newblock {\em The large scale structure of space-time}.
\newblock Cambridge University Press, London-New York, 1973.
\newblock Cambridge Monographs on Mathematical Physics, No. 1.

\bibitem{hijazi.zhang:03}
{\sc Hijazi, O., and Zhang, X.}
\newblock The {D}irac-{W}itten operator on spacelike hypersurfaces.
\newblock {\em Comm. Anal. Geom. 11}, 4 (2003), 737--750.

\bibitem{hirsch:76}
{\sc Hirsch, M.~W.}
\newblock {\em Differential Topology}.
\newblock No.~33 in Graduate Texts in Mathematics. Springer-Verlag, 1976.

\bibitem{hitchin:74}
{\sc Hitchin, N.}
\newblock Harmonic spinors.
\newblock {\em Adv. Math. 14\/} (1974), 1--55.

\bibitem{joyce:book}
{\sc Joyce, D.}
\newblock {\em Riemannian holonomy groups and calibrated geometry}, vol.~12 of
  {\em Oxford Graduate Texts in Mathematics}.
\newblock Oxford University Press, Oxford, 2007.

\bibitem{kazdan.warner:75}
{\sc Kazdan, J.~L., and Warner, F.~W.}
\newblock {Scalar curvature and conformal deformation of Riemannian structure}.
\newblock {\em J. Diff. Geom. 10}, 1 (1975), 113 -- 134.

\bibitem{kroencke:15}
{\sc Kr\"{o}ncke, K.}
\newblock On infinitesimal {E}instein deformations.
\newblock {\em Diff.\ Geom.\ Appl. 38\/} (2015), 41--57.

\bibitem{lawson:77}
{\sc Lawson, H.~B.}
\newblock {\em The quantitative theory of foliations}.
\newblock Amer, Math. Soc., 1977.
\newblock CBMS Regional Conference 1975, No. 27.

\bibitem{lawson.michelsohn:89}
{\sc Lawson, H.~B., and Michelsohn, M.-L.}
\newblock {\em Spin Geometry}.
\newblock Princeton University Press, Princeton, 1989.

\bibitem{leistner.jdg2007}
{\sc Leistner, T.}
\newblock On the classification of {L}orentzian holonomy groups.
\newblock {\em J. Diff. Geom. 76}, 3 (2007), 423--484.

\bibitem{leistner.lischewski:15}
{\sc Leistner, T., and Lischewski, A.}
\newblock The ambient obstruction tensor and conformal holonomy.
\newblock {\em Pacific J. Math. 290\/} (2017), 403--436.

\bibitem{leistner.lischewski:19}
{\sc Leistner, T., and Lischewski, A.}
\newblock Hyperbolic {E}volution {E}quations, {L}orentzian {H}olonomy, and
  {R}iemannian {G}eneralised {K}illing {S}pinors.
\newblock {\em J. Geom. Anal. 29\/} (2019), 33--82.

\bibitem{lischewski:15-preprint}
{\sc Lischewski, A.}
\newblock The {C}auchy problem for parallel spinors as first-order symmetric
  hyperbolic system.
\newblock \arxiv{1503.04946}, 2015.

\bibitem{moore.schochet:06}
{\sc Moore, C.~C., and Schochet, C.~L.}
\newblock {\em Global analysis on foliated spaces}, second~ed., vol.~9 of {\em
  Mathematical Sciences Research Institute Publications}.
\newblock Cambridge University Press, New York, 2006.

\bibitem{oneill:83}
{\sc O'Neill, B.}
\newblock {\em Semi-{R}iemannian geometry}, vol.~103 of {\em Pure and Applied
  Mathematics}.
\newblock Academic Press, Inc. [Harcourt Brace Jovanovich, Publishers], New
  York, 1983.
\newblock With applications to relativity.

\bibitem{parker.taubes:82}
{\sc Parker, T., and Taubes, C.~H.}
\newblock {On Witten's proof of the positive energy theorem}.
\newblock {\em Comm. Math. Phys. 84\/} (1982), 223--238.

\bibitem{pfaeffle:00}
{\sc Pf{\"a}ffle, F.}
\newblock The {D}irac spectrum of {B}ieberbach manifolds.
\newblock {\em J. Geom. Phys. 35\/} (2000), 367--385.

\bibitem{sakai:96}
{\sc Sakai, T.}
\newblock {\em Riemannian geometry}, vol.~149 of {\em Translations of
  Mathematical Monographs}.
\newblock Amer.\ Math.\ Soc., 1996.

\bibitem{schick.wraith.2016}
{\sc Schick, T., and Wraith, D.}
\newblock Non-negative versus positive scalar curvature.
\newblock \arxiv{1607.00657}, 2016--2020.

\bibitem{seipel:19}
{\sc Seipel, J.}
\newblock {\em Cauchy problems on Lorentzian manifolds with parallel vector and
  spinor fields}.
\newblock Master thesis, Universit\"at Regensburg, 2019.

\bibitem{stolz:92}
{\sc Stolz, S.}
\newblock Simply connected manifolds of positive scalar curvature.
\newblock {\em Ann. Math. 136}, 3 (1992), 511--540.

\bibitem{wang_m:91}
{\sc Wang, M.~Y.}
\newblock Preserving parallel spinors under metric deformations.
\newblock {\em Indiana Univ. Math. J. 40\/} (1991), 815--844.

\bibitem{witten:81}
{\sc Witten, E.}
\newblock {A new proof of the positive energy theorem}.
\newblock {\em Comm. Math. Phys. 80\/} (1981), 381--402.

\bibitem{yau:71}
{\sc Yau, S.-T.}
\newblock On the fundamental group of compact manifolds of non-positive
  curvature.
\newblock {\em Ann. of Math. (2) 93\/} (1971), 579--585.

\end{thebibliography}

\end{document}